\documentclass[12 pt]{amsart}
\usepackage{amssymb,latexsym,amsmath,amscd,amsthm,graphicx, color}
\usepackage[all]{xy}
\usepackage{pgf,tikz}
\usepackage{mathrsfs}
\usepackage{cite}
\usepackage{soul}
\usetikzlibrary{arrows}
\definecolor{uuuuuu}{rgb}{0.26666666666666666,0.26666666666666666,0.26666666666666666}
\definecolor{xdxdff}{rgb}{0.49019607843137253,0.49019607843137253,1.}
\definecolor{ffqqqq}{rgb}{1.,0.,0.}

\definecolor{uuuuuu}{rgb}{0.26666666666666666,0.26666666666666666,0.26666666666666666}
\definecolor{qqwuqq}{rgb}{0.,0.39215686274509803,0.}
\definecolor{zzttqq}{rgb}{0.6,0.2,0.}
\definecolor{xdxdff}{rgb}{0.49019607843137253,0.49019607843137253,1.}
\definecolor{qqqqff}{rgb}{0.,0.,1.}
\definecolor{cqcqcq}{rgb}{0.7529411764705882,0.7529411764705882,0.7529411764705882}

\theoremstyle{plain}

\newtheorem{theo}[subsubsection]{Theorem}

\newtheorem{open}[subsection]{Open}
\newtheorem{lemma}[subsection]{Lemma}
\newtheorem{lemma1}[subsubsection]{Lemma}

\newtheorem{prop}[subsection]{Proposition}
\newtheorem{propo}[subsubsection]{Proposition}
\theoremstyle{definition}
\newtheorem{deli}[subsection]{Delineation}
\newtheorem{defi}[subsubsection]{Definition}
\newtheorem{cor}[subsection]{Corollary}

\newtheorem{remark}[subsection]{Remark}
\newtheorem{remark1}[subsubsection]{Remark}

\newtheorem{note}[subsection]{Note}
\newtheorem{note1}[subsubsection]{Note}

\newcommand{\uu}{\cup}
\newcommand{\ii}{\cap}
\newcommand{\UU}{\bigcup}

\newcommand{\ci}{\subseteq}
\newcommand{\sci}{\subset}
\newcommand{\es}{\emptyset}
\newcommand{\set}[1]{\{#1\}}

\newcommand{\ga}{\alpha}
\newcommand{\gb}{\beta}

\renewcommand{\gg}{\gamma}

\newcommand{\gk}{\kappa}

\newcommand{\gm}{\mu}
\newcommand{\gn}{\nu}
\newcommand{\go}{\omega}

\newcommand{\gt}{\tau}

\newcommand{\gG}{\Gamma}

\newcommand{\tit}{\textit}

\newcommand{\C}[1]{\mathcal{#1}}
\newcommand{\D}[1]{\mathbb{#1}}

\newcommand{\te}{\text}

\newcommand{\ol}{\overline}

\begin{document}
To appear, in the Journal `Mathematics'
\title{Quantization for a condensation system}
\author{ Shivam Dubey}
\address{Department of Applied Sciences\\
	Indian Institute of Information Technology Allahabad\\
Prayagraj, 211015, UP, India.}

\email{rss2022509@iiita.ac.in}
\author{ Mrinal Kanti Roychowdhury}
\address{School of Mathematical and Statistical Sciences\\
	University of Texas Rio Grande Valley\\
	1201 West University Drive\\
	Edinburg, TX 78539-2999, USA.}
\email{mrinal.roychowdhury@utrgv.edu}
\author{Saurabh Verma}
\address{Department of Applied Sciences\\
	Indian Institute of Information Technology Allahabad\\
	 prayagraj, 211015, UP, India.}
\email{Saurabhverma@iiita.ac.in}

%

	\date{}
\maketitle

\pagestyle{myheadings}\markboth{S. Dubey, M.K. Roychowdhury, and S. Verma }{Quantization for a condensation system}

\begin{abstract}
	For a given $r \in (0, +\infty)$, the quantization dimension of order $r$, if it exists, denoted by $D_r(\mu)$, represents the rate at which the $n$th quantization error of order $r$ approaches to zero as the number of elements $n$ in an optimal set of $n$-means for $\mu$ tends to infinity. If $D_r(\mu)$ does not exist, we define $\underline{D}_r(\mu)$ and $\overline{D}_r(\mu)$ as the lower and the upper quantization dimensions of $\mu$ of order $r$, respectively. In this paper, we investigate the quantization dimension of the condensation measure $\mu$ associated with a condensation system $(\{S_j\}_{j=1}^N, (p_j)_{j=0}^N, \nu).$ We provide two examples: one where $\nu$ is an infinite discrete distribution on $\mathbb{R}$, and one where $\nu$ is a uniform distribution on $\mathbb{R}$. For both the discrete and uniform distributions $\nu$, we determine the optimal sets of $n$-means, and calculate the quantization dimensions of condensation measures $\mu$, and show that the $D_r(\mu)$-dimensional quantization coefficients do not exist. Moreover, we demonstrate that the lower and upper quantization coefficients are finite and positive.
\end{abstract}
\section{Introduction}
Various types of dimensions, including Hausdorff and packing dimensions or the lower and the upper box-counting dimensions are important to characterize the complexity of highly irregular sets. In recent years, paralleling methods have been adopted to study the corresponding dimensions of measures (see \cite{F}). In this paper, we investigate the quantization dimensions of the \tit{condensation measures}. The quantization problem consists in studying the quantization error induced by the approximation of a given probability measure with discrete probability measures of finite supports. This problem originated in information theory and some engineering technology. For rigorous mathematical foundation of this theory one can see Graf-Luschgy's book \cite{GL2}. Further theoretical results and promising applications are contained in \cite{AW, GG, GKL, GL1, GN, Z1}. Two important objects in the quantization theory are the quantization coefficient and the quantization dimension. 

Let $\gm$ be a Borel probability measure on a $d$-dimensional normed space $\D R^d$ equipped with a metric $\rho$ induced by the norm $\|\cdot\|$. 
Let $r \in (0, +\infty)$ and  $n \in \D N$, where $\D N$ is the set of natural numbers, and $\ga\ci \D R^d$ be a locally finite (i.e., intersection of $\ga$ with any bounded subset of $\D R^d$ is finite) subset of $\D R^d$. This implies that $\ga$ is countable and closed. Then, the \tit{$n$th quantization error} of order $r$ for $\gm$ is defined by
\[V_{n, r}(\gm):=\te{inf}\set{\int \rho(x, \ga)^r d\gm(x): \ga \ci \D R^d, \, \te{card}(\ga) \leq n},\]
where $\rho(x, \ga)$ denotes the distance from the element $x$ to the set $\ga$ with respect to a given norm $\|\cdot\|$ on $\D R^d$.
 If $\int \| x\|^r d\gm(x)<\infty$, then there is some set $\ga$ for which the infimum is achieved (see \cite{AW, GKL, GL1, GL2}). Such a set $\ga$ is called an \tit{optimal set of $n$-means}, or \tit{optimal set of $n$-quantizers}. For some recent work in this direction one can see \cite{F1, DL, DR, GL1, GL2, GL3, LM, P, RR, R1,  R3, R5, R6, Z1, Z2, Z3}.
An optimal set $\ga$ of $n$-means can then be used to give a best approximation of $\gm$ by a discrete probability supported on a set with no more than $n$ elements. This can be done by giving each element $a \in \ga $ a mass corresponding to $\gm(M(a|\ga))$, where $M(a|\ga)$ is the set of elements $x \in \D R^d$ which are nearest to $a$, i.e., $\rho(x, \ga)=\rho(x,a)$. So, the set $\set{M(a|\ga) : a \in \ga}$ is the \tit{Voronoi diagram} or \tit{Voronoi tessellation} of $\D R^d$ with respect to $\ga$. Of course, the idea of `best approximation' is, in general,
dependent on the choice of $r$. Such a set $\ga$ for which the infimum occurs and contains no more than $n$ elements is called an \tit{optimal set of $n$-means}, or \tit{optimal set of $n$-quantizers} (of order $r$). An optimal set of $n$-means for a probability distribution $\gm$ is also referred to as $V_{n, r}(\gm)$-optimal set. The collection of all optimal sets of $n$-means for a probability distribution $\gm$ is denoted by $\C C_{n, r}(\gm)$. It can be shown that for a continuous Borel probability measure $\gm$, an optimal set of $n$-means always has exactly $n$ elements (see \cite{GL2}).
The numbers
\[\underline D_r(\gm):=\liminf_{n\to \infty} \frac{r\log n}{-\log V_{n, r}(\gm)}, \te{ and }  \ol D_r(\gm):=\limsup_{n\to \infty}  \frac{r\log n}{-\log V_{n,r}(\gm)}\]
are, respectively, called the \tit{lower} and the \tit{upper quantization dimensions} of $\gm$ of order $r$. If $\ol D_r (\gm)=\underline D_r (\gm)$, the common value is called the \tit{quantization dimension} of $\gm$ of order $r$, and is denoted by $D_r(\gm)$. Quantization dimension measures the speed at which the specified measure of the error approaches to zero as $n$ tends to infinity.
For any $\gk_r>0$, the two numbers
\[\underline Q_r^{\gk_r}(\gm):=\liminf_n n^{r/\gk_r} V_{n, r}(\gm) \te{ and } \ol Q_r^{\gk_r}(\gm):=\limsup_n  n^{r/\gk_r} V_{n, r}(\gm)\] are, respectively, called the \tit{$\gk_r$-dimensional lower} and the \tit{upper quantization coefficients} for $\gm$ (of order $r$). The quantization coefficients provide us with more accurate information about the asymptotics of the quantization error than the quantization dimension. Compared to the calculation of quantization dimension, it is usually much more difficult to determine whether the lower and the upper quantization coefficients are finite and positive.
Let $\gn$ be a Borel probability measure on $\D R^d$ with compact support $C$, and $(p_0, p_1, p_2, \cdots, p_N)$ be a probability vector with $p_j>0$ for all $1\leq j\leq N$, where $N$ is a positive integer and $N\geq 2$. Let $S_1, S_2, \cdots, S_N$ be a family of contractive mappings on $\D R^d$. Then, there exist a unique Borel probability measure $\gm$ on $\D R^d$ and a unique nonempty compact set $K$ satisfying \[\gm=\sum_{j=1}^N p_j \gm\circ S_j^{-1} +p_0\gn \te{ and } K=\mathop{\uu}\limits_{j=1}^N S_j(K)\uu C.\] Following \cite{B, L1}, we call $(\set{S_j}_{j=1}^N, (p_j)_{j=0}^N, \gn)$ a condensation system. The measure $\gm$ is called the \tit{attracting measure} or the \tit{condensation measure} for $(\set{S_j}_{j=1}^N, (p_j)_{j=0}^N, \gn)$, and the set $K$, which is the support of the measure $\gm$, is called the \tit{attractor} or \tit{invariant set} for the system. If each $S_j$ is contractive similitude,  then such a measure $\gm$ is also termed as \tit{inhomogeneous self-similar measure} (see \cite{OS1}). There are several works done on condensation system, for example, $L^q$ spectra and R\'enyi dimensions of inhomogeneous self-similar measures were studied by Olsen-Snigireva, and then by Liszka (see \cite{L2, OS2}).
 Let $\gm$ be a condensation measure on $\D R^d$ associated with a condensation system $(\set{S_j}_{j=1}^N, (p_j)_{j=0}^N, \gn)$, where $\gn$ is any Borel probability measure on $\D R^d$ with a compact support. Let $K$ be the attractor of the condensation system, and $C$ be the support of $\gn$. We say that $\set{S_1(K), S_2(K), \cdots, S_N(K), C}$ satisfies the \tit{strong separation condition} (SSC) if $S_1(K), S_2(K), \cdots, S_N(K), C$ are pairwise disjoint. On the other hand, $\set{S_1(K), S_2(K), \cdots, S_N(K), C}$ satisfies the \tit{inhomogeneous open set condition} (IOSC), which is the modified version of the inhomogeneous open set condition proposed in \cite{OS1}, if there exists a bounded nonempty open set $U \sci \D R^d$ such that the following conditions are satisfied:
$(i)$ $\mathop{\uu}\limits_{j=1}^N S_j(U) \sci U$ and $S_i(U)\ii S_j(U)=\es$ for $1\leq i\neq j\leq N$;
$(ii)$ $K\ii U\neq \es$ \te{ and } $C\sci U$;
$(iii)$ $P(\partial(U))=0$ and $C\ii S_j(\te{cl}(U))=\es$ for all $1\leq j\leq N$.
Notice that if  $\set{S_1(K), S_2(K), \cdots, S_N(K), C}$ satisfies the SSC, then it also satisfies the IOSC. Take $\set{S_j}_{j=1}^N$ as a family of bi-Lipschitz mappings on $\D R^d$, i.e., there exist $0<a_j\le b_j<1,$ such that $a_j\rho(x,y) \le \rho(S_j(x),S_j(y)) \le b_j \rho(x,y)$ for all $1 \le j \le N,$ and for all $x,y \in \D R^d,$ we call $a_j$ and $b_j$ the lower and the upper bi-Lipschitz constants, respectively. For $r\in (0, +\infty)$, let $ l_r, \gk_r \in(0, +\infty)$ be the unique numbers such that $\sum_{j=1}^{N} (p_j a_{j}^r)^{\frac{l_r}{r+l_r}}=1$ and $\sum_{j=1}^{N} (p_j b_{j}^r)^{\frac{\gk_r}{r+\gk_r}}=1$, respectively. In \cite{PRV}, Priyadarshi et al. showed that, under the strong separation condition,  $\max \set{l_r, \underline D_r(\gn)} \le \underline D_r(\gm)$. Further, let $ \set{\D R^d; g_1,g_2,\cdots,g_M}$ be an IFS with invariant set $C,$ satisfying the strong open set condition such that $\rho(g_i(x),g_i(y)) \le c_i \rho(x,y)$ for all $x,y \in \D R^d,$ where $0 < c_i<1$ and $1\le i \le M.$ For a given probability vector $(q_1,q_2,\cdots,q_M),$ let $\gn$ be the invariant measure associated with the IFS  $ \set{\D R^d; g_1,g_2,\cdots,g_M}.$ Then $\ol D_r(\gm)\le \max \set{\gk_r, d_r},$ where $d_r$ is uniquely determined by the relation $\sum_{i=1}^{M} (q_i c_{i}^r)^{\frac{d_r}{r+d_r}}=1.$ In our recent work \cite{DRV}, by considering the Borel probability measure $\gn$ as the image measure of an ergodic measure with bounded distortion on the symbolic space $\set{1,2,\cdots,M}^ \D N$ with support a conformal set $C,$ which is a generalization of \cite{PRV}, we have shown that
\[\ol D_r(\gm)\le \max \set{\gk_r, D_r(\gn)}.\] The following problem remains open (also see Remark~\ref{rem24}). 
\begin{open}
Let $\gn$ be any Borel probability measure and $\gm$ be the condensation measure associated with the condensation system $(\{S_i\}_{i=1}^N, (p_i)_{i=0}^N, \gn).$ Then, whether under the strong separation condition, 
\[
\max \set{l_r, \underline D_r(\gn)}\leq \underline D_r(\gm)\leq  \ol D_r(\gm)\leq \max \set{\gk_r, \ol D_r(\gn)}, 
\]
is not known yet. 
\end{open} 
\begin{deli}
	In this paper, we estimate the quantization dimension of a condensation system given by  $(\{S_j\}_{j=1}^N, (p_j)_{j=0}^N, \gn).$
	In this regard, in Section~\ref{sec1} and Section~\ref{sec2}, we give two examples of condensation measures $\gm$ associated with a condensation system $(\{S_j\}_{j=1}^N, (p_j)_{j=0}^N, \gn),$ consisting of contractive similarity mappings $S_j,$ with similarity ratios $s_j$ for $1 \leq j \leq N$ such that one takes $\gn$ as an infinite discrete distribution on $\D R$, and one takes $\gn$ as a uniform distribution on $\D R$, which is the `main result' of the paper. Notice that in the two examples, we fix $r=2$, and the underlying norm is the Euclidean distance. When $r=2$, and $\gk_r$ is determined by the relation $\sum_{j=1}^N (p_j s_j)^{\frac{\gk_r}{r+ \gk_r}},$ we denote $\gk_r$, $D_r(\gm)$, and $D_r(\gn)$, respectively, by $\gk$, $D(\gm)$, and $D(\gn)$. In Section~\ref{sec1}, with respect to the squared Euclidean distance, we first determine the optimal sets of $n$-means, the $n$th quantization error, and the quantization dimension of an infinite discrete distribution $\gn$ with a compact support. For the discrete distribution $\gn$, we see that $D(\gn)=0$. Then, in Section~\ref{sec1}, we also show that  $D(\gm)=\gk=\max\set{\gk, D(\gn)}$, and the $D(\gm)$-dimensional quantization coefficient does not exist, and the lower and the upper quantization coefficients are finite and positive. In Section~\ref{sec2}, we take $\gn$ as a uniform distribution with a compact support, determine the optimal sets of $n$-means, the $n$th quantization error, and show that $D(\gm)=1=D(\gn)=\max\set{\gk, D(\gn)}$, and the $D(\gm)$-dimensional quantization coefficient does not exist, and the lower and the upper quantization coefficients are finite and positive.
\end{deli}

 \section{Preliminaries}\label{sec0}
First we introduce some basic definitions, and then state and prove some lemmas and propositions. By a word $\go$ of length $k$ over the alphabet $I:=\set{1, 2, \cdots, N}$, we mean $\go:=\go_1\go_2\cdots \go_k \in I^k$. A word of length zero is called the empty word and is denoted by $\es$. Length of a word $\go$ is denoted by $|\go|$. By $I^\ast$, it is meant the set of all words over the alphabet $I$ including the empty word $\es$. For any two words $\go:=\go_1\go_2\cdots \go_{|\go|}$ and $\gt:=\gt_1\gt_2\cdots\gt_{|\gt|}$ in $I^\ast$, by $\go\gt$ it is meant the concatenation of the words $\go$ and $\gt$, i.e., $\go\gt:=\go_1\go_2\cdot\go_{|\go|}\gt_1\gt_2\cdots\gt_{|\gt|}$. Let $S_j$ be the bi-Lipschitz mappings with the lower and upper bi-lipschitz constants $a_j$ and $b_j$, respectively, for $1\leq j\leq N$, $(p_0, p_1, \cdots, p_N)$ be the probability vector, and $\gn$ be the Borel probability measure on $\D R^d$ with a compact support $C$ in the condensation system. Let $K$ be the attractor of the condensation system as defined before.
For $\go=\go_1\go_2 \cdots\go_k \in I^k$, we define
 \[S_\go:=S_{\go_1} \circ \cdots \circ S_{\go_k}, \quad p_\go:=\prod_{j=1}^kp_{\go_j},\te{ and }  K_\go:=S_\go(K).\] For the empty word $\es$ in $I^\ast$, by $S_\es$ it is meant the identity mapping on $\D R$. For every $n\geq 1$, by iterating, one easily gets,
\begin{align} \label{eq010} K= \Big(\UU_{|\go|=n}S_\go(K)\Big)\uu \Big(\UU_{k=0}^{n-1} \UU_{|\go|=k} S_\go (C)\Big), \te{ and }
\gm= \sum_{|\go|=n} p_\go \gm\circ S_\go^{-1}+p_0\sum_{k=0}^{n-1} \sum_{|\go|=k} p_\go \gn \circ S_\go^{-1}.
\end{align}
We call $\gG \sci I^\ast$ a \tit{finite maximal antichain} if $\gG$ is a finite set of words in $I^\ast$, such that every sequence in $\set{1, 2, \cdots, N}^{\D N}$ is an extension of some word in $\gG$, but no word of $\gG$ is an extension of another word in $\gG$.
\begin{prop} \label{prop010} Let $r\in(0, +\infty)$. Further, let $l_r\in(0, +\infty)$ be defined by $\sum_{j=1}^{N} (p_ja_j^r)^{\frac{l_r}{r+l_r}}=1$.
	Then,
	\[
	\max \set{l_r, \underline D_r(\gn)} \leq \underline D_r(\gm) \te{ and } \max \set{l_r, \ol D_r(\gn)}\leq \ol D_r(\gm).
	\]
\end{prop}
\begin{lemma}(see, \cite[Lemma~3.1]{PRV})   \label{lemma0031} Let $r\in(0, +\infty)$, and let $\gm$ be the condensation measure associated with the condensation system $(\set{S_j}_{j=1}^N, (p_j)_{j=0}^N, \gn)$. Then,
\[\underline D_r(\gn)\leq \underline D_r(\gm) \te{ and } \ol D_r(\gn)\leq \ol D_r(\gm).\]
\end{lemma}

\begin{lemma}(see, \cite[Proposition 3.2]{PRV})  \label{lemma32} Assume that the sets $S_1 K,\dots, S_N K, C$ satisfy the strong separation condition. Let $r\in(0, +\infty)$. Let $l_r\in(0, +\infty)$ be defined by $\sum_{j=1}^{N} (p_ja_j^r)^{\frac{l_r}{r+l_r}}=1$. Let $\gm$ be the condensation measure associated with the condensation system $(\set{S_j}_{j=1}^N, (p_j)_{j=0}^N, \gn)$, and $e_{n, r}(\gm):=V_{n, r}^{1/r}(\gm)$. Then, it follows that
\[
\liminf_{n\to \infty} n e_{n, r}^{l_r}(\gm) >0.
\]
\end{lemma}
\begin{cor}\label{cor010}
Lemma~\ref{lemma32} implies that $l_r\leq \underline D_r(\gm)$.
\end{cor}
\subsection*{Proof of Proposition~\ref{prop010}}  
Combining Lemma~\ref{lemma0031} and Corollary~\ref{cor010}, we obtain the proof of the proposition. \qed

In the condensation system $(\set{S_j}_{j=1}^N, (p_j)_{j=0}^N, \gn),~\te{if}~ \gn$ is an image measure of an ergodic measure with bounded distortion on the symbolic space $\set{1,2,\cdots,M}^\D N$ with the support a conformal set $C,$ then under strong separation condition $\underline{D}_r(\gm) \leq \overline{D}_r(\gm) \leq \max \{ \gk_r, D_r(\gn) \}.$ For more details see, \cite{DRV}.
\begin{prop}(see, \cite[Theorem 1.4]{DRV})\label{pro99}
Assume that the sets $S_1 K,\dots, S_N K, C$ satisfy the strong separation condition. Let $r\in(0, +\infty)$. Let $\gk_r\in(0, +\infty)$ be defined by $\sum_{j=1}^{N} (p_jb_j^r)^{\frac{\gk_r}{r+\gk_r}}=1$.
Let $C$ be the conformal set generated by a conformal IFS $(\mathbb{R}^d; \{g_i\}_{i=1}^M)$ satisfying the strong separation condition, and $\gn$ is the image measure of an ergodic measure with bounded distortion on the code space $\set{1, 2, \cdots, M}^{\D N}$ with support $C,$ and $\gm$ be the condensation measure associated with the condensation system $(\set{S_j}_{j=1}^N, (p_j)_{j=0}^N, \gn)$, and $e_{n, r}(\gm):=V_{n, r}^{1/r}(\gm)$. Then, it follows that
$\underline{D}_r(\gm) \leq \overline{D}_r(\gm) \leq \max \{ \gk_r,D_r(\gn) \}.$
\end{prop} \qed

Proposition \ref{pro99} leads us to the following corollory.
\begin{cor}
Let $s_j:=a_j=b_j$ in the above theorem, i.e., if $S_j$ are the similarity mappings for $1\leq j\leq N$ in the condensation system $(\set{S_j}_{j=1}^N, (p_j)_{j=0}^N, \gn)$ and $\gk_r$ is uniquely determined by the relation $\sum_{j=1}^N(p_j s^r_j)^{\frac{\gk_r}{r+\gk_r}}=1$,  then $D_r(\gm)= \max\{\gk_r, D_r(\gn)\}.$\qed
\end{cor}
In the following remark, we examine whether the inequalities  
\[
\underline{D}_r(\gm) \leq \max \{ \gk_r, \underline{D}_r(\gn) \} \quad \text{and} \quad \overline{D}_r(\gm) \leq \max \{ \gk_r, \overline{D}_r(\gn) \}
\]  
hold in general for a condensation system $(\set{S_j}_{j=1}^N, (p_j)_{j=0}^N, \gn)$, when $\gn$ is any Borel probability measure not necessarily obtained by some IFS or an image measure of an ergodic measure. However, we cannot make any conclusive statement regarding their validity. In Section \ref{sec1} and Section \ref{sec2}, we give two examples in which for $r=2$, both $D_r(\gn)$ and $D_r(\gm)$ exist and satisfy the following relations: \[
\underline D_r(\gm)=\max \set{\gk_r, \underline D_r(\gn)} \te{ and }  \ol D_r(\gm)= \max \set{\gk_r, \ol D_r(\gn)}.
\]

\begin{remark}\label{rem24}
Proceeding in a similar way as \cite[Lemma~4.14]{GL2}, we have
\begin{align*} \label{eq111}V_{n, r} (\gm) \left\{ \begin{array} {ll} \geq p_0 V_{n, r}(\gn) +\sum_{j=1}^N p_js_j^r V_{n, r}(\gm),
\\
\leq p_0 V_{\lfloor \frac n{N+1}\rfloor, r}(\gm)+\sum_{j=1}^N p_js_j^r  V_{\lfloor \frac n{N+1}\rfloor, r}(\gm),
\end{array}
\right.
\end{align*}
where  $\lfloor x\rfloor$ represents the greatest integer not exceeding $x$ for any real $x$. Thus,  for any $s>0$, we have
\begin{align*}
& (1-\sum_{j=1}^N p_j s_j^r)\liminf_n n^{r/s} V_{n, r}(\gn) \leq \liminf_n n^{r/s} V_{n, r}(\gm) \leq \limsup_n n^{r/s} V_{n, r}(\gm) \\
& \leq  p_0 2^{r/s} (N+1)^{r/s}\limsup_n n^{r/s} V_{n, r}(\gn) +\sum_{j=1}^N p_js_j^r  2^{r/\gk_r} (N+1)^{r/\gk_r}\limsup_n n^{r/s} V_{n, r}(\gm).
\end{align*}
Let $s=\max\set {\gk_r, \underline D_r(\gn)}$, or $s=\max\set {\gk_r, \ol D_r(\gn)}$. Then, from the last inequalities, we see that $\liminf_n n^{r/s} V_{n, r}(\gm)$ and $\limsup_n n^{r/s} V_{n, r}(\gm)$, can either or both be zero, positive, or infinity. Thus, we can not draw any useful conclusion whether the following inequalities are correct:
\[
\underline D_r(\gm)\leq \max \set{\gk_r, \underline D_r(\gn)}, \te{ and }  \ol D_r(\gm) \leq \max \set{\gk_r, \ol D_r(\gn)}.
\]
\end{remark}

In the sequel we will denote the condensation measure as $P$. Take $r=2$, and consider the condensation system as $(\set{S_1, S_2}, (\frac 13, \frac 13, \frac 13), \gn)$ on the Euclidean space $\D R$, i.e., the condensation measure $P$ is given by $P:=\frac 13 P \circ S_1^{-1}+\frac 13 P\circ S_2^{-1}+\frac 13 \gn$, where $\gn$ is a Borel probability measure on $\D R$ with a compact support. We define the two contractive similarity mappings $S_1$ and $S_2$ on $\D R$ as follows: $S_1(x)=\frac 15 x$ and $S_2(x)=\frac 1 5 x+\frac 45$ for all $x\in \D R$.  Then, for $n\in \D N$, by \eqref{eq010}, we have
\begin{equation} \label{eq020} P=\frac 1{3^n} \mathop{\sum}\limits_{|\go|=n} P\circ S_{\go}^{-1} +\mathop{\sum}\limits_{k=0}^{n-1}\frac 1{3^{k+1}} \mathop{\sum}\limits_{|\go|=k} \gn \circ S_{\go}^{-1}.
\end{equation}
Write $J:=[0, 1]$. For $\go=\go_1\go_2 \cdots\go_k \in I^k$, set  $J_\go:=S_\go(J)$, and $C_\go:=S_\go(C)$.

The following lemma can easily be proved.

\begin{lemma} Let $K$ be the support of the condensation measure. Then, for any $n\geq 1$,
\[K\sci (\mathop{\uu}\limits_{\go\in I^n} J_{\go})\UU  \Big (\mathop{\uu}\limits_{k=0}^{n-1}(\mathop{\uu}\limits_{\go\in I^k} C_\go)\Big )\sci J.\]
\end{lemma}

By equation~\eqref{eq020}, we can deduce the following lemma.

\begin{lemma} \label{lemma1} Let $g: \D R \to \D R^+$ be Borel measurable and $n\in \D N$. Then,
\[\int g(x) \,dP(x)=\frac 1 {3^n} \sum_{|\go|=n}\int (g\circ S_\go)(x) \,dP(x)+\sum_{k=0}^{n-1} \frac 1 {3^{k+1}}\sum_{|\go|=k} \int (g\circ S_\go)(x) \,d\gn(x).\]
\end{lemma}
The following lemma that appears in \cite{CR} is also true.

\begin{lemma} (see \cite[Lemma~3.7]{CR})  \label{lemma001} Let $\ga$ be an optimal set of $n$-means for the condensation measure $P$. Then, for any $\go \in I^\ast$, the set $S_\go(\ga):=\set{S_\go(a) : a \in \ga}$ is an optimal set of $n$-means for the image measure $P\circ S_\go^{-1}$. Conversely, if $\gb$ is an optimal set of $n$-means for the image measure $P\circ S_\go^{-1}$, then $S_\go^{-1}(\gb)$ is an optimal set of $n$-means for $P$.
\end{lemma}
In the following two sections, we investigate the quantization for two condensation measures $P$: in Section~\ref{sec1}, $P$ is associated with an infinite discrete distribution $\gn$ borrowed from our recent paper \cite{DRV}; and in Section~\ref{sec2}, $P$ is associated with a uniform distribution.
\section{Quantization for a condensation measure associated with a discrete distribution} \label{sec1}
Let $\gn$ be an infinite discrete distribution with support $C:=\set{a_j : j\in \D N}$, where $a_j:=\frac{2}{5}+\frac 1 5\frac{2^{j-1}-1}{2^{j-1}}$. Notice that $C=\left\{\frac{2}{5},\frac{1}{2},\frac{11}{20},\frac{23}{40},\frac{47}{80},\frac{19}{32},\frac{191}{320},\frac{383}{640}, \cdots \right\}\sci [\frac 25, \frac 35]\sci \D R$. Let the probability mass function of $\gn$ be given by $f$ such that
\[f(x)=\left\{\begin{array}{cc}
\frac 1{2^j} & \te{ if }   x=a_j \te{ for } j\in \D N,  \\
 0 &  \te{ otherwise}.\end{array}\right.\]
  Set $A_k:=\set{a_k, a_{k+1}, a_{k+2}, \cdots}$, and $b_k=E(X: X \in A_k)$, where $k\in \D N$. Then, by the definition of conditional expectation, we see that
 \[b_k=\frac{\sum_{j=k}^\infty a_j\frac 1{2^j}}{\sum_{j=k}^\infty \frac 1{2^j}}.\]
 For any Borel measurable function $g$ on $\D R$, by $\int g d\gn$, it is meant
 \[\int g d\gn=\sum_{j=1}^\infty g(a_j) \gn(a_j)=\sum_{j=1}^\infty g(a_j)\frac 1{2^j}.\]
 Then $E(\gn)=\frac{7}{15}, \te{ and } V(\gn)=\frac{8}{1575},$ where $E(\gn)$ and $V(\gn)$ represent the expected value of a random variable with distribution $\gn,$ and its variance, respectively.
The quantization dimension $D(\gn)$ of the measure $\gn$ exists and equals zero, for more details one can see \cite{DRV}.
In order to compute the quantization dimension of condensation measure $P$ associated with discrete distribution $\gn,$ first we give the following lemmas and proposition.
%

\begin{lemma} \label{lemma3} Let $E(P)$ represent the expected value and $V:=V(P)$ represent the variance of the condensation measure $P$. Then,
$E(P)=\frac{19}{39},$ $V=\frac{86696}{777231}$, and for any $x_0\in \D R$, $\int(x-x_0)^2 dP=(x_0-\frac{19}{39})^2+V(P)$.
\end{lemma}
\begin{proof}
By Lemma~\ref{lemma1} and Lemma~\ref{lemma001}, we can see that $E(P)=\int x dP=\frac{19}{39}$. Again,
$\int x^2 d\gn=\sum_{j=1}^\infty a_j^2 \frac 1{2^j}=\frac{39}{175}.$ We have
\begin{align*}
&\int x^2 dP=\frac 1 3 \int  (S_1(x))^2 dP+\frac 13 \int (S_2(x))^2 dP +\frac 13 \int x^2 d\gn\\
&=\frac 1 3  \int\Big (\frac {1} {5} x\Big)^2 dP +\frac 13  \int \Big(\frac 15  x +\frac 45 \Big)^2 dP+\frac 13 \frac{39}{175}\\
&= \frac 1{75} \int x^2 dP +\frac 1{75} \int x^2 dP +\frac 8{75} \int x dP+\frac {16}{75} +\frac{13}{175}
\end{align*}
which implies $\int x^2 dP=\frac{6953}{19929}$, and hence,
$V(P)=\int x^2 dP-(\int x dP)^2=\frac{6953}{19929}-\left(\frac{19}{39}\right)^2=\frac{86696}{777231}$. For any $x_0 \in \D R$, $\int(x-x_0)^2 dP=(x_0-\frac{19}{39})^2+V(P)$ follows from the standard theory of probability. Thus, the proof of the lemma is complete.
\end{proof}
We now give the following proposition.
\begin{prop}
Let $\go\in I^k$, $k\geq 0$, and let $X$ be the random variable with probability distribution $P$. Then,
$E(X : X \in J_\go)=S_\go(\frac{19}{39}), \te{ and }E(X : X \in C_\go)=S_\go(\frac 7{15}).
$
Moreover, for any $x_0\in \D R$, any $\go\in I^k$, $k\geq 0$, we have
\begin{align} \label{eq4} \left\{\begin{array}{cc}  \int_{J_\go} (x-x_0)^2 dP(x) =\frac 1{3^k} \Big(\frac 1{25^k} V  +(S_\go(\frac{19}{39})-x_0)^2\Big), \\
\int_{C_\go} (x-x_0)^2 dP(x) =\frac 1{3^{k+1}} \Big(\frac 1{25^k} V(\gn)  +(S_\go(\frac{7}{15})-x_0)^2\Big),\\
\int_{S_\go(\set {a_1, a_2})} (x-x_0)^2 dP(x) =\frac 1{3^{k+1}}\Big((S_\go(a_1)-x_0)^2\frac 1{2}+(S_\go(a_2)-x_0)^2\frac 1{2^2}\Big), \te{ and } \\
\int_{S_\go(A_3)} (x-x_0)^2 dP(x) =\frac 1{3^{k+1}}\sum_{j=3}^\infty (S_\go(a_j)-x_0)^2\frac 1{2^j}.\end{array}\right.
\end{align}
\end{prop}

\begin{proof}  By equation~\eqref{eq020}, we have $P(J_\go)=\frac 1{3^k}$ and $P(C_\go)=\frac 1{3^{k+1}}$. Then, by Lemma~\ref{lemma1}, the proof follows.
\end{proof}

\begin{note}
By Lemma~\ref{lemma3}, it follows that the optimal set of one-mean for the condensation measure $P$ consists of the expected value $\frac{19}{39}$ and the corresponding quantization error is the variance $V(P)$ of $P$, i.e., $V_1(P)=V(P)$.
\end{note}

\subsection{Essential lemmas and propositions}
In this subsection, we give some lemmas and propositions that we need to determine the optimal sets of $n$-means and the $n$th quantization errors for all $n\geq 2$. To determine the quantization error we will frequently use the formulas given in the expression \eqref{eq4}.
\begin{prop} (see \cite{GG, GL1})\label{prop1000}
	Let $\ga$ be an optimal set of $n$-means, $a \in \ga$, and $M(a|\ga)$ be the Voronoi region generated by $a\in \ga$.
	Then, for every $a \in\ga$,
	$(i)$ $\gm(M(a|\ga))>0$, $(ii)$ $ \gm(\partial M(a|\ga))=0$, $(iii)$ $a=E(X : X \in M(a|\ga))$, and $(iv)$ $\gm$-almost surely the set $\set{M(a|\ga) : a \in \ga}$ forms a Voronoi partition of $\D R^d$.
\end{prop}
\begin{propo} \label{prop01111}
Let $\ga:=\set{c_1, c_2}$ be an optimal set of two-means with $c_1<c_2$, and let $V_{2,1}$ and $V_{2,2}$, respectively, be the distortion errors contributed by the points $c_1$ and $c_2$. Then, $c_1=\frac{659}{2730}$, $c_2=\frac{1621}{1950}$,
 and the corresponding quantization error is $V_2=V_{2,1}+V_{2,2}=0.0269686$, where $V_{2,1}=\frac{8469163}{466338600} \te{ and } V_{2,2}=\frac{20536627}{2331693000}$, respectively, denote the distortion errors due to $c_1$ and $c_2$.
\end{propo}
\begin{proof} Consider the set of two points $\gb:=\set{d_1, d_2}$, where $d_1=E(X : X \in J_1\uu \set{a_1, a_2})$ and $d_2=E(X : X \in A_3\uu J_2)$, i.e.,
\begin{align*}
d_1&=\frac{1}{P(J_1\uu \set{a_1, a_2})}\Big(\int_{J_1} x dP+\frac 13(a_1\frac 12+a_2\frac 1{2^2})\Big)=\frac{659}{2730},
\end{align*}
and
\[d_2=\frac{1}{P(A_3\uu J_2)}\Big(\frac 13\sum_{j=3}^\infty a_j\frac 1{2^j}+\int_{J_2}x dP\Big)=\frac{1621}{1950}.\]
The distortion error due to the set $\gb$ is given by
\begin{align*}
&\int\min_{d\in \gb} (x-d)^2 dP\\
& =\int_{J_1}(x-\frac{659}{2730})^2 dP+\frac 1 3\Big((a_1-\frac{659}{2730})^2\frac 12+(a_2-\frac{659}{2730})^2 \frac 1{2^2} + \sum_{j=3}^\infty(a_j-\frac{1621}{1950})^2\frac 1{2^j}\\
&+\int_{J_2}(x-\frac{1621}{1950})^2dP=\frac{499067}{18505500}=0.0269686.
\end{align*}
Since $V_2$ is the quantization error for two-means, we have $V_2\leq 0.0269686$. Let $\ga:=\set{c_1, c_2}$ be an optimal set of two-means for $P$. Suppose that $\frac 25\leq c_1$. Then,
\[V_2\geq \int_{J_1}(x-\frac 25)^2 dP=\frac{47833}{1494675}=0.0320023>V_2,\]
which leads to a contradiction. So, we can assume that $c_1<\frac 25$. Similarly, we can show that $\frac 35<c_2$. Since $\frac 3{10}\leq \frac 12(c_1+c_2)\leq \frac 7{10}$, the Voronoi region of $c_1$ cannot contain any point from $J_2$, and the Voronoi region of $c_2$ cannot contain any point from $J_1$. Suppose that the Voronoi region of $c_1$ does not contain any point from $A_1$. Then,
$c_1=E(X: X \in J_1)=\frac{19}{195}$, and $c_2=E(X : X \in A_2\uu J_2)=\frac{133}{195}$, and then the distortion error is
\[\int_{J_1}(x-c_1)^2 dP+\frac 13\sum_{j=1}^\infty (a_j-c_2)^2 \frac 1{2^j}+\int_{J_2}(x-c_2)^2 dP=\frac{76848}{2158975}=0.0355947>V_2,\]
which leads to a contradiction. Similarly, we can show that if the Voronoi region of $c_2$ does not contain any point from $A_1$ a contradiction will arise. Thus, we conclude that both the Voronoi regions of $c_1$ and $c_2$ contain at least one point from $A_1$. Let the Voronoi region of $c_1$ contains points $a_1, a_2, \cdots, a_k$ from $A_1$, and the Voronoi region of $c_2$ contains $\set{a_j : j> k}$ for some positive integer $k\geq 1$. Let $F(k)$ be the corresponding distortion error. Then,
\[F(k)=\int_{J_1}(x-c_1)^2 dP+\frac 13\Big(\sum_{j=1}^k (a_j-c_1)^2\frac 1{2^j}+\sum_{j=k+1}^\infty (a_j-c_2)^2 \frac 1{2^j}\Big)+\int_{J_2}(x-c_2)^2dP.\]
Using Calculus we see that
$\min\set{F(k) : k\in \D N}=\frac{499067}{18505500}=0.0269686,$
which occurs when $k=2$. This yields the fact that $c_1=E(X: X\in J_1\uu  \set{a_1, a_2})=\frac{659}{2730}$, and $c_2=E(X : X \in A_3\uu J_2)=\frac{1621}{1950}$, and the corresponding quantization error is $V_2= V_{2,1}+V_{2,2}=\frac{499067}{18505500}=0.0269686$, where
$V_{2,1}=\int_{J_1}(x-c_1)^2 dP+\frac 13\sum_{j=1}^2 (a_j-c_1)^2\frac 1{2^j}=\frac{8469163}{466338600}$, and $V_{2,2}=\frac 13\sum_{j=3}^\infty (a_j-c_2)^2 \frac 1{2^j}+\int_{J_2}(x-c_2)^2dP=\frac{20536627}{2331693000}$. Thus, the proof of the proposition is complete.
\end{proof}

\begin{lemma1}  \label{lemma0} Let $\go \in I^k$ for  $k\geq 0$. Then,
\[\int_{J_{\go 1}\uu S_{\go}(\set{a_1, a_2})}(x-S_{\go}(\frac{659}{2730}))^2 dP=\frac {1}{75^{k}}  V_{2,1}, \te{ and } \int_{S_{\go}(A_3)\uu J_{\go 2}}(x-S_{\go}(\frac{1621}{1950}))^2 dP=\frac {1}{75^{k}}V_{2,2}.\]
\end{lemma1}
\begin{proof}For any $\go \in I^k$, $k\geq 0$, we have
\begin{align*}
&\int_{J_{\go 1}\uu S_{\go}(\set{a_1, a_2})}(x-S_{\go}(\frac{659}{2730}))^2 dP=\frac 1{3^k} \int_{J_{\go 1}\uu S_{\go}(\set{a_1, a_2})}(x-S_{\go}(\frac{659}{2730}))^2 d(P\circ S_\go^{-1})\\
&=\frac 1{3^k}\int_{J_1\uu \set{a_1, a_2}}(S_\go(x)-S_{\go}(\frac{659}{2730}))^2 dP=\frac 1{3^k} \frac 1{25^k} \int_{J_1\uu \set{a_1, a_2}}(x-\frac{659}{2730})^2 dP= \frac {1}{75^{k}}  V_{2,1}.
\end{align*}
Similarly, $\int_{S_{\go}(A_3)\uu J_{\go 2}}(x-S_{\go}(\frac{1621}{1950}))^2 dP=\frac {1}{75^{k}}V_{2,2}$. Thus, the proof of the lemma is complete.
\end{proof}

From the above lemma the following corollary follows.
\begin{cor} \label{cor1}
Let $\go \in I^k$ for  $k\geq 0$. Then, for any $a\in \D R$,
\begin{align*}
&\int_{J_{\go 1}\uu S_{\go}(\set{a_1, a_2})}(x-a)^2 dP=\frac 1{3^k}\Big(\frac 1{25^k}  V_{2,1}+\frac{7}{12}(S_\go(\frac{659}{2730})-a)^2\Big), \te{ and } \\
&\int_{S_{\go}(A_3)\uu J_{\go 2}}(x-a)^2 dP=\frac 1{3^k}\Big(\frac 1{25^k} V_{2,2}+\frac 5{12}(S_\go(\frac{1621}{1950})-a)^2\Big).
\end{align*}
\end{cor}

\begin{propo}\label{prop00002}
Let $\ga:=\set{a_1, a_2, a_3}$ be an optimal set of three-means with $a_1<a_2<a_3$. Then, $a_1=S_1(\frac{19}{39})$, $a_2=E(\gn)=\frac 7{15}$, and $a_3=S_2(\frac{19}{39})$. The corresponding quantization error is $V_3=\frac{30232}{6476925}=0.00466765$.
\end{propo}

\begin{proof} Let $\gb:=\set{S_1(\frac{19}{39}), \frac 7{15}, S_2(\frac{19}{39})}$. Then, using \eqref{eq4}, we have
\begin{align*}
&\int\min_{a\in \gb} (x-a)^2 dP=\int_{J_1}(x-\frac {1}{10})^2 dP+\int_{C}(x-\frac 7{15})^2 dP+\int_{J_2}(x-\frac 9{10})^2 dP\\
&=\frac 13\frac 1{25} V+\frac 13 V(\gn)+\frac 13\frac 1{25} V=\frac{30232}{6476925}=0.00466765.
\end{align*}
Since $V_3$ is the quantization error for three-means, we have $0.00466765\geq V_3$.
 Let $\ga:=\{c_1, c_2, c_3\}$ be an optimal set of three-means with $c_1<c_2<c_3$. Since $c_1$, $c_2$, and $c_3$ are the expected values of their own Voronoi regions, we have $0 <c_1<c_2<c_3< 1$. Suppose that $\frac 15 \leq c_1$. Then,
 \[V_3\geq \int_{J_1}(x-\frac 15)^2 dP=\frac{2488}{498225}=0.00499373>V_3,\]
 which is a contradiction. So, we can assume that $c_1<\frac 15$. We now show that the Voronoi region of $c_1$ does not contain any point from $C$. For the sake of contradiction, assume that the Voronoi region of $c_1$ contains at least one point, say $a_1$, from $C$. Then,
 \[V_3\geq \min\set{\int_{J_1}(x-c_1)^2dP+\frac 13(a_1-c_1)^2 \frac 12 : 0<c_1<\frac 15}=\frac{2038879}{174876975}=0.0116589>V_3,\]
 which leads to a contradiction. Hence, the Voronoi region of $c_1$ does not contain any point from $C$.
 If $c_3\leq (\frac{4}{5}-\frac{1}{1000})$, then
 \[V_3\geq \int_{J_2}(x-(\frac{4}{5}-\frac{1}{1000}))^2 dP=\frac{94007843}{19929000000}=0.00471714>V_3,\]
 which leads to a contradiction. Hence, $(\frac{4}{5}-\frac{1}{1000})<c_3$. Suppose that $c_2\leq  \frac 25$. Then,  $a_9<\frac{1}{2} \left(\frac{2}{5}+(\frac{4}{5}-\frac{1}{1000})\right)<a_{10}$, and so, we have
\begin{align*}
V_3&\geq \int_{J_1}\min_{c\in S_1(\ga_2)}(x-c)^2 dP+\frac 13\sum_{j=1}^9(a_j-\frac 25)^2\frac 1{2^j}+\int_{J_2}(x-S_2(\frac{19}{39})^2 dP\\
&=\frac{1}{75} (V_{2,1}+V_{2,2}) +\frac 13\sum_{j=1}^9(a_j-\frac 25)^2\frac 1{2^j}+\frac 1{75} V=\frac{1221371823004331}{244495731916800000}=0.00499547>V_3,
\end{align*}
which leads to a contradiction. Hence, we can assume that $\frac 25< c_2$. Then, as $\frac 15<\frac 12(c_1+c_2)$, the Voronoi region of $c_2$ does not contain any point from $J_1$. Recall that we already proved that the Voronoi region of $c_1$ does not contain any point from $C$. Assume that $(\frac{4}{5}-\frac{1}{1000})<c_3\leq \frac 45$. Then, $\frac 12(c_2+c_3)\leq \frac 35$ implying $c_2\leq \frac 65-c_3\leq \frac{6}{5}-(\frac{4}{5}-\frac{1}{1000})=\frac{401}{1000}<a_2$, and so
\[V_3\geq \frac 13\sum_{j=2}^\infty(a_j-0.401)^2\frac 1{2^j}+\int_{J_2}(x-\frac 45)^2 dP=\frac{310181843}{39858000000}=0.00778217>V_3,\]
which is a contradiction. So, we can assume that $\frac 45<c_3$. If the Voronoi region of $c_3$ contains points form $C$, we have $\frac 12(c_2+c_3)<\frac 35$ implying $c_2\leq \frac 65-c_3\leq \frac 65-\frac 45=\frac 25$, which leads to a contradiction as we have seen $\frac 25<c_2$. Hence, the Voronoi region of $c_3$ does not contain any point from $C$. If $\frac 35\leq c_2$, then
\begin{align*}
V_3&\geq \int_{J_1}(x-S_1(\frac {19}{39}))^2 dP+\frac 13\sum_{j=1}^\infty(a_j-\frac 35)^2 \frac 1{2^j}+\int_{J_2}\min_{c\in \ga_2}(x-c)^2 dP\\
&=\frac 1{75} V+\frac 13\sum_{j=1}^\infty(a_j-\frac 35)^2 \frac 1{2^j}+\frac 1{75}(V_{2,1}+V_{2,2})=\frac{275894407}{29146162500}=0.00946589>V_3
\end{align*}
which gives a contradiction. So, we can assume that $c_2<\frac 35$. Then, notice that the Voronoi region of $c_2$ does not contain any point from $J_3$ as $\frac 45<c_3< 1$. Hence, we have $c_1=S_1(\frac{19}{39})$, $c_2=E(\gn)=\frac 7{15}$, and $c_3=S_2(\frac{19}{39})$, and the corresponding quantization error is $V_3=\frac{30232}{6476925}=0.00466765$. Thus, the proof of the proposition is complete.
\end{proof}

\begin{lemma1}  \label{lemma00001} Let $\ga$ be an optimal set of four-means. Then $\ga\ii J_1\neq \es$, $\ga\ii J_2\neq \es$, and $\te{card}(\ga\ii C)=2$.
\end{lemma1}

\begin{proof}
The distortion error due to the set $\gb:=\set{S_1(\frac{19}{39}), S_2(\frac{19}{39})}\uu \ga_2(\gn)$ is given by
\[\int \min_{a\in \gb} (x-a)^2 dP=\frac 1{75}V+\frac 13 \sum_{j=2}^\infty (a_j-b_2)^2\frac 1{2^j}+\frac 1{75}V=\frac{185729}{58292325}=0.00318617.\]
Since $V_4$ is the quantization error for four-means, we have $V_4\leq 0.00318617$. Let an optimal set of four-means be given by $\ga:=\set{c_1, c_2, c_3, c_4}$ such that $c_1<c_2<c_3<c_4$. Since $c_1, c_2, c_3$, and $c_4$ are the expected values of their own Voronoi regions, we have $0<c_1<c_2<c_3< c_4<1$. If $c_1\geq \frac 15$, then
\[V_4\geq \int_{J_1}(x-\frac 15)^2 dP=\frac{2488}{498225}=0.00499373>V_4,\]
which leads to a contradiction. So, we can assume that $c_1<\frac 15$. If $c_4\leq \frac 45$, then
 \[V_4\geq \int_{J_2}(x-\frac 45)^2 dP= \frac{6953}{1494675}=0.00465185>V_4,\]
 which is a contradiction. So, we can assume that $\frac 45<c_4$. Thus, we see that $\ga\ii J_1\neq \es$, and $\ga\ii J_2\neq \es$. Assume that $\ga\ii C=\es$. Then, the following cases can arise:

 Case~1. $c_3<\frac 25$.

 Then,
 \[V_4\geq \frac 13 \sum_{j=1}^\infty (a_j-\frac 25)^2 \frac 1{2^j}+\int_{J_2}(x-S_2(\frac {19}{39}))^2 dP=\frac{271751}{58292325}=0.00466187>V_4,\]
 which is a contradiction.

Case~2. $c_2<\frac 25<\frac 35<c_3$.

First, assume that $\frac 7{20}\leq c_2<\frac 25$, and $\frac 35<c_3\leq \frac {13}{20}$. Then, $\frac 12(c_1+c_2)<\frac 15$ implying $c_1<\frac 25-\frac 7{20}=\frac 1{20}$. Similarly $\frac 12(c_3+c_4)>\frac 45$ implying $\frac {19}{20}<c_4$. Then,
\begin{align*} V_4&\geq \int_{S_1(A_3)\uu J_{12}}(x-\frac 1{20})^2 dP+\frac 13 \sum_{j=1}^2 (a_j-\frac 25)^2 \frac 1{2^j}+\frac 13 \sum_{j=3}^\infty (a_j-\frac 35)^2 \frac 1{2^j}\\
&+ \int_{J_{21}\uu S_1(\set{a_1, a_2})}(x-\frac {19}{20})^2 dP=\frac{143519}{27594000}=0.00520109>V_4,
\end{align*}
which gives a contradiction. Next, assume that $\frac 7{20}\leq c_2<\frac 25$, and $\frac {13}{20}\leq c_3$. Then, $c_1<\frac 1{20}$, and $a_2<\frac 12(\frac 25+\frac {13}{20})<a_3$, and so
\begin{align*} V_4&\geq \int_{S_1(A_3)\uu J_{12}}(x-\frac 1{20})^2 dP+\frac 13 \sum_{j=1}^2 (a_j-\frac 25)^2 \frac 1{2^j}+\frac 13 \sum_{j=3}^\infty (a_j-\frac{13}{20})^2 \frac 1{2^j}\\
&=\frac{473663}{137970000}=0.00343309>V_4,
\end{align*}
which leads to a contradiction.
Last, assume that $c_2<\frac 7{20}$, and $\frac{13}{20}<c_3$. Then, as
  $\frac 12(\frac 7{20}+\frac {13}{20})=\frac 12$, and $\int_{J_1} \min_{c\in S_1(\ga_2)}(x-c)^2 dP=\frac 1{75}(V_{2,1}+V_{2,2})$, we have
 \[V_4\geq\frac 1{75}(V_{2,1}+V_{2,2})+ \frac 13 \sum_{j=1}^2 (a_j-\frac 7{20})^2 \frac 1{2^j}+\frac 13 \sum_{j=3}^\infty (a_j-\frac {13}{20})^2 \frac 1{2^j}=\frac{126549001}{38861550000}=0.00325641>V_4,\]
 which is a contradiction.

Case~3. $\frac 35<c_2$.

Since $c_1<\frac 15$, this is the reflection of Case~1 about the point $\frac 12$, thus a contradiction arises in this case as well.

Hence, we can conclude that $\ga\ii C\neq \es$. We now show that $\te{card}(\ga\ii C)=2$. For the sake of contradiction, assume that $\te{card}(\ga\ii C)=1$. Recall that $\ga\ii J_1\neq \es$, and $\ga\ii J_2\neq \es$. Thus, either $c_2\in C$, or $c_3\in C$. Without any loss of generality assume that $c_3\in C$. Then, $c_2<\frac 25$. The following cases can arise:

Case~I.  $\frac 7{20}\leq c_2<\frac 25$.

Then, $\frac 12(c_1+c_2)<\frac 15$ implying $c_1<\frac 1{20}$, and so
\[V_4\geq  \int_{S_1(A_3)\uu J_{12}}(x-\frac 1{20})^2 dP +\frac 13 V_2(\gn)+\frac 1{75}V=\frac{344484113}{93267720000}=0.0036935>V_4,\]
which leads to a contradiction.

Case~II. $\frac 3{10} \leq c_2\leq \frac 7{20}$.

Then, $\frac 12(c_1+c_2)<\frac 15$ implying $c_1<\frac 1{10}$. Moreover, $a_1<\frac 12(\frac 7{20}+b_2)<a_2$. First assume that $\frac 2{25}\leq c_1<\frac 1{10}$.
Then,
\begin{align*} V_4&\geq \int_{J_{11}}(x-\frac 2{25})^2 dP+\int_{S_{1}(A_3)\uu J_{12}}(x-\frac 1{10})^2 dP+\frac 13 (a_1-\frac 7{20})^2 \frac 12+\frac 13\sum_{j=2}^\infty (a_j-b_2)^2 \frac 1{2^j}+\frac 1{75} V\\
&=\frac{152465267}{46633860000}=0.00326941>V_4,
\end{align*}
which is a contradiction. Next, assume that $c_1\leq \frac 2{25}$.
Then, $S_{12}(\frac 35)<\frac 12(\frac 2{25}+\frac 3{10})<S_{122}(0)$, and so
\begin{align*}
V_4&\geq \int_{J_{11}}(x-S_{11}(\frac {19}{39}))^2 dP+\int_{S_1(C)}(x-\frac 2{25})^2 dP+\int_{J_{121}}(x-\frac 2{25})^2 dP+\int_{S_{12}(C)}(x-\frac 2{25})^2 dP\\
&+\int_{J_{122}}(x-\frac 3{10})^2 dP+\frac 13 (a_1-\frac 7{20})^2 \frac 12+\frac 13\sum_{j=2}^\infty (a_j-b_2)^2 \frac 1{2^j}+\frac 1{75} V\\
&=\frac{89570473}{27980316000}=0.0032012>V_4,
\end{align*}
which leads to a contradiction.

Case~III. $c_2\leq \frac 3{10}$.

Then, $a_1<\frac 12(\frac 3{10}+b_2)<a_2$, and so
\begin{align*}
V_4&\geq \frac 13 (a_1-\frac 3{10})^2 \frac 12+\frac 13\sum_{j=2}^\infty (a_j-b_2)^2 \frac 1{2^j}+\frac 1{75} V=\frac{104633}{31089240}=0.00336557>V_4,
\end{align*}
which is a contradiction.

Thus, $\te{card}(\ga\ii C)=1$ leads to a contradiction. Hence, we can assume that  $\te{card}(\ga\ii C)=2$. Thus, the proof of the lemma is complete.
\end{proof}

From Lemma~\ref{lemma00001}, the following proposition follows.

\begin{propo}  \label{prop000021} The set $\set{S_1(\frac{19}{39}), S_2(\frac{19}{39})}\uu \ga_2(\gn)$ is an optimal set of four-means with quantization error $V_4=\frac{185729}{58292325}=0.00318617$.
\end{propo}

We now prove the following proposition.

\begin{propo}\label{prop00003}
Let $n\in \D N$ and $n\geq 3$, and let $\ga$ be an optimal set of $n$-means. Then, $\ga\ii J_1\neq \es$, $\ga\ii J_2\neq \es$, and $\ga\ii C\neq \es$.
\end{propo}
\begin{proof} By Proposition~\ref{prop00002} and Lemma~\ref{lemma00001}, the proposition is true for $n=3, 4$. We now prove that the proposition is true for $n\geq 5$. Consider the set of five points $\gb:=S_1(\ga_2)\uu \ga_2(\gn)\uu \set{S_2(\frac {19}{39})}$. The distortion error due to the set $\gb$ is given by \[\int\min_{c\in\gb}(x-c)^2 dP=\frac 1{75}(V_{2,1}+V_{2,2})+\frac 13 V_2(\gn)+\frac1{75}V=\frac{6666323}{3238462500}=0.00205848.\]
Since $V_n$ is the quantization error for $n$-means with $n\geq 5$, we have $V_n\leq V_5\leq 0.00205848$. Let $\ga:=\set{c_1<c_2<\cdots<c_n}$ be an optimal set of $n$-means for $P$. In the similar way as shown in Lemma~\ref{lemma00001}, we can show that $\ga\ii J_1\neq \es$ and $\ga\ii J_2\neq \es$. For the sake of contradiction, assume that $\ga\ii C=\es$. Let $j:=\max\set{i : c_i<\frac 25}$. Then, $c_j<\frac 25$ and $c_{j+1}>\frac 35$. The following cases can arise:

Case~1. $\frac 3{10}<c_j<\frac 25$ and $\frac 35<c_{j+1}<\frac 7{10}$.

Then, $\frac 1 2(c_{j-1}+c_j)<\frac 15$ and $\frac 12(c_{j+1}+c_{j+2})>\frac 45$ yielding $c_{j-1}<\frac 1{10}$ and $c_{j+1}>\frac 9{10}$. Thus,
\begin{align*} V_n&\geq \int_{S_1(A_3)\uu J_{12}}(x-\frac 1{10})^2 dP+\frac 13 \sum_{j=1}^2 (a_j-\frac 25)^2 \frac 1{2^j}+\frac 13 \sum_{j=3}^\infty (a_j-\frac 35)^2 \frac 1{2^j}\\
&+ \int_{J_{21}\uu S_1(\set{a_1, a_2})}(x-\frac {9}{10})^2 dP=\frac{1123}{459900}=0.00244184>V_n,
\end{align*}
which gives a contradiction.

Case~2.  $\frac 3{10}<c_j<\frac 25$ and $\frac 7{10}\leq c_{j+1}$.
 Then, $c_{j-1}<\frac 1{10}$, and $a_2<\frac 12(\frac 25+\frac {7}{10})=a_3$, and so
\begin{align*} V_n&\geq \int_{S_1(A_3)\uu J_{12}}(x-\frac 1{10})^2 dP+\frac 13 \sum_{j=1}^2 (a_j-\frac 25)^2 \frac 1{2^j}+\frac 13 \sum_{j=3}^\infty (a_j-\frac{7}{10})^2 \frac 1{2^j}\\
&=\frac{1834513}{597870000}=0.00306841>V_n,
\end{align*}
which leads to a contradiction.

Case~3.  $c_j\leq \frac 3{10}$ and $\frac 7{10}\leq c_{j+1}$.

Then, as
  $\frac 12(\frac 3{10}+\frac {7}{10})=\frac 12$, we have
 \[V_n\geq \frac 13 (a_1-\frac 3{10})^2 \frac 1{2}+\frac 13 \sum_{j=2}^\infty (a_j-\frac {7}{10})^2 \frac 1{2^j}=\frac{41}{6300}=0.00650794>V_n,\]
 which is a contradiction.

Thus, we see that $\ga\ii C=\es$ leads to a contradiction. Hence, we can assume that $\ga\ii C\neq \es$. Thus, the proof of the proposition is complete.
\end{proof}

\begin{lemma1}  \label{lemma00002} Let $\ga$ be an optimal set of five-means. Then, $\ga$ does not contain any point from the open intervals $(\frac 15, \frac 25)$ and $(\frac 35, \frac 45)$.
\end{lemma1}

\begin{proof}
Let $\ga:=\set{c_1<c_2<c_3<c_4<c_5}$ be an optimal set of five means. As shown in the proof of Proposition~\ref{prop00003}, we have $V_5\leq 0.00205848$. First, prove the following claim.

Claim. $\te{card}(\ga\ii C)=2$.

Since $\ga\ii J_1\neq \es$ and $\ga\ii J_2\neq \es$, we have $\te{card}(\ga\ii C)\leq 3$. If $\te{card}(\ga\ii C)=3$, then
\[V_5\geq \int_{J_1}(x-S_1(\frac {19}{39}))^2 dP+\int_{J_2}(x-S_2(\frac{19}{39}))^2 dP=\frac 2{75} V=\frac{173392}{58292325}=0.00297453>V_5,\]
which is a contradiction. So, we can assume that $\te{card}(\ga\ii C)\leq 2$. Suppose that $\te{card}(\ga\ii C)=1$. Then, the following cases can arise:

Case~1. $c_4\in C$.

Then, $c_3<\frac 25$. Assume that $\frac 3{10}<c_3<\frac 25$. Then, $\frac 12(c_2+c_3)<\frac 15$ implying $c_2<\frac 1{10}$, and so
\[V_5\geq \int_{S_1(A_3)\uu J_{12}}(x-\frac 1{10})^2 dP+\int_{J_2}(x-S_2(\frac{19}{39}))^2 dP=\frac{25816591}{11658465000}=0.00221441>V_5,\]
which gives a contradiction. Next, assume that $c_3\leq \frac 3{10}$. Then, $a_1<\frac 1 2(\frac 3{10}+b_2)<a_2$, and so
\[V_5\geq \frac 13(a_1-\frac 3{10})^2 \frac 12+\frac 13\sum_{j=2}^\infty(a_j-b_2)^2 \frac 1{2^j}+\int_{J_2}(x-S_2(\frac{19}{39}))^2 dP=\frac{104633}{31089240}=0.00336557>V_5,\]
which is a contradiction. Thus, $c_4\in C$ leads to a contradiction.

Case~2. $c_3\in C$.

Then, $c_2<\frac 25$ and $\frac 35<c_4$. First, assume that $\frac 13<c_2<\frac 25$ and $\frac 35<c_4<\frac 23$. Then, $\frac 12(c_1+c_2)<\frac 15$ and $\frac 12(c_4+c_5)>\frac 45$ implying $c_1<\frac1{15}$ and $\frac {14}{15}<c_5$. Thus,
\begin{align*} V_5&\geq \int_{S_1(A_3)\uu J_{12}}(x-\frac 1{15})^2 dP+\int_{J_{21}\uu S_1(\set{a_1, a_2})}(x-\frac {14}{15})^2 dP=\frac{32531}{10347750}=0.00314378>V_5,
\end{align*}
which is a contradiction. Next, assume that $\frac 3{10} \leq c_2\leq \frac 13$ and $\frac 35<c_4<\frac 23$. Then, $\frac 12(c_1+c_2)<\frac 15$ and $\frac 12(c_4+c_5)>\frac 45$ implying $c_1<\frac1{10}$ and $\frac {14}{15}<c_5$. Moreover, $\frac 12(c_2+c_3)\geq \frac 25$ yielding $c_3\geq \frac 45-c_2\geq \frac 45-\frac 13=\frac 7{15}$, and so,
\begin{align*} V_5&\geq \int_{S_1(A_3)\uu J_{12}}(x-\frac 1{10})^2 dP+\frac 13(a_1-\frac 7{15})^2 \frac 12+\int_{J_{21}\uu S_1(\set{a_1, a_2})}(x-\frac {14}{15})^2 dP\\
&=\frac{3354097}{1076166000}=0.00311671>V_5,
\end{align*}
which leads to a contradiction. Now, assume that $\frac 3{10} \leq c_2\leq \frac 13$ and $\frac 23\leq c_4\leq\frac 7{10}$. Then, $c_1<\frac 1{10}$ and $c_5>\frac 9{10}$. Moreover, $\frac 12(c_2+c_3)\geq \frac 25$ and $\frac 12(c_3+c_4)\leq \frac 35$ implying $a_1<\frac 7{15}\leq c_3\leq \frac 8{15}<a_3$. Thus, we have
\begin{align*} V_5& \geq \int_{S_1(A_3)\uu J_{12}}(x-\frac 1{10})^2 dP+\frac 13\Big((a_1-\frac 7{15})^2 \frac 12+ \sum_{j=3}^\infty (a_j-\frac 8{15})^2 \frac 1{2^j}\Big)\\
& +\int_{J_{21}\uu S_1(\set{a_1, a_2})}(x-\frac {9}{10})^2 dP=\frac{2593}{1103760}=0.00234924>V_5,
\end{align*}
which is a contradiction. Last, assume that $c_2\leq \frac 3{10}$ and $\frac 7{10}\leq c_4$. Then, as $\frac 12(\frac 3{10}+\frac 7{10})=\frac 12$, we have
\[V_5\geq \frac 13\Big((a_1-\frac 3{10} )^2 \frac 12+ \sum_{j=3}^\infty (a_j-\frac 7{10})^2 \frac 1{2^j}\Big)=\frac{1}{315}=0.0031746>V_5,\]
which gives a contradiction.

Case~3. $c_2\in C$.

This case is similar to Case~1, and similarly can be shown that a contradiction arises.

Thus, considering all the above possible cases, we can say that $\te{card}(\ga\ii C)=1$ leads to a contradiction. Hence, we can assume that $\te{card}(\ga\ii C)=2$, i.e., the claim is true.

We now show that $\ga$ does not contain any point from $(\frac 15, \frac 25)\uu (\frac 25, \frac 35)$. Recall that $\ga\ii J_1\neq \es$, $\ga\ii J_2\neq \es$, and $\te{card}(\ga\ii C)=2$. Suppose that $\ga$ contains a point from $(\frac 15, \frac 25)$. Then, $\ga$ contains $c_2$ from $(\frac 15, \frac 25)$. First assume that $\frac 3{10}< c_2<\frac 25$. Then, $\frac 12(c_1+c_2)<\frac 15$ implying $c_1<\frac 1{10}$, and so
\[V_5\geq \int_{S_1(A_3)\uu J_{12}}(x-\frac 1{10})^2 dP+\int_{J_2}(x-S_2(\frac{19}{39}))^2 dP=\frac{25816591}{11658465000}=0.00221441>V_5,\]
which is a contradiction. Next, assume that $c_2\leq \frac 3{10}$. Then, $\frac 12(c_2+c_3)\geq \frac 2 5$ implying $c_3\geq \frac 45-c_2\geq \frac 45-\frac 3{10}=\frac 12=a_2$, and so
\[V_5\geq \frac 13(a_1-\frac 12)^2 \frac 12+\int_{J_2}(x-S_2(\frac{19}{39}))^2 dP=\frac{1470799}{466338600}=0.00315393>V_5,\]
which is a contradiction. Hence, $\ga$ cannot contain any point from $(\frac 15, \frac 25)$. Suppose that $\ga$ contains a point from $(\frac 35, \frac 45)$. Then, $\ga$ contains $c_4$ from $(\frac 35, \frac 4 5)$. First, assume that $\frac 35<c_4<\frac 7{10}$. Then, $\frac 12(c_4+c_5)>\frac 45$ implying $c_5>\frac 9{10}$, and so
\[V_5\geq  \int_{J_1}(x-S_1(\frac {19}{39}))^2 dP+ \int_{J_{21}\uu S_1(\set{a_1, a_2})}(x-\frac {9}{10})^2dP=\frac{26226559}{11658465000}=0.00224957>V_5,\]
which gives a contradiction. Next, assume that $\frac 7{10}\leq c_4$. Then, $\frac 12(c_3+c_4)\leq \frac 35$ implying $c_3\leq \frac 65-c_4\leq \frac 65-\frac 7{10}=\frac 12$, and so
\begin{align*}
V_5& \geq  \int_{J_1}(x-S_1(\frac {19}{39}))^2 dP+ \frac 13 \sum_{j=3}^\infty (a_j-\frac 12)^2 \frac 1{2^j}+\int_{J_2}\min_{c\in S_2(\ga_2)}(x-c)^2 dP\\
&=\frac 1{75} V+ \frac 13 \sum_{j=3}^\infty (a_j-\frac 12)^2 \frac 1{2^j}+\frac 1{75}(V_{2,1}+V_{2,2})=\frac{130788689}{58292325000}=0.00224367>V_5,
\end{align*}
which is a contradiction. Hence, $\ga$ cannot contain any point from $(\frac 35, \frac 45)$. Thus, the proof of the lemma is complete.
\end{proof}

\begin{lemma1}  \label{lemma00003} Let $\ga$ be an optimal set of six-means. Then, $\ga$ does not contain any point from the open intervals $(\frac 15, \frac 25)$ and $(\frac 35, \frac 45)$.
\end{lemma1}

\begin{proof}
Consider the set of six points $\gb:=S_1(\ga_2)\uu \ga_2(\gn)\uu S_2(\ga_2)$. The distortion error due to set $\gb$ is given by
\[\int\min_{c\in \gb} (x-c)^2 dP=\frac 2{75} (V_{2,1}+V_{2,2})+\frac 13 V_2(\gn)=\frac{13564657}{14573081250}=0.000930802.\]
Since $V_6$ is the quantization error for six-means, we have $V_6\leq 0.000930802$. Let $\ga:=\set{c_1<c_2<\cdots<c_6}$ be an optimal set of six-means.

First, prove the following claim.

Claim. $\te{card}(\ga\ii C)=2$.

Since $\ga\ii J_1\neq \es$ and $\ga\ii J_2\neq \es$, we have $\te{card}(\ga\ii C)\leq 4$. If $\te{card}(\ga\ii C)=4$, then
\[V_6\geq \int_{J_1}(x-S_1(\frac {19}{39}))^2 dP+\int_{J_2}(x-S_2(\frac{19}{39}))^2 dP=\frac 2{75} V=\frac{173392}{58292325}=0.00297453>V_6,\]
which is a contradiction. Assume that $\te{card}(\ga\ii C)=3$. Recall that $\ga\ii J_1\neq \es$, and $\ga\ii J_2\neq \es$. First, assume that $\ga$ does not contain any point from the open intervals $(\frac 15, \frac 25)$ and $(\frac 35, \frac 45)$. Then, either $\ga$ contains two points from $J_1$ and one point from $J_2$, or $\ga$ contains one point from $J_1$ and two points from $J_2$. In any case, we have the distortion error as
\[\frac 1{75}(V_{2,1}+V_{2,2})+\frac 13 V_3(\gn)+\frac 1{75}V=\frac{109198939}{58292325000}=0.0018733>V_6,\]
which is a contradiction. Next, assume that $\ga$ contains a point from the set $(\frac 15, \frac 25)\uu(\frac 35, \frac 45)$. First assume that $\ga$ contains the point from the open interval $(\frac 15, \frac 25)$. Then, $c_2\in (\frac 15, \frac 25)$. The following two cases can arise:

Case~1. $\frac 3{10}\leq c_2<\frac 25$.

Then, $\frac 12(c_1+c_2)<\frac 15$ implying $c_1<\frac 1{10}$, and so
\[V_6\geq \int_{S_1(A_3)\uu J_{12}}(x-\frac 1{10})^2 dP+\int_{J_2}(x-S_2(\frac{19}{39}))^2 dP=\frac{25816591}{11658465000}=0.00221441>V_6,\]
which is a contradiction.

Case~2. $\frac 15 <c_2\leq \frac 3{10}$.

Then, $\frac 12(c_2+c_3)>\frac 25$ implying $c_4>\frac 45-c_2\geq \frac 45-\frac 3{10}=\frac 12$, and so
\[V_6\geq \frac 13(a_1-\frac 12)^2 \frac 12+\int_{J_2}(x-S_2(\frac{19}{39}))^2 dP=\frac{1470799}{466338600}=0.00315393>V_6,\]
which gives a contradiction.

Similarly, we can show that if $\ga$ contains the point from the open interval $(\frac 35, \frac 45)$, a contradiction arises. Hence, we can assume that  $\te{card}(\ga\ii C)\leq 2$. Suppose that $\te{card}(\ga\ii C)=1$. Then, the following cases can arise:

Case~A. $\ga$ does not contain any point from the set $(\frac 15, \frac 25)\uu (\frac 35, \frac 45)$.

Then, $V_6\geq \frac 13 V_1(\gn)=\frac 13 V(\gn)=\frac{8}{4725}=0.00169312>V_6$, which leads to a contradiction.

Case~B. $\ga$ contains only one point from the set $(\frac 15, \frac 25)\uu (\frac 35, \frac 45)$.

In this case the following subcases can arise:

Subcase~(i).  $c_1, c_2, c_3 \in J_1$.

Then, $c_4 \in (\frac 15, \frac 25)$, $c_5\in C$, and $c_6\in J_2$; or $c_4\in C$, $c_5\in (\frac 35, \frac 45)$, and $c_6\in J_2$. First, assume that $c_4 \in (\frac 15, \frac 25)$, $c_5\in C$, and $c_6\in J_2$. Suppose that $\frac 3{10}<c_4<\frac 25$. Then, $\frac 12(c_3+c_4)<\frac 15$ implying $c_3<\frac 25-c_4\leq \frac 1{10}$, and hence
\[V_6\geq \int_{S_1(A_3)\uu J_{12}}(x-\frac 1{10})^2 dP+\int_{J_2}(x-S_2(\frac{19}{39}))^2 dP=\frac{25816591}{11658465000}=0.00221441>V_6,\]
which is a contradiction. Suppose that $\frac 15<c_4\leq \frac3{10}$. Then, $\frac 12(c_4+c_5)\geq \frac 25$ implying $c_5\geq \frac 12$, and so
\[V_6\geq \frac 13 (a_1-\frac 12)^2 \frac 12+\int_{J_2}(x-S_2(\frac{19}{39}))^2 dP=\frac{1470799}{466338600}=0.00315393>V_6,\]
which yields a contradiction. Similarly, we can show that if $c_4\in C$, $c_5\in (\frac 35, \frac 45)$, and $c_6\in J_2$, a contradiction arises.

Subcase~(ii).  $c_1, c_2,\in J_1$.

Then, $c_3\in (\frac 15, \frac 25)$, $c_4\in C$, $c_5, c_6\in J_2$; or $c_3\in C$, $c_4\in (\frac 35, \frac 45)$, and $c_5, c_6\in J_2$. First, assume that $c_3\in (\frac 15, \frac 25)$, $c_4\in C$, $c_5, c_6\in J_2$. Suppose that $\frac 3{10}<c_3<\frac 25$. Then, $\frac 12(c_2+c_3)<\frac 15$ implying $c_2<\frac 25-c_3\leq \frac 1{10}$, and hence
\[V_6\geq \int_{S_1(A_3)\uu J_{12}}(x-\frac 1{10})^2 dP+\int_{J_2}\min_{c\in S_2(\ga_2)}(x-c)^2 dP=\frac{7038641}{6476925000}=0.00108673>V_6,\]
which is a contradiction. Suppose that $\frac 15<c_3\leq \frac3{10}$. Then, $\frac 12(c_3+c_4)\geq \frac 25$ implying $c_4\geq \frac 12$, and so
\[V_6\geq \frac 13 (a_1-\frac 12)^2 \frac 12+\int_{J_2}\min_{c\in S_2(\ga_2)}(x-c)^2 dP=\frac{5624509}{2775825000}=0.00202625>V_6,\]
which yields a contradiction. Similarly, we can show that if $c_3\in C$, $c_4\in (\frac 35, \frac 45)$, and $c_5, c_6\in J_2$, a contradiction arises.

Subcase~(iii).  $c_1 \in J_1$.

This case is the reflection of Subcase~(i) about the point $\frac 12$, and thus a contradiction also arises in this case.

Case~C. $\ga$ contains two points from the set $(\frac 15, \frac 25)\uu (\frac 35, \frac 45)$.

By Proposition~\ref{prop1000}, $\ga$ contains one point from $(\frac 15, \frac 25)$ and one point from $(\frac 35, \frac 45)$. Without any loss of generality, we can assume that $c_1, c_2\in J_1$, $c_3\in (\frac 15, \frac 25)$, $c_4\in C$, $c_5\in (\frac 35, \frac 45)$, and $c_6\in J_6$. Then, we can also show that it leads to a contradiction.

Thus, we can assume that $\te{card}(\ga\ii C)=2$, which is the claim.

We now show that $\ga$ does not contain any point from $(\frac 15, \frac 25)\uu (\frac 35, \frac 45)$. For the sake of contradiction, first assume that $\ga$ contains a point from $(\frac 15, \frac 25)$, and does not contain any point from $(\frac 35, \frac 45)$. Suppose that $c_1, c_2 \in J_1$, $c_3\in (\frac 15, \frac 25)$, $c_4, c_5\in C$, and $c_6\in J_2$, and then
\[V_6\geq \int_{J_2}(x-S_2(\frac {19}{39}))^2 dP=\frac{86696}{58292325}=0.00148726>V_6,\]
which is a contradiction. Suppose that $c_1\in J_1$, $c_2\in (\frac 15, \frac 25)$, $c_3, c_4\in C$, and $c_5, c_6\in J_2$. Then, either $\frac 3{10} \leq c_2<\frac 25$, or $\frac 15<c_2\leq \frac 3{10}$. First, assume that $\frac 3{10} \leq c_2<\frac 25$. Then, $\frac 12(c_{1}+c_2)<\frac 15$ implies $c_1\leq \frac 1{10}$, and so
\[V_6\geq  \int_{S_1(A_3)\uu J_{12}}(x-\frac 1{10})^2 dP  + \int_{J_2}\min_{c\in S_2(\ga_2)}(x-c)^2 dP=\frac{7038641}{6476925000}=0.00108673>V_6,\]
which leads to a contradiction. Next, assume that $\frac 15<c_2\leq \frac 3{10}$. Then, $\frac 12(c_2+c_3)>\frac 25$ implying $c_3>\frac 45-c_2\geq \frac 45-\frac 3{10}=\frac 12>a_1$, and so
\[V_6\geq \frac 13(a_1-\frac 13)^2 \frac 1{2}+ \int_{J_2}\min_{c\in S_2(\ga_2)}(x-c)^2 dP=\frac{5624509}{2775825000}=0.00202625>V_6,\]
which is a contradiction.
Likewise, we can show that if $\ga$ contains a point from $(\frac 35, \frac 45)$, and does not contain any point from $(\frac 15, \frac 25)$, a contradiction arises.

Next, assume that $\ga$ contains a point from $(\frac 15, \frac 25)$ and a point from $(\frac 35, \frac 45)$. Then, we have $c_1\in J_1$, $c_2\in (\frac 15, \frac 25)$, $c_3, c_4\in C$, $c_5\in (\frac 35, \frac 45)$, and $c_6\in J_2$.
Then, the following cases can arise:

Case~I. $\frac 3{10} \leq c_2<\frac 25$ and $\frac 3 5 <c_5\leq \frac 7{10}$.

Then, $\frac 12(c_{1}+c_2)<\frac 15$ and $\frac 12(c_5+c_6)>\frac 45$ implying $c_1<\frac 1{10}$ and $c_6>\frac 9{10}$. Thus,

\[V_6\geq  \int_{S_1(A_3)\uu J_{12}}(x-\frac 1{10})^2 dP  + \int_{J_{21}\uu S_2(\set{a_1, a_2})}(x-\frac 9{10})^2 dP=\frac{137}{91980}=0.00148945>V_6,\]
which leads to a contradiction.

Case~II. $\frac 15 <c_2\leq \frac 3{10}$ and  $\frac 3 5 <c_5\leq \frac 7{10}$.

Then, $\frac 12(c_2+c_3)>\frac 25$ implying $c_3>\frac 45-c_2\geq \frac 45-\frac 3{10}=\frac 12>a_1$, and $c_6>\frac 9{10}$. Thus,
\[V_6\geq \frac 13(a_1-\frac 13)^2 \frac 1{2}+ \int_{J_{21}\uu S_2(\set{a_1, a_2})}(x-\frac 9{10})^2 dP=\frac{363053}{149467500}=0.00242898>V_6,\]
which is a contradiction.

Case~III. $\frac 15 <c_2\leq \frac 3{10}$ and $\frac 7{10} \leq c_5<\frac 45$.

Then, $\frac 12(c_2+c_3)\geq \frac 25$ and $\frac 12(c_4+c_5)\leq \frac 35$ implying $c_3\geq \frac 12$ and $c_4\leq \frac 12$, which leads to a contradiction because $c_3<c_4$.

Thus, we can conclude that $\ga$ does not contain any point from the open intervals $(\frac 15, \frac 25)$ and $(\frac 35, \frac 45)$, which is the lemma.
\end{proof}

\begin{propo}  \label{prop00004} Let $\ga_n$ be an optimal set of $n$-means for all $n\geq 3$. Then, $\ga_n$ does not contain any point from the open intervals $(\frac 15, \frac 25)$ and $(\frac 35, \frac 45)$.
\end{propo}

\begin{proof}
By Proposition~\ref{prop00002}, Lemma~\ref{lemma00001}, Lemma~\ref{lemma00002}, and Lemma~\ref{lemma00003}, the proposition is true for $n=3, 4, 5, 6$. Proceeding in the similar way as Lemma~\ref{lemma00003}, we can show that the proposition is true for $n=7$. We now show that the proposition is true for all $n\geq 8$. Consider the set of eight points $\gb:=S_1(\ga_3)\uu \ga_2(\gn)\uu S_2(\ga_3)$. The distortion error due to set $\gb$ is given by
\begin{align*}
& \int\min_{c\in \gb} (x-c)^2 dP=\frac 1{75^2}V+\frac 13 \frac 1{75}V(\gn)+ \frac 1{75^2}V +\frac 13V_2(\gn)+\frac 1{75^2}V+\frac 13 \frac 1{75}V(\gn)+ \frac 1{75^2}V \\
&=\frac{489817}{1457308125}=0.000336111.
\end{align*}
Since $V_n$ is the quantization error for $n$-means with $n\geq 8$, we have $V_n\leq V_8\leq 0.000633457$. Let $\ga:=\set{c_1<c_2<\cdots<c_n}$ be an optimal set of $n$-means for $P$. Suppose that $\ga$ contains a point from the open interval $(\frac 15, \frac 25)$. Let $j:=\max\set{i : c_i<\frac 25}$.
Then, $\frac 15<c_j<\frac 25$. The following cases can arise:

Case~1. $\frac 3{10} <c_j<\frac 25$.

Then, $\frac 12(c_{j-1}+c_j)<\frac 15$ implying $c_{j-1}<\frac 25-\frac 3{10}=\frac 1{10}$, and so
\[V_n\geq  \int_{S_1(A_3)\uu J_{12}}(x-\frac 1{10})^2 dP =\frac{217369}{298935000}=0.000727145>V_n,\]
which leads to a contradiction.

Case~2. $\frac 15 \leq c_j\leq \frac 3{10}$.

Then, $\frac 12(c_j+c_{j+1})>\frac 25$ implying $c_{j+1}>\frac 45-c_j\geq \frac 45-\frac 3{10}=\frac 12$, and so
\[V_n\geq \frac 13(a_1-\frac 12)^2 \frac 12=\frac{1}{600}=0.00166667>V_n,\]
which gives a contradiction.

Hence, we can say that $\ga$ does not contain any point from $(\frac 25, \frac 35)$. Suppose that $\ga$ contains a point from $(\frac 35, \frac 45)$. Let $j:=\min\set{i : \frac 3 5<c_i}$. Then, $\frac 35<c_j<\frac 45$. The following cases can arise:

Case~I. $\frac 3{5} <c_j<\frac 7{10}$.

Then, $\frac 12(c_{j}+c_{j+1})>\frac 45$ implying $c_{j+1}>\frac 85-\frac 7{10}=\frac 9{10}$, and so
\[V_n\geq  \int_{J_{21}\uu S_2(\set{a_1, a_2})}(x-\frac 9{10})^2 dP =\frac{227881}{298935000}=0.00076231>V_n,\]
which leads to a contradiction.

Case~II. $\frac 7{10} \leq c_j<\frac 45$.

Then, $\frac 12(c_{j-1}+c_{j})<\frac 3 5$ implying $c_{j-1}<\frac 65-c_j\leq \frac 6 5-\frac 7{10}=\frac 12<a_2$, and so
\[V_n\geq \frac 13\sum_{j=2}^\infty(a_j-\frac 12)^2\frac 1{2^j}=\frac{1}{2520}=0.000396825>V_n,\]
which is a contradiction.

Hence $\ga$ does not contain any point from the open interval $(\frac 35, \frac 45)$. This completes the proof of the proposition.
\end{proof}

\subsection{Optimal sets and the quantization error for a given sequence $F(n)$}
In this subsection we first define the two sequences $\set{a(n)}_{n\geq 1}$ and  $\set{F(n)}_{n\geq 1}$. These two sequences play important role in this section.
\begin{defi} \label{defi031}  Define the sequence $\set{a(n)}_{n\geq 1}$ such that $a(n)=2n$ for all $n\geq 1$.
Define the sequence $\set{F(n)}_{n\geq 1}$ such that $F(n)=5\cdot 2^n-2n-4$, i.e.,
\[\set{F(n)}_{n\geq 1}=\{4,12,30,68,146,304,622,1260,2538,5096,10214,20452,40930, \cdots\}.\]
\end{defi}

\begin{note1} \label{note023}
Let $\set{a(n)}_{n\geq 1}$ and $\set{F(n)}_{n\geq 1}$ be the sequences defined by Definition~\ref{defi031}. Then, notice that $F(n+1)=a(n+1)+2 F(n)$. For $n=1$, we set
$\ga_{F(1)}:=\ga_{a(1)}(\gn)\uu \set{S_1(\frac{19}{39}), S_2(\frac {19}{39})}$, and for $n\geq 2$, we set
$\ga_{F(n)}:=\ga_{a(n)}(\gn)\uu S_1(\ga_{F(n-1)})\uu S_2(\ga_{F(n-1)}).$ Then, inductively, we have
\begin{align} \label{eq404}
\ga_{F(n)}& =\ga_{a(n)}\uu \Big(\uu_{\go\in I}S_\go(\ga_{a(n-1)}(\gn))\Big)\uu \Big(\uu_{\go\in I^2} S_\go(\ga_{a(n-2)}(\gn))\Big)\\
& \cdots \uu \Big(\uu_{\go\in I^{n-1}}S_\go(\ga_{a(1)}(\gn))\Big)\uu \set{S_\go (\frac {19}{39}) : \go \in I^n}.\notag
\end{align}
Notice that the set $\ga_{F(n)}$ satisfies: $\ga_{F(n+1)}=S_1(\ga_{F(n)})\uu \ga_{a(n+1)}(\gn)\uu \ga_{F(n)}$.  We will show that $\ga_{F(n)}$ is an optimal set of $F(n)$-means for $P$.
For $n\in \D N$, we identify the sequence of sets
$ \ga_{a(n)}(\gn), \, \uu_{\go \in I} S_\go(\ga_{a(n-1)}(\gn)), \,  \uu_{\go \in I^2} S_\go(\ga_{a(n-2)}(\gn)), \, \cdots, \, \uu_{ \go \in I^{n-2}} S_\go(\ga_{a(2)}(\gn)), \, \uu_{\go \in I^{n-1}} S_\go(\ga_{a(1)}(\gn))$, and
$ \set{S_\go(\frac {19}{39}) : \go \in I^{n}}$, respectively, by $S(n)$, $S(n-1)$, $S(n-2)$,  $\cdots$, $S(2)$, $S(1)$, and $S(n+1)$. For $1\leq \ell\leq n$, write
\[S^{(2)}(\ell):=\uu_{\go \in I^{n-\ell}} S_\go(\ga_{a(\ell)+1}(\gn)) \te{ and } S^{(2)(2)}(\ell):=\uu_{\go \in I^{n-\ell}} S_\go(\ga_{a(\ell)+2}(\gn))=\uu_{\go \in I^{n-\ell}} S_\go(\ga_{a(\ell+1)}(\gn)). \]
Further, we write
\begin{align*} S^{(2)}(n+1)& :=\uu_{\go \in I^n} S_\go(\ga_2(P))=\uu_{\go \in I^n} \set{S_\go(\frac {659}{2730}), S_\go(\frac{1621}{1950})}, \\
S^{(2)(2)}(n+1)&:=\uu_{\go \in I^{n}} S_\go(\ga_{1}(\gn))\uu \set{S_\go(\frac {19}{39}) : \go \in I^{n+1}}, \te{ and }\\
S^{(2)(2)(2)}(n+1)&:=\uu_{\go \in I^{n}} S_\go(\ga_{a(1)}(\gn))\uu \set{S_\go(\frac {19}{39}) : \go \in I^{n+1}}.
 \end{align*}
 Moreover, for any $\ell\in \D N$, if $A:=S(\ell)$, we identify $S^{(2)}(\ell)$, $S^{(2)(2)}(\ell)$, and $S^{(2)(2)(2)}(\ell)$,  respectively, by $A^{(2)}$, $A^{(2)(2)}$, and  $A^{(2)(2)(2)}$. For $n\in\D N$, set
 \[SF(n):=\set{S(n), S(n-1), S(n-2), \cdots, S(1), S(n+1)}.\]
In addition, write
\begin{align*} \label{eq35} SF^\ast(n): =&\set{S(n), S(n-1), S(n-2),\cdots, S(1), S(n+1), S^{(2)}(n),  S^{(2)}(n-1), S^{2)}(n-2), \cdots,\\
&  S^{(2)}(1), S^{(2)}(n+1),  S^{(2)(2)}(n), S^{(2)(2)}(n-1), \cdots,  S^{(2)(2)}(1), S^{(2)(2)}(n+1),\\
&  S^{(2)(2)(2)}(n+1)}.
\end{align*}
For any element $a\in A \in SF^\ast(n)$, by the Voronoi region of $a$ it is meant the Voronoi region of $a$ with respect to the set $\uu_{B\in SF^\ast(n)} B$.
Similarly, for any $a\in A\in SF(n)$, by the Voronoi region of $a$ it is meant the Voronoi region of $a$ with respect to the set $\uu_{B\in SF(n)} B$. Notice that if $a, b\in A$, where $A \in SF(n)$ or $A\in SF^\ast(n)$ except the sets $S^{(2)}(n+1)$, $S^{(2)(2)}(n+1)$, and $S^{(2)(2)(2)}(n+1)$, the error contributed by $a$ in the Voronoi region of $a$ equals to the error contributed by $b$ in the Voronoi region of $b$. On the other hand, each of $S^{(2)}(n+1)$, $S^{(2)(2)}(n+1)$, and $S^{(2)(2)(2)}(n+1)$ contains two different sets: the error in the Voronoi region of an element in one set is larger than the error in the Voronoi region of an element in the other set.
Let us now define an order $\succ $ on the set $SF^\ast(n)$ as follows: For $A, B\in SF^\ast(n)$ by $A\succ B$ it is meant that the error contributed by any element $a\in A$ in the Voronoi region of $a$ is greater or equal to the error contributed by any element $b \in B$ in the Voronoi region of $b$. Similarly, we define the order relation $\succ $ on the set $SF(n)$.
\end{note1}

\begin{remark1} Let $\ga_{F(n)}$ be the set defined by \eqref{eq404}. Then,
\[\ga_{F(n+1)}=S^{(2)(2)}(1)\uu  S^{(2)(2)}(2)\uu   S^{(2)(2)}(3)\uu \cdots\uu S^{(2)(2)}(n)\uu S^{(2)(2)(2)}(n+1).\]
\end{remark1}

\begin{remark1} \label{remark00} Take any $1\leq \ell  <n$.
The distortion errors due to any element in the sets $S^{(2)(2)}(\ell+1)$ and $S^{(2)(2)}(\ell)$ are, respectively,  $\frac 13 \frac 1{75^{n-(\ell+1)}} V_{a(\ell+2)}(\gn)=\frac 13 \frac 1{75^{n-(\ell+1)}}\frac{2^{-6 -6\ell}}{1575}$ and  $\frac 13 \frac 1{75^{n-\ell}} V_{a(\ell+1)}(\gn)=\frac 13 \frac 1{75^{n-\ell}}\frac{2^{-6\ell}}{1575}$. Notice that
\[\frac 13 \frac 1{75^{n-(\ell+1)}}\frac{2^{-6 -6\ell}}{1575} >\frac 13 \frac 1{75^{n-\ell}}\frac{2^{-6\ell}}{1575},\]
and so, $S^{(2)(2)}(n)\succ S^{(2)(2)}(n-1)\succ S^{(2)(2)}(n-2)\succ \cdots \succ S^{(2)(2)}(2)\succ S^{(2)(2)}(1)$.
Similarly, $S^{(2)}(n)\succ S^{(2)}(n-1)\succ S^{(2)}(n-2)\succ \cdots \succ S^{(2)}(2)\succ S^{(2)}(1)$, and  $S(n)\succ S(n-1)\succ S(n-2)\succ \cdots \succ S(2)\succ S(1)$.
\end{remark1}

\begin{lemma1}\label{lemma99} Let $n$ be a large positive integer, in fact, $n\geq 40$. Then, the following inequalities are true:

$(i)$ If $2<n-k<n-k+26\leq n$, then $S(n-k)\succ S^{(2)}(n-k+13)\succ S^{(2)(2)}(n-k+26)\succ S(n-k-1)$,

$(ii)$ $S(14)\succ S^{(2)}(27)\succ S^{(2)(2)}(40)\succ S(n+1)\succ S(13)$.

$(iii)$ $S(5)\succ S^{(2)}(18)\succ S^{(2)(2)}(31)\succ S^{(2)} (n+1)\succ S(4)$.

$(iv)$ $S^{(2)(2)}(13)\succ S^{(2)(2)} (n+1)\succ S^{(2)(2)(2)} (n+1)\succ S^{(2)(2)}(12)$.

\end{lemma1}

\begin{proof} Let $2<n-k<n-k+26\leq n$. Then, $S(n-k)\succ S^{(2)}(n-k+13)\succ S^{(2)(2)}(n-k+26)\succ S(n-k-1)$ will be true if
\[\frac 1 3\frac 1 {75^{k}} V_{a(n-k)}(\gn)>\frac 1 3\frac 1 {75^{k-13}} V_{a(n-k+13)+1}(\gn)>\frac 1 3\frac 1 {75^{k-26}} V_{a(n-k+26)+2}(\gn)>\frac 1 3\frac 1 {75^{k+1}} V_{a(n-k-1)}(\gn)\]
i.e., if
\begin{align} \label{eq5555} V_{a(n-k)}(\gn)>75^{13} V_{a(n-k+13)+1}(\gn)>75^{26} V_{a(n-k+26)+2}(\gn)>\frac 1{75}   V_{a(n-k-1)}(\gn).\end{align}
By putting the values of $V_{a(n-k)}(\gn)$, $V_{a(n-k+13)+1}(\gn)$, $V_{a(n-k+26)+2}(\gn)$, and $V_{a(n-k-1)}(\gn)$, we see that the inequalities in \eqref{eq5555} are clearly true. Thus, the proof of $(i)$ is complete.

By $(i)$ we know that $S(14)\succ S^{(2)}(27)\succ S^{(2)(2)}(40)$. Thus, to prove $(ii)$ it is enough to prove that $S^{(2)(2)}(40)\succ S(n+1)\succ S(13)$, which is true if \begin{align} \label{eq55551} \frac 1 3\frac 1 {75^{n-40}} V_{a(40)+2}(\gn)>\frac 1 {75^{n}} V>\frac 1 3\frac 1 {75^{n-13}} V_{a(13)}(\gn).\end{align}
By putting the values of  $V_{a(40)+2}(\gn)$, $V$, and $V_{a(13)}(\gn)$, we see that the inequalities in \eqref{eq55551} are clearly true. Thus, the proof of $(ii)$ is complete. Similarly, we can prove the inequalities in $(iii)$ and $(iv)$. Thus, the lemma is yielded.
\end{proof}

\begin{remark1} \label{remark5656}  By Remark~\ref{remark00} and Lemma~\ref{lemma99}, for any large positive integer $n$, in fact if $n\geq 40$, it follows that
\begin{align*}& S(n)\succ S(n-1)\succ \cdots\succ S(n-13)\succ S^{(2)}(n)\succ S(n-14)\succ S^{(2)}(n-1)\\
& \succ S(n-15)\succ S^{(2)}(n-2)\succ  S(n-16) \succ \cdots \\
&\succ S(n-26)\succ S^{(2)}(n-13)\succ S^{(2)(2)}(n)\succ S(n-27)\succ \cdots \\
&S(14)\succ  S^{(2)}(27)\succ S^{(2)(2)}(40)\succ S(n+1)\succ S(13)\succ  S^{(2)}(26)\succ S^{(2)(2)}(39)\succ S(12)\succ \cdots \\
&S(5)\succ  S^{(2)}(18)\succ S^{(2)(2)}(31)\succ S^{(2)}(n+1)\succ S(4)\succ S^{(2)}(17)\succ S^{(2)(2)}(30)\succ S(3)\succ \cdots \\
&S(1)\succ  S^{(2)}(14)\succ S^{(2)(2)}(27)\succ S^{(2)}(13)\succ   S^{(2)(2)}(26)\succ S^{(2)}(12)\succ S^{(2)(2)}(25)\succ \cdots \\
&S^{(2)}(1)\succ S^{(2)(2)}(14)\succ S^{(2)(2)}(13)\succ  S^{(2)(2)}(n+1)\succ S^{(2)(2)(2)}(n+1)\succ S^{(2)(2)}(12)\\
& \succ S^{(2)(2)}(11)\succ S^{(2)(2)}(10) \succ \cdots S^{(2)(2)}(1).
\end{align*}
\end{remark1}

\begin{lemma1}\label{lemma0711}  For any two sets $A, B\in SF (n)$, let $A\succ B$. Then, the distortion error due to the set $ (SF(n)\setminus A)\uu A^{(2)}\uu B$ is less than the distortion error due to the set $ (SF(n)\setminus B)\uu B^{(2)}\uu A$.
\end{lemma1}

\begin{proof}
Let $V(\ga_{F(n)})$ be the distortion error due to the set $\ga_{F(n)}$ with respect to the condensation measure $P$. First take $A=S(k)$ and $B=S(k')$ for some $1\leq k, k'\leq n$. Then, $A\succ B$ implies $k>k'$. The distortion error due to the set $(\ga_{F(n)}\setminus A)\uu A^{(2)}\uu B$ is given by
\begin{align}\label{eq045}
&V(\ga_{F(n)})-\frac 1 3(\frac 2{75})^{n-k}  V_{a(k)}(\gn)+\frac 13 (\frac 2{75})^{n-k} V_{a(k)+1}(\gn)+\frac 1 3(\frac 2{75})^{n-k'} V_{a(k')}(\gn) \\
&=V(\ga_{F(n)})-\frac {56} 3 (\frac 2{75})^{n-k}  \frac{2^{-6k}}{1575}+\frac {64}3 (\frac 2{75})^{n-k'}  \frac{2^{-6k'}}{1575}. \notag
\end{align}
Similarly, the distortion error due to the set $(\ga_{F(n)}\setminus B)\uu B^{(2)}\uu A$ is
\begin{align} \label{eq046}
V(\ga_{F(n)})-\frac {56} 3 (\frac 2{75})^{n-k'}  \frac{2^{-6k'}}{1575}+\frac {64}3 (\frac 2{75})^{n-k}  \frac{2^{-6k}}{1575}
\end{align}
Thus, \eqref{eq045} will be less than \eqref{eq046} if $\Big(\frac{75}{128}\Big)^{k'}<\Big(\frac{75}{128}\Big)^{k}$, which is clearly true since $k>k'$.
As mentioned in Remark~\ref{remark5656}, assume that $S(14)\succ S(n+1)\succ S(13)$. Take $A=S(14)$, and $B=S(13)$. Here $A\succ B$.  Proceeding as before, it can also be seen that  the distortion error due to the set  $(\ga_{F(n)}\setminus A)\uu A^{(2)}\uu B$ is less than the distortion error due to the set  $(\ga_{F(n)}\setminus B)\uu B^{(2)}\uu A$. Similarly, we can show it for any two sets $A, B\in SF(n)$ with $A\succ B$. Thus, the lemma is yielded.
\end{proof}

Let us now state the following lemma. The proof is similar to the proof of Lemma~\ref{lemma0711}.
\begin{lemma1}\label{lemma0721}  For any two sets $A, B\in SF^\ast (n)$, let $A\succ B$. Then, the distortion error due to the set $ (SF^\ast(n)\setminus A)\uu A^{(2)}\uu B$ is less than the distortion error due to the set $ (SF^\ast(n)\setminus B)\uu B^{(2)}\uu A$.
\end{lemma1}

\begin{propo} \label{prop0004} For $n\geq 1$, let $\ga_{F(n)}$ be the set defined by Note~\ref{note023}. Then, $\ga_{F(n)}$ forms an optimal set of $F(n)$-means for $P$ and the corresponding quantization error is given by
\[V_{F(n)}=\frac{769208}{5884749}(\frac 2{75})^n-\frac {64}{3339}(\frac 1{64})^n\]
\end{propo}

\begin{proof} We see that $\ga_{F(1)}=\ga_{a(1)}(\gn)\uu \set{S_\go (\frac{19}{39}) : \go \in I}$, which by Proposition~\ref{prop000021}, is an optimal set of four-means for $P$. Thus, the proposition is true for $n=1$. Let $\ga_{F(n)}$ be an optimal set of $F(n)$-means for some $n\geq 1$. We now show that $\ga_{F(n+1)}$ is an optimal set of $F(n+1)$-means. We have
$\ga_{F(n)}=\uu_{A\in SF(n)}A$. Recall that, by Proposition~\ref{prop00004}, an optimal set of $n$-means for any $n\geq 3$ does not contain any point from the open intervals $(\frac 15, \frac 25)$ and $(\frac 35, \frac 45)$. In the first step, let $A(1) \in SF(n)$ be such that $A(1)\succ B$ for any other $B\in SF(n)$. Then, by Lemma~\ref{lemma0711},
the set  $(\ga_{F(n)}\setminus A(1)) \uu A^{(2)}(1)$ gives an optimal set of $F(n)-\te{card}(A(1))+\te{card}(A^{(2)}(1))$-means. In the 2nd step, let $A(2) \in (SF(n)\setminus \set{A(1)})\uu \set{A^{(2)}(1)}$ be such that $A(2)\succ B$ for any other set $B\in  (SF(n)\setminus \set{A(1)})\uu \set{A^{(2)}(1)}$.
Then, using the similar technique as Lemma~\ref{lemma0711}, we can show that the distortion error due to the following set:
\begin{equation} \label{eq044} \Big(((\ga_{F(n)}\setminus A(1)) \uu A^{(2)}(1))\setminus A(2)\Big)\uu A^{(2)}(2)
\end{equation}
 with cardinality $F(n)-\te{card}(A(1))+\te{card}(A^{(2)}(1))-\te{card}(A(2))+\te{card}(A^{(2)}(2))$ is smaller than the distortion error due to the set obtained by replacing $A(2)$ in the set \eqref{eq044} by the set $B$. In other words, $\Big(((\ga_{F(n)}\setminus A(1)) \uu A^{(2)}(1))\setminus A(2)\Big)\uu A^{(2)}(2)$ forms an optimal set of
$F(n)-\te{card}(A(1))+\te{card}(A^{(2)}(1))-\te{card}(A(2))+\te{card}(A^{(2)}(2))$-means. Proceeding inductively in this way, we can see that the set $\ga_{F(n+1)}=S^{(2)(2)}(1)\uu  S^{(2)(2)}(2)\uu   S^{(2)(2)}(3)\uu \cdots\uu S^{(2)(2)}(n)\uu S^{(2)(2)(2)}(n+1)$ forms an optimal set of $F(n+1)$-means.
Thus, by the induction principle, we can say that for any $n\geq 1$, the set $\ga_{F(n)}$ forms an optimal set of $F(n)$-means for the condensation measure $P$.  The corresponding quantization error is given by
\begin{align*}
&V_{F(n)}=\frac 1 3\Big(V_{a(n)}(\gn)+ (\frac 2{75})V_{a(n-1)}(\gn)+ (\frac 2{75})^2V_{a(n-2)}(\gn)+\cdots+ (\frac 2{75})^{n-1}V_{a(1)}(\gn)\Big)+ (\frac 2{75})^n V\\
&=\frac 1 3 \frac {1}{1575} \Big(\frac {1}{64^{n-1}}+ (\frac 2{75})\frac {1}{64^{n-2}}+ (\frac 2{75})^2\frac {1}{64^{n-3}}+\cdots+ (\frac 2{75})^{n-2}\frac {1}{64}+(\frac 2{75})^{n-1}\Big)+ (\frac 2{75})^n V,
\end{align*}
which after simplification implies that
\[V_{F(n)}=\frac {64}{3339} \Big((\frac 2{75})^n-(\frac 1{64})^n\Big)+(\frac 2{75})^nV=\frac{769208}{5884749}(\frac 2{75})^n-\frac {64}{3339}(\frac 1{64})^n.\]
Thus the proof of the proposition is complete.
\end{proof}

We now prove the following proposition.

\begin{propo} \label{prop0005} Let $n\in \D N$ be large. Then, the set
\begin{equation} \label{eq5655} \Big(\ga_{F(n)}\setminus \Big(\mathop{\uu}\limits_{k=18}^{n}S(k)\uu S(n+1)\Big)\Big){\uu}\Big(\mathop{\uu}\limits_{k=31}^{n}S^{(2)(2)}(k) \Big)\uu \Big(\mathop{\uu}\limits_{k=18}^{30} S^{(2)}(k)\Big)\uu S^{(2)}(n+1)
 \end{equation} forms an optimal set of $\Big(2^n\Big(6+2^{-30}(2^{13}+1)\Big)-2n -6\Big)$-means for $P$ and the corresponding quantization error is given by
\begin{align*} &V_{2^n(6+2^{-30}(2^{13}+1))-2n -6}\\
&=V_{F(n)}-\frac{128}{3975}(\frac 2{75})^n\Big((\frac {75}{128})^{31}-(\frac {75}{128})^{n+1}\Big)-\frac{1024}{35775}(\frac 2{75})^n\Big((\frac {75}{128})^{18}-(\frac {75}{128})^{31}\Big)\\
&\qquad \qquad -\frac{450241}{5323500} (\frac 2{75})^n.
\end{align*}
\end{propo}

\begin{proof} Due to Remark~\ref{remark5656}, Lemma~\ref{lemma0721}, using the similar technique as Proposition~\ref{prop0004}, we can show that the set
given by \eqref{eq5655} is an optimal set. Recall the definitions of $S(k)$, $S^{(2)}(k)$, $S^{(2)(2)}(k)$, $S(n+1)$, and $S^{(2)}(n+1)$. Thus, for $31\leq k\leq n$, in a set  when $S(k)$ is replaced by $S^{(2)(2)}(k)$, we are removing $2^{n-k}\cdot a(k)$ elements, and adding $2^{n-k}\cdot (a(k)+2)$ elements to the set leading to the number of elements increased by $2^{n-k}\cdot 2$ and the quantization error is decreased by $\frac 13 (\frac 2{75})^{n-k} (V_{a(k)}(\gn)-V_{a(k)+2}(\gn))$. For $18\leq k\leq 30$, when $S(k)$ is replaced by $S^{(2)}(k)$ in a set, we are removing $2^{n-k} a(k)$ elements, and adding $2^{n-k}\cdot (a(k)+1)$ elements to the set leading to the number of elements in the set increased by $2^{n-k}$, and the quantization error is decreased by $\frac 13 (\frac 2{75})^{n-k} (V_{a(k)}(\gn)-V_{a(k)+1}(\gn))$.  When  $S(n+1)$ is replaced by $S^{(2)}(n+1)$, the number of elements in the set is increased by $2^n$, and the quantization error is decreased by $(V-(V_{2,1}+V_{2,2}))$. Thus, the cardinality of the set  \eqref{eq5655} is given by
\begin{align*}
&F(n)+2\sum_{k=31}^n 2^{n-k}+\sum_{k=18}^{30} 2^{n-k}+2^n=F(n)+2 ( 2^{n-30}-1)+2^{n-30}(2^{13}-1)+2^n\\
&=(5\cdot 2^n-2n-4)+2^{n-30}(2^{13}+1)+2^n-2=2^n\Big(6+2^{-30}(2^{13}+1)\Big)-2n -6.
\end{align*}
The corresponding quantization error is
\begin{align*}
&V_{F(n)}-\frac 1 3 \sum_{k=31}^n(\frac 2{75})^{n-k} (V_{a(k)}(\gn)-V_{a(k)+2}(\gn))-\frac 1 3 \sum_{k=18}^{30}(\frac 2{75})^{n-k} (V_{a(k)}(\gn)-V_{a(k)+1}(\gn))\\
&\qquad \qquad -(\frac 2{75})^n(V-(V_{2,1}+V_{2,2}))\\
&=V_{F(n)}-\frac 1{75} (\frac 2{75})^n\sum_{k=31}^n(\frac {75}{128})^k -\frac{8}{675}(\frac 2{75})^n\sum_{k=18}^{30} (\frac {75}{128})^k-\frac{450241}{5323500} (\frac 2{75})^n\\
&=V_{F(n)}-\frac{128}{3975}(\frac 2{75})^n\Big((\frac {75}{128})^{31}-(\frac {75}{128})^{n+1}\Big) -\frac{1024}{35775}(\frac 2{75})^n\Big((\frac {75}{128})^{18}-(\frac {75}{128})^{31}\Big)\\
&\qquad \qquad-\frac{450241}{5323500} (\frac 2{75})^n.
\end{align*}
\end{proof}

\subsection{Asymptotics for the $n$th quantization error $V_n(P)$} \label{sec04}
In this subsection, we show that the quantization dimension $D(P)$ of the condensation measure $P$ exists, quantization coefficient does not exist, and the $D(P)$-dimensional lower and upper quantization coefficients are finite and positive.

\begin{theo} \label{Th2}
Let $P$ be the condensation measure associated with the self-similar measure $\gn$ and $\gk$ be the unique number given by $(\frac 13 (\frac 15)^2)^{\frac {\gk}{2+\gk}}+(\frac 13 (\frac 15)^2)^{\frac {\gk}{2+\gk}}=1$. Then, $\lim_{n\to \infty} \frac{2 \log n}{-\log V_n(P)}=\gk$, i.e., the quantization dimension $D(P)$ of the measure $P$ exists and equals $\gk$.
\end{theo}
\begin{proof} $(\frac 13 (\frac 15)^2)^{\frac {\gk}{2+\gk}}+(\frac 13 (\frac 15)^2)^{\frac {\gk}{2+\gk}}=1$ implies $\gk=\frac{2\log 2}{\log 75-\log 2}$.
For $n\in \D N$,  let $\ell(n)$ be the least positive integer such that $F(\ell(n))\leq n<F(\ell(n)+1)$. Then,
$V_{F(\ell(n)+1)}<V_n\leq V_{F(\ell(n))}$. Thus, we have
\begin{align*}
\frac {2\log\left(F(\ell(n))\right)}{-\log\left(V_{F(\ell(n)+1)}\right)}< \frac {2\log n}{-\log V_n}< \frac {2\log\left(F(\ell(n)+1)\right)}{-\log\left(V_{F(\ell(n))}\right)}.
\end{align*}
Notice that when $n\to \infty$, then $\ell(n)\to \infty$. Recall that $F(\ell(n))=5\cdot 2^{\ell(n)}-2 \ell(n)-4$, and so we have
\begin{align*}
&\lim_{\ell(n)\to\infty} \frac {2\log\left(F(\ell(n))\right)}{-\log\left(V_{F(\ell(n)+1)}\right)}=2 \lim_{\ell(n)\to\infty} \frac {\log \Big(5\cdot 2^{\ell(n)}-2 \ell(n)-4 \Big)}{-\log\Big(\frac{769208}{5884749}(\frac 2{75})^n-\frac {64}{3339}(\frac 1{64})^n\Big)}\\
& =\frac {2\log 2}{-\log \frac 2{75}}=\frac{2\log 2}{\log 75-\log 2}=\gk.
\end{align*}
Similarly, $\mathop{\lim}\limits_{\ell(n)\to\infty} \frac {2\log\left(F(\ell(n)+1)\right)}{-\log\left(V_{F(\ell(n))}\right)}=\gk$. Thus, $\gk\leq \liminf_n \frac{2\log n}{-\log V_n}\leq \limsup_n \frac{2\log n}{-\log V_n}\leq  \gk$ implying the fact that the quantization dimension of the measure $P$ exists and equals $\gk$.
\end{proof}

\begin{theo}\label{Th3}
$D(P)$-dimensional quantization coefficient for the condensation measure $P$ does not exist, and the $D(P)$-dimensional lower and upper quantization coefficients for $P$ are finite and positive.
\end{theo}

\begin{proof} We have $D(P)=\gk=\frac{2\log 2}{\log 75-\log 2}$, and so for any $n \in \D N$,  $2^{2n/\gk}=(\frac {75}{2})^n$. Hence, by Proposition~\ref{prop0004}, we have
\begin{align*}
& \lim_{n\to \infty} (F(n))^{2/\gk}V_{F(n)}(P) =\lim_{n\to \infty} (5\cdot 2^n-2n-4)^{2/\gk}\Big(\frac{769208}{5884749}(\frac 2{75})^n-\frac {64}{3339}(\frac 1{64})^n\Big)\\
&=\lim_{n\to \infty} (\frac {75}{2})^n (5-\frac{2n}{2^n}-\frac 4{2^n})^{2/\gk}\Big(\frac{769208}{5884749}(\frac 2{75})^n-\frac {64}{3339}(\frac 1{64})^n\Big)=5^{2/\gk} \frac{769208}{5884749}.
\end{align*}
Again, by Proposition~\ref{prop0005}, we have
\begin{align*}
& \lim_{n\to \infty} (2^n(6+2^{-30}(2^{13}+1))-2n -6)^{2/\gk}V_{2^n(6+2^{-30}(2^{13}+1))-2n -6}\\
 &=\lim_{n\to \infty} (2^n(6+2^{-30}(2^{13}+1))-2n -6)^{2/\gk}\Big(V_{F(n)}-\frac{128}{3975}(\frac 2{75})^n\Big((\frac {75}{128})^{31}-(\frac {75}{128})^{n+1}\Big)\\
 & \qquad -\frac{1024}{35775}(\frac 2{75})^n\Big((\frac {75}{128})^{18}-(\frac {75}{128})^{31}\Big)-\frac{450241}{5323500} (\frac 2{75})^n\Big)\\
&=\lim_{n\to \infty} (\frac {75}{2})^n (6+2^{-30}(2^{13}+1)-\frac{2n}{2^n}-\frac 6{2^n})^{2/\gk}\Big(\frac{769208}{5884749}(\frac 2{75})^n-\frac {64}{3339}(\frac 1{64})^n\\
 & \qquad -\frac{128}{3975}(\frac 2{75})^n\Big((\frac {75}{128})^{31}-(\frac {75}{128})^{n+1}\Big)-\frac{1024}{35775}(\frac 2{75})^n\Big((\frac {75}{128})^{18}-(\frac {75}{128})^{31}\Big)-\frac{450241}{5323500} (\frac 2{75})^n\Big)\\
 &=6^{2/\gk}\Big(\frac{769208}{5884749}-\frac{128}{3975}(\frac {75}{128})^{31}-\frac{1024}{35775}\Big((\frac {75}{128})^{18}-(\frac {75}{128})^{31}\Big)-\frac{450241}{5323500}\Big).
\end{align*}

Since $(F(n))^{2/\gk}V_{F(n)}(P)$, and $(2^n(6+2^{-30}(2^{13}+1))-2n -6)^{2/\gk}V_{2^n(6+2^{-30}(2^{13}+1))-2n -6}(P)$ are two subsequences of $(n^{2/\gk}V_n(P))_{n\in \D N}$ having two different limits, we can say that the sequence $(n^{2/\gk}V_n(P))_{n\in \D N}$ does not converge, in other words, the $D(P)$-dimensional quantization coefficient for $P$ does not exist.

To show that the $D(P)$-dimensional lower and upper quantization coefficients for $P$ are finite and positive, for $n\in \D N$, let $\ell(n)$ be the least positive integer such that $F(\ell(n))\leq n<F(\ell(n)+1)$. Then,
$V_{F(\ell(n)+1)}<V_n\leq V_{F(\ell(n))}$ implying $(F(\ell(n)))^{2/\gk}V_{F(\ell(n)+1)}<n^{2/\gk}V_n< (F(\ell(n)+1))^{2/\gk} V_{F(\ell(n))}$. As $\ell(n)\to \infty$ whenever $n\to \infty$, we have
\[\lim_{n\to \infty} \frac{F(\ell(n))}{F(\ell(n)+1)}=\lim_{n\to \infty} \frac{5\cdot 2^{\ell(n)}-2 \ell(n)-4 }{5\cdot 2^{\ell(n)+1}-2 (\ell(n)+1)-4 }=\frac 1{2},\]
and so,
\begin{align*} \label{eq56}
&\lim_{n\to \infty} (F(\ell(n)))^{2/\gk} V_{F(\ell(n)+1)}=\frac 1{4^{1/\gk}} \lim_{n\to \infty} (F(\ell(n)+1))^{2/\gk} V_{F(\ell(n)+1)}\\
&=\frac 1{4^{1/\gk}}\lim_{n\to \infty}(5\cdot 2^{\ell(n)+1}-2(\ell(n)+1)-4)^{2/\gk}\Big(\frac{769208}{5884749}(\frac 2{75})^{\ell(n)+1}-\frac {64}{3339}(\frac 1{64})^{\ell(n)+1}\Big)\\
&=\frac 1{4^{1/\gk}} 5^{2/\gk} \frac{769208}{5884749}=\Big(\frac  {25} 4\Big)^{1/\gk} \frac{769208}{5884749} \notag,
\end{align*} and similarly
\begin{align*}
\lim_{n\to \infty} (F(\ell(n)+1))^{2/\gk}V_{F(\ell(n))}=4^{1/\gk}\lim_{n\to \infty} (F(\ell(n)))^{2/\gk}V_{F(\ell(n))}=(100)^{1/\gk} \frac{769208}{5884749}
\end{align*}
yielding the fact that $(\frac  {25} 4)^{1/\gk} \frac{769208}{5884749} \leq \mathop{\liminf}\limits_{n\to\infty} n^{2/\gk} V_n(P)\leq \mathop{\limsup}\limits_{n\to \infty} n^{2/\gk}V_n(P)\leq (100)^{1/\gk} \frac{769208}{5884749}$, i.e., the $D(P)$-dimensional lower and upper quantization coefficients for the condensation measure $P$ is finite and positive. Thus, the proof of the theorem is complete.
\end{proof}

\begin{remark1} Notice that $(\frac 13 (\frac 15)^2)^{\frac {\gk}{2+\gk}}+(\frac 13(\frac 15)^2)^{\frac {\gk}{2+\gk}}=1$ implies $\gk=\frac{2\log 2}{\log75-\log 2}\approx 0.382496>0=D(\gn)$. In Theorem~\ref{Th2}, we have proved that $D(P)=\gk=\max\set{\gk, D(\gn)}.$ 
\end{remark1}

\section{Quantization for a condensation measure associated with a uniform distribution} \label{sec2}

In this section, we investigate the optimal quantization for the condensation measure $P$ when $\gn$ is the uniform distribution on the closed interval $[\frac 25, \frac 35]$.

\begin{lemma} \label{lemma22} Let $X$ and $Y$ be the random variables with the probability distributions $P$ and $\gn$, respectively. Let $E(X)$, $E(Y)$ represent the expected values, and $V:=V(X)$, $W:=V(Y)$ represent the variances of $X$ and $Y$, respectively. Then,
\[E(X)=E(Y)=\frac 12,  \, V=\frac{97}{876}, \te{ and } W=\frac 1{300}.\]

\end{lemma}
\begin{proof} Using Lemma~\ref{lemma1}, we have
$
E(X)=\int x dP(x)=\frac 1 3\int \frac 15 x dP(x) +\frac 13  \int (\frac 15 x +\frac 45 ) dP(x)+\frac 13 \int x d\gn(x)= \frac 1{15}  \,E(X) +\frac 1{15} E(X) +\frac 4{15}+ \frac 16,
$
which implies $E(X)=\frac 1 2$. Again, $E(X^2)=\int x^2 dP(x)=\frac 1 3 \int x^2 d(P\circ S_1^{-1})(x)+\frac 13 \int x^2 d(P\circ S_2^{-1})(x) +\frac 13 \int x^2 d\gn=\frac 1 3  \int(\frac {1} {5} x)^2 dP(x) +\frac 13  \int (\frac 15  x +\frac 45 )^2 dP(x)+\frac 13 \int_{\frac 25}^{\frac 35} 5x^2 dx= \frac 1{75} E(X^2) +\frac 1{75} E(X^2) +\frac 8{75} E(X)+\frac {16}{75} +\frac{19} {225}
$
which implies $E(X^2)=\frac{79}{219}$, and hence
$V(X)=E(X-E(X))^2=E(X^2)-\left(E(X)\right)^2 =\frac{79}{219}-\frac{1}{4}=\frac{97}{876}$. On the other hand, $E(Y)=\int y d\gn=\int_{\frac 2 5}^{\frac 35} 5 y dy =\frac 12$, and $V(Y)=E(Y-\frac 12)^2=\int_{\frac 25}^{\frac 35} 5(y-\frac 12)^2 dy=\frac 1{300}$. Thus, the proof of the lemma is complete.
\end{proof}

\begin{lemma} \label{lemma31} Let $P$ be the condensation measure and $\gn$ be the uniform distribution that occurs in $P$. Then, for any $x_0\in \D R$, we have
\[\int (x-x_0)^2 dP(x) =\frac{97}{876}+(x_0-\frac 12)^2, \te{ and } \int (x-x_0)^2 d\gn(x) =\frac 1{300} +(x_0-\frac 12)^2.\]
\end{lemma}
\begin{proof}  Let $X$ and $Y$ be the random variables as defined in Lemma~\ref{lemma22}. Then, following the standard rule of probability, we have
$\int(x-x_0)^2 dP(x) =V(X)+(x_0-\frac 12)^2$, and
$\int(x-x_0)^2 d\gn(x)=V(Y) +(x_0-\frac 12)^2$, which is the lemma.
\end{proof}
\begin{note}
Let $\go\in I^k$, $k\geq 0$. Then, by equation~\eqref{eq020}, it follows that $P(J_\go)=\frac 1{3^k}$, $P(C_\go)=\frac 1{3^{k+1}}$,
\begin{align*} & E(X : X \in J_\go)=\frac 1{P(J_\go)} \int_{J_\go} x dP=\int x d(P\circ S_\go^{-1})=\int S_\go(x) dP=S_\go(\frac 12), \te{ and }\\
&E(X : X \in C_\go)=\frac 1{P(C_\go)} \int_{C_\go} x dP=\int x d(\gn \circ S_\go^{-1})=\int S_\go(x) d\gn=S_\go(\frac 12).
\end{align*}
For any $x_0 \in \mathbb R$,
\begin{align} \label{eq2} \left\{\begin{array}{cc}  \int_{J_\go} (x-x_0)^2 dP(x) =\frac 1{3^k} \Big(\frac 1{25^k} V  +(S_\go(\frac 12)-x_0)^2\Big), \\
\int_{C_\go} (x-x_0)^2 d\gn(x) =\frac 1{3^{k+1}} \Big(\frac 1{25^k} W  +(S_\go(\frac 12)-x_0)^2\Big).&\end{array}\right.
\end{align}

\end{note}

\begin{remark}
By Lemma~\ref{lemma31}, it follows that the optimal set of one-mean for the condensation measure $P$ consists of the expected value $\frac 12$ and the corresponding quantization error is the variance $V(X)$ of the random variable $X$.
\end{remark}

The following proposition is useful.

\begin{prop}\label{prop011} Let $\gn$ be the uniform distribution on the interval $[\frac 25, \frac 35]$. Then, for $n\geq 1$, the set $\set{\frac 25+\frac {2i-1}{10n} : 1\leq i\leq n}$ is a unique optimal set of $n$-means for $\gn$ with quantization error $V_n(\gn)=\frac{W}{n^2}$. Moreover, the quantization dimension $D(\gn)$ of $\gn$ is one.

\end{prop}
\begin{proof}
Since $\gn$ is uniformly distributed on $[\frac 25, \frac 35]$, the boundaries of the Voronoi regions of an optimal set of $n$-means will divide the interval $[\frac 25, \frac 35]$ into $n$ equal subintervals, i.e., the boundaries of the Voronoi regions are given by
\[\set{\frac 25, \, \frac 25+\frac {1}{5n}, \, \frac 25+\frac 2{5n},\, \cdots, \frac 25+\frac{n-1}{5n}, \frac 35}.\]
This implies that an optimal set of $n$-means for $\gn$ is unique, and it consists of the midpoints of the boundaries of the Voronoi regions, i.e., the optimal set of $n$-means for $\gn$ is given by $\set{\frac 25+\frac{2i-1}{10n} : 1\leq i\leq n}$ for any $n\geq 1$. Then, the $n$th quantization error for $\gn$ is
\begin{align*}
V_n(\gn)&=n \times (\te{the quantization error in each Voronoi region})\\
&=n \Big(\int_{\frac{2}{5}}^{\frac{2}{5}+\frac{1}{5 n}} 5 \Big(x-(\frac{2}{5}+\frac{1}{10 n})\Big)^2  dx\Big)
\end{align*}
which after simplification implies $V_n(\gn)=\frac{1}{300\, n^2}=\frac{W}{n^2}$. Hence,
$D(\gn)=\mathop{\lim}\limits_{n\to \infty } \frac{2 \log (n)}{-\log(\frac{1}{300\, n^2})}=1.$
Thus, the proof of the proposition is complete.
\end{proof}

\begin{remark}
By Proposition~\ref{prop011}, it follows that $\mathop{\lim}\limits_{n\to\infty} n^2 V_n(\gn)=\frac 1{300}$, i.e., one-dimensional quantization coefficient for $\gn$ is finite and strictly positive. In fact, for any probability measure $\gm$ on $\D R^d$ with non-vanishing absolutely continuous part  $\mathop{\lim}\limits_{n\to\infty} n^{2/d} V_n(\gm)$ is finite and strictly positive; in other words, the quantization dimension of a probability measure with non-vanishing absolutely continuous part is equal to the dimension $d$ of the underlying space (see \cite{BW}).
\end{remark}

\begin{lemma} \label{lemma0021}
$\ga_n(\gn)$ be an optimal set of $n$-means for $\gn$. Then, for any $\go \in I^k$, $k\geq 0$, $S_\go(\ga_n(\gn))$ is an optimal set of $n$-means for the image measure $\gn\circ S_\go^{-1}$. Moreover,
\[\int_{C_\go}\min_{a\in S_\go(\ga_n(\gn))}(x-a)^2 dP=\frac 1 3\cdot \frac{1}{75^k}\cdot \frac {W}{n^2}.\]
\end{lemma}
\begin{proof}Let $\ga_n(\gn)$ be an optimal set of $n$-means for $\gn$. Then, $S_\go(\ga_n(\gn))$ is an optimal set of $n$-means for the image measure $\gn\circ S_\go^{-1}$ follows equivalently from Lemma~\ref{lemma001}. Now, using \eqref{eq020} and Proposition~\ref{prop011}, we have
\begin{align*}
&\int_{C_\go}\min_{a\in S_\go(\ga_n(\gn))}(x-a)^2 dP=\frac {1}{3^{k+1}} \int_{C_\go}\min_{a\in S_\go(\ga_n(\gn))}(x-a)^2 d(\gn\circ S_\go^{-1})\\
&=\frac {1}{3^{k+1}} \int_{C}\min_{a\in S_\go(\ga_n(\gn))}(S_\go(x)-a)^2 d\gn=\frac {1}{3^{k+1}}\cdot \frac{1}{25^k} \int_{C}\min_{a\in \ga_n(\gn)}(x-a)^2 d\gn=\frac 1 3\cdot \frac{1}{75^k}\cdot \frac {W}{n^2}.
\end{align*}
Thus, the proof of the lemma is complete.
\end{proof}

\subsection{Essential lemmas and propositions}
In this subsection, we give some lemmas and propositions that we need to determine the optimal sets of $n$-means and the $n$th quantization errors for all $n\geq 2$. To determine the quantization error we will frequently use the formulas given in the expression \eqref{eq2}.

Let us now state the following proposition. The proof of this proposition follows similarly as the proof of Proposition~\ref{prop01111}.
\begin{propo} \label{prop41}
Let $\ga:=\set{a_1, a_2}$ be an optimal set of two-means with $a_1<a_2$. Then, $a_1=\frac {13}{60}$, $a_2=\frac{47}{60}$, and the corresponding quantization error is $V_2=\frac{8003}{262800}=0.0304528$.
\end{propo}

\begin{lemma1}\label{lemma012} Let $\go \in I^k$ for any $k\geq 0$. Then,
\[\int_{J_{\go 1}\uu S_{\go}[\frac 25, \frac 12]}(x-S_{\go}(\frac{13}{60}))^2 dP=\frac {1}{75^{k}} \cdot \frac 12 V_2=\int_{J_{\go 2}\uu S_{\go}[\frac 12, \frac 35]}(x-S_{\go}(\frac{47}{60}))^2 dP.\]
\end{lemma1}

\begin{cor} \label{cor2}
Let $\go \in I^k$ for  $k\geq 0$. Then, for any $a\in \D R$,
\begin{align*}
&\int_{J_{\go 1}\uu S_{\go}([\frac 25, \frac 12])}(x-a)^2 dP=\frac 1{3^k}\Big(\frac 1{25^k} \frac 12 V_2+\frac 12(S_\go(\frac{13}{60})-a)^2\Big), \te{ and } \\
&\int_{S_{\go}([\frac 12, \frac 35])\uu J_{\go 2}}(x-a)^2 dP=\frac 1{3^k}\Big(\frac 1{25^k} \frac 12 V_2+\frac 12(\frac{47}{60})-a)^2\Big).
\end{align*}
\end{cor}

\begin{propo}\label{prop51}
Let $\ga:=\set{a_1, a_2, a_3}$ be an optimal set of three-means with $a_1<a_2<a_3$. Then, $a_1=S_1(\frac 12)=\frac 1{10}$, $a_2=\frac 12$, and $a_3=S_2(\frac 12)=\frac 9{10}$. The corresponding quantization error is $V_3=\frac{89}{21900}=0.00406393$.
\end{propo}

\begin{proof} Let $\gb:=\set{\frac 1{10}, \frac 12, \frac 9{10}}$. Then, using \eqref{eq2}, we have
\begin{align*}
&\int\min_{a\in \gb} (x-a)^2 dP=\int_{J_1}(x-\frac {1}{10})^2 dP+\int_{C}(x-\frac 12)^2 dP+\int_{J_2}(x-\frac 9{10})^2 dP\\
&=\frac 13\Big(\frac 1{25} V+\frac 13 W+\frac 1{25} V\Big)=\frac{89}{21900}=0.00406393.
\end{align*}
Since $V_3$ is the quantization error for three-means, we have $0.00406393\geq V_3$.
 Let $\ga:=\{a_1, a_2, a_3\}$ be an optimal set of three-means with $a_1<a_2<a_3$. Since $a_1, a_2$, and $a_3$ are the centroids of their own Voronoi regions, we have $0 <a_1<a_2<a_3< 1$. Suppose that $\frac 15 \leq a_1$. Then,
 \[V_3\geq \int_{J_1}(x-\frac 15)^2 dP=\frac{79}{16425}=0.00480974>V_3,\]
 which is a contradiction. So, we can assume that $a_1<\frac 15$. Similarly, we have $\frac 45 <a_3$. Suppose that $a_2\leq \frac 25$. Then,
 \[V_3\geq \int_{C}(x-\frac 25)^2 dP=\frac{1}{3}(W+(\frac{1}{2}-\frac{2}{5})^2)=\frac{1}{225}=0.00444444>V_3,\]
 which is a contradiction. So, we can assume that $\frac 25<a_2$. Similarly, $a_2<\frac 35$. Thus, we have $\frac 25<a_2<\frac 35$. We now show that the Voronoi region of $a_1$ does not contain any point from $C$. Suppose that $\frac 12(a_1+a_2)\geq \frac 25$. Then, $a_2\geq \frac 45-a_1>\frac 45-\frac 15=\frac 35$, which is a contradiction, and so the Voronoi region of $a_1$ does not contain any point from $C$. Similarly, the Voronoi region of $a_3$ does not contain any point from $C$. In the similar fashion, we can show that the Voronoi region of $a_2$ does not contain any point from $J_1$ and $J_2$. Hence,
 \[a_1=S_1(\frac 12)=\frac 1{10}, \, a_2=E(X : X\in C)=\frac 12, \te{ and } a_3=S_2(\frac 12)=\frac 9{10},\]
 and the corresponding quantization error is $V_3=\frac{89}{21900}=0.00406393$. Thus, the proof of the lemma is complete.
\end{proof}

\begin{lemma1}  \label{lemma61} Let $\ga$ be an optimal set of four-means. Then, $\ga\ii J_1\neq \es$, $\ga \ii J_2\neq \es$, and $\ga\ii C\neq \es$. Moreover, $\ga$ does not contain any point from the open intervals $(\frac 15, \frac 25)$ and $(\frac 35, \frac 45)$.
\end{lemma1}

\begin{proof} Let us first consider a set of four points $\gb:=\set{S_1(\frac {13}{60}), S_1(\frac {47}{60}), \frac 12, S_2(\frac 12)}$. Then, using Lemma~\ref{lemma012},
 \begin{align*}
& \int \min_{a\in \gb}(x-a)^2 dP =2\int_{J_{11}\uu S_1[\frac 25,\frac 12]}(x-S_{1}(\frac {13}{60}))^2 dP+\int_C(x-\frac 12)^2 dP   +\int_{J_{2}}(x-S_{2}(\frac {1}{2}))^2 dP\\
&=\frac {1}{75}  V_2+\frac 13 W+\frac 1{75} V=\frac{59003}{19710000}=0.00299356.
\end{align*}
Since $V_4$ is the quantization error for four-means, we have $V_4\leq 0.00299356$. Let $\ga:=\set{a_1<a_2<a_3<a_4}$ be an optimal set of four-means. Proceeding in the similar way as shown in the proof of Proposition~\ref{prop51}, we have $0<a_1<\frac 15$ and $\frac 45<a_4<1$ implying $\ga \ii J_1\neq\es$ and $\ga \ii J_2\neq \es$. We now show that $\ga \ii C\neq \es$. For the sake of contradiction, assume that $\ga \ii C=\es$. Then, if $a_3<\frac 25$, we have
\begin{align*}
V_4&\geq \int_C(x-\frac 25)^2dP+\int_{J_2}(x-S_2(\frac 12))^2 dP=\frac{1}{225}+\frac 1{75} V=\frac{389}{65700}=0.00592085>V_4,
\end{align*}
which leads to a contradiction. Hence, $\frac 35<a_3$.
Suppose that $a_2\leq \frac 7{20}$. Then, due to symmetry, we can assume that $\frac {13}{20}\leq a_3$, and so
\begin{align*}
V_4&\geq \int_{[\frac 2 5, \frac 12]} (x-\frac 7{20})^2dP+\int_{[\frac 12, \frac 35]} (x-\frac {13}{20})^2dP=\frac{2}{3} \int_{[\frac{2}{5}, \frac 12]} (x-\frac{7}{20})^2  d\gn=\frac{13}{3600}=0.00361111>V_4,
\end{align*}
which gives a contradiction. So, we can assume that $\frac 7{20}<a_2<\frac 25<\frac 35<a_3<\frac {13}{20}$. Then, $\frac 12(a_1+a_2)<\frac 15$ and $\frac 12(a_3+a_4)>\frac 45$ yielding $a_1<\frac 1{20}$ and $a_4>\frac {19}{20}$; otherwise, the quantization error can strictly be reduced by moving $a_2$ to $\frac 25$ and $a_3$ to $\frac 35$ contradicting the fact that $\ga$ is an optimal set of four-means. Then,
\begin{align*}
V_4&\geq 2\Big(\int_{J_{12}}(x-\frac 1{20})^2 dP+\int_{[\frac 25, \frac 12]}(x-\frac 25)^2dP\Big)=\frac{48349}{9855000}=0.00490604>V_4,
\end{align*}
which is a contraction. Hence, $\ga\ii C\neq \es$. Thus, we see that $\ga\ii J_1\neq \es$, $\ga \ii J_2\neq \es$, and  $\ga\ii C\neq \es$. To complete the rest of the proof, for the sake of contradiction, assume that $\ga$ contains a point from the open interval $(\frac 15, \frac 25)$. Then, $\ga$ can not contain any point from $(\frac 35, \frac 45)$ because $\te{card}(\ga)=4$, and $\ga\ii J_1\neq \es$, $\ga \ii J_2\neq \es$, and  $\ga\ii C\neq \es$. Assume that $\ga$ contains $a_2$ from $(\frac 15, \frac 25)$. Then, $0<a_1<\frac 15<a_2<\frac 25<a_3<\frac 35<\frac 45<a_4<1$, and so, the Voronoi region of $a_4$ does not contain any point from $C$, and the Voronoi region of $a_3$ does not contain any point from $J_2$.
Under the assumption $a_2 \in (\frac 15, \frac 25)$ the following three cases can arise:

Case 1:  $\frac 1 3 \leq a_2<\frac 25$.

Then,
\begin{align} \label{eq001} & \min_{a\in \set{a_2, a_3}} \left\{\int_C (x - a)^2 dP : \frac{1}{3}\leq a_2< \frac{2}{5} \te{ and } \frac{2}{5}\leq a_3\leq \frac{3}{5}\right\} \notag \\
&=\min_{a\in \set{a_2, a_3}} \left\{\frac{1}{3} \Big(\int_{\frac{2}{5}}^{\frac{a_2+a_3}{2}} 5 (x-a_2)^2  dx+\int_{\frac{a_2+a_3}{2}}^{\frac{3}{5}} 5 (x-a_3)^2  dx\Big) : \frac{1}{3}\leq a_2< \frac{2}{5}, \, \frac{2}{5}\leq a_3\leq \frac{3}{5} \right\}  \\
&= \frac{1}{2025} \, (\te{when } a_2=\frac{2}{5} \te{ and } a_3=\frac{8}{15}). \notag
\end{align}
Again, $\frac 12(a_1+a_2)<\frac 15$ implies $a_1<\frac 25-a_2\leq \frac 25-\frac 13=\frac 1{15}<\frac 1{10}=S_1(\frac 12)$, and so
\[V_4\geq \int_{S_1[\frac 12, \frac 35]\uu J_{12}}(x-\frac 1{15})dP+ \frac{1}{2025}+\int_{J_2}(x-S_2(\frac 12))^2 dP=\frac{20833}{5913000}=0.00352325>V_4,\]
which leads to a contradiction.

Case 2:  $\frac 4{15} \leq a_2\leq \frac 13$.

Then, proceeding as \eqref{eq001}, we have  \begin{align*} \min_{a\in \set{a_2, a_3}} \left\{\int_C (x - a)^2 dP : \frac 4{15} \leq a_2\leq \frac 13, \,  \frac{2}{5}\leq a_3\leq \frac{3}{5}, \te{ and }  \frac 25\leq \frac{a_2+a_3}{2}\right\}
=\frac{11}{10935}
\end{align*}
which occurs when $a_2=\frac 13$ and $a_3=\frac{23}{45}$. Moreover, $\frac 12(a_1+a_2)<\frac 15$ implying $a_1<\frac 25-a_2\leq \frac 25-\frac 4{15}=\frac 2{15}$.
First, assume that $\frac 1{15} \leq a_1<\frac 2{15}$.
Then,
\begin{align*}
V_4\geq \int_{J_{11}}(x-\frac 1{15})^2 dP+\int_{C_1}(x-S_1(\frac 12))^2 dP+\int_{J_{12}}(x-\frac 2{15})^2 dP+\frac{11}{10935}+\frac 1{75} V=\frac{608693}{199563750},
\end{align*}
implying $V_4\geq\frac{608693}{199563750}=0.00305012>V_4$, which yields a contradiction.
Next, assume that $0< a_1<\frac 1{15}$. Then, for any $x\in J_{12}$, we have
\begin{equation*} \mathop{\min}\limits_{a\in \set{a_1, a_2}}(x-a)^2 \geq (x-\frac 1{30})^2,
\end{equation*}
because $a_2-x \geq \frac 4{15}-x\geq  \frac 1{15}\geq x-\frac 1{30}\geq \frac 2{75}>0$, and $x-a_1\geq x- \frac 1{15}\geq \frac{7}{75}> x-\frac 1{30}\geq \frac 2{75}>0$. Thus, we have
\begin{align*}
V_4\geq \int_{J_{12}}(x-\frac 1{30})^2 dP+\frac{11}{10935}+\frac 1{75} V=\frac{488149}{99781875}=0.00489216>V_4,
\end{align*}
which leads to another contradiction. Thus, Case 2 gives a contradiction.

Case 3:  $\frac 1{5}< a_2\leq \frac 4{15}$.

Then, proceeding as \eqref{eq001}, we have  \begin{align*} \min_{a\in \set{a_2, a_3}} \left\{\int_C (x - a)^2 dP : \frac 1{5}< a_2\leq \frac 4{15}, \,  \frac{2}{5}\leq a_3\leq \frac{3}{5}, \te{ and }   \frac 25\leq \frac{a_2+a_3}{2}\right\}
=\frac{1}{675},
\end{align*}
which occurs when $a_2=\frac 4{15}$ and $a_3=\frac{8}{15}$. Moreover,
\begin{align*}  & \min_{a\in \set{a_1, a_2}} \left\{\int_{J_1} (x - a)^2 dP : 0<a_1<\frac 15<a_2< \frac{2}{5} \right\}\geq \int_{J_1}\min_{a\in S_1(\ga_2)}(x-a)^2 dP=\frac 1{75} V_2.
\end{align*}
Thus, we see that
\begin{align*}
V_4\geq\frac 1{75} V_2 +\frac{1}{675}+\frac 1{75} V=\frac{7367}{2190000}=0.00336393>V_4,
\end{align*}
which is a contradiction.

Therefore, we can assume that $\ga$ does not contain any point from $(\frac 15, \frac 25)$. Reflecting the situation with respect to the point $\frac 12$, we can also say that $\ga$ does not contain any point from the open interval  $(\frac 35, \frac 45)$. Thus the proof of the lemma is complete.
\end{proof}

\begin{remark1} \label{remark35}
By Lemma~\ref{lemma61}, it is easy to see that the sets $\set{S_1(\frac {13}{60}), S_1(\frac {47}{60}), \frac 12, S_2(\frac 12)}$, and $\set{S_1(\frac 12), \frac 12, S_2(\frac {13}{60}), S_2(\frac {47}{60})}$ form two different optimal sets of four-means.
\end{remark1}

\begin{propo} \label{prop11}
Let $\ga$ be an optimal set of $n$-means with $n\geq 3$. Then, $\ga\ii J_1\neq \es$, $\ga \ii J_2\neq \es$, and $\ga\ii C\neq \es$.
\end{propo}
\begin{proof}
Due to Proposition~\ref{prop51} and Lemma~\ref{lemma61}, the proposition is true for $n=3$ and $n=4$. Let $\ga:=\set{a_1, a_2, \cdots, a_n}$ be an optimal set of $n$-means for $n\geq 5$, where $0<a_1<a_2<\cdots<a_n<1$. Proceeding in the similar way as shown in the proof of Proposition~\ref{prop51}, we see that $0<a_1<\frac 15$, and $\frac 45<a_n<1$ yielding $\ga\ii J_1\neq \es$ and $\ga\ii J_2\neq \es$. Consider a set of five points $\gb:=S_1(\ga_2)\uu \set{\frac 12}\uu S_2(\ga_2)$, where $\ga_2$ is an optimal set of two-means for $P$.
Then, using Lemma~\ref{lemma012}, and the symmetry of $P$, we obtain the distortion error for the set $\gb$ as
\begin{align} \label{eq451}
& \int \min_{a\in \gb}(x-a)^2 dP \notag =2\int_{J_{11}\uu S_1[\frac 25, \frac 12]}(x-S_{1}(\frac {13}{60}))^2 dP +\int_C(x-\frac 12)^2 dP  \\
&= \frac {1}{75} V_2+\frac 13 W=\frac{29903}{19710000}=0.00151715.
\end{align}
Since $V_n$ is the quantization error for $n$-means with $n\geq 5$, we have $V_n\leq 0.00151715$.  Notice that for any $a\in \ga$, since $P(M(a|\ga)>0$, there can not be more than one point from $\ga$ in any of the open intervals $(\frac 15, \frac 25)$ and $(\frac 35, \frac 45)$. We now show that $\ga \ii C\neq \es$. On the contrary assume that $\ga\ii C=\es$. Let $j=\max\set{i : a_i< \frac 25 \te{ and } 1\leq i\leq n}$. Then, $a_j< \frac 25$ and $\frac 35< a_{j+1}$. If $a_j\leq \frac 12(\frac 15+\frac 25)=\frac 3{10}$, then
\begin{align*}
& \int \min_{a\in \ga}(x-a)^2 dP\geq  \int_C \min_{a\in \ga}(x-a)^2 dP=\frac 13\int_C\min_{a\in \set{a_j, a_{j+1}}}(x-a)^2 d\gn\\
&\geq \frac 13\int_{[\frac 2 5, \frac 9{20}]}(x-\frac 3{10})^2 d\gn+\frac 13\int_{[\frac 9{20}, \frac 35]}(x-\frac 3 5)^2 d\gn=\frac{23}{7200}=0.00319444>V_n,
\end{align*}
which yields a contradiction. Similarly, if $a_{j+1}\geq \frac 7{10}$ a contradiction will arise. So, we can assume that $\frac 3{10}<a_j<\frac 25<\frac 35<a_{j+1}< \frac 7{10}$. Then, $\frac 12 (a_{j-1}+a_j)<\frac 15$ implying $a_{j-1}<\frac 25-a_j\leq \frac 25-\frac 3{10}=\frac 1{10}$, and similarly $\frac 12(a_{j+1}+a_{j+2})>\frac 45$ implies $a_{j+2}\geq \frac 9{10}$. Thus, we see that
\begin{align*}
&V_n \geq \int_{J_{12}} (x-\frac 1{10})^2dP+\int_{[\frac 2{5}, \frac 12]}(x-\frac 25)^2 dP+\int_{[\frac 12, \frac 35]}(x-\frac 35)^2 dP+\int_{J_{21}} (x-\frac 9{10})^2dP,
\end{align*}
which after simplification yields $V_n\geq \frac{12677}{4927500}=0.0025727>V_n$, which is a contradiction. All these contradictions arise due to our assumption $\ga\ii C=\es$. Hence, we can conclude that $\ga\ii C\neq \es$. This completes the proof of the proposition.
\end{proof}

\begin{lemma1}\label{lemma71}
Let $\ga$ be an optimal set of five-means. Then, $\ga$ does not contain any point from $(\frac 1 5, \frac 25)\uu (\frac 35, \frac 45)$.
\end{lemma1}

\begin{proof} Let $\ga:=\set{a_1<a_2<a_3<a_4<a_5}$ be an optimal set of five-means. By the expression \eqref{eq451}, we have $V_5 \leq 0.00151715$. Proposition~\ref{prop11} says that $\ga$ contains points from $J_1$, $C$ and $J_2$. Suppose that $\ga$ contains a point from $(\frac 15, \frac 25)$.
Then, two cases can arise:

Case 1. $\frac 3{10}\leq a_2<\frac 25$.

Due to symmetry of $P$ about $\frac 12$, we can assume that $a_3=\frac 12$ and $\frac 35<a_4\leq \frac 7{10}$. Then, $\frac 12(a_1+a_2)<\frac 15$ and $\frac 45<\frac 12(a_4+a_5)$ yielding $a_1<\frac 1{10}$ and $\frac 9{10}<a_5$, and so, we have
\begin{align*} V_5&\geq 2\Big(\int_{J_{11}\uu S_1[\frac 2{5}, \frac 1{2}]} (x-S_1(\frac {13}{60}))^2 dP+\int_{S_1[\frac 1{2}, \frac {3}{5}]\uu J_{12}}(x-\frac 1{10})^2 dP+\int_{[\frac{2}{5}, \frac{9}{20}]}  (x-\frac{2}{5})^2 dP\\
& \qquad \qquad \qquad +\int_{[\frac{9}{20}, \frac{1}{2}]}  (x-\frac{1}{2})^2 dP\Big)=\frac{21289}{9855000}=0.00216022>V_5,
\end{align*}
which is a contradiction.

Case 2. $\frac 1{5}< a_2\leq \frac 3{10}$.

Then, $\frac 12(a_2+a_3)>\frac 25$ implying $a_3>\frac 45 - a_2\geq \frac 45-\frac 3{10}=\frac 12$ which leads to a contradiction, because by our assumption $a_3=\frac 12$.
Therefore, $\ga$ does not contain any point from $(\frac 15, \frac 25)$. Similarly, $\ga$ does not contain any point from $(\frac 35, \frac 45)$, which completes the proof of the lemma.
\end{proof}

\begin{remark1}\label{remark351} By Lemma~\ref{lemma71} and due to symmetry of $P$, it can be seen that the set $S_1(\ga_2)\uu \set{\frac 12}\uu S_2(\ga_2)$, where $\ga_2$ is an optimal set of two-means for $P$, forms an optimal set of five-means.

\end{remark1}

\begin{lemma1}\label{lemma81}
Let $\ga$ be an optimal set of six-means. Then, $\ga$ does not contain any point from $(\frac 1 5, \frac 25)\uu (\frac 35, \frac 45)$.
\end{lemma1}

\begin{proof}
Let $\ga:=\set{a_1<a_2<a_3<a_4<a_6}$ be an optimal set of six-means. Suppose that $\ga$ contains a point from $(\frac 15, \frac 25)$. Proposition~\ref{prop11} says that $\ga$ contains points from $J_1$, $C$ and $J_2$. Consider the set of six points $\gb:=S_1(\ga_2) \uu S_2(\ga_2)\uu \ga_2(\gn)$, where $\ga_2\in \C C_2(P)$ and $\ga_2(\gn) \in \C C_2(\gn)$.  Then,
\begin{align*}
\int\min_{a \in \gb}(x-a)^2 dP&=4\int_{J_{11}\uu S_1[\frac25, \frac 12]}(x-S_{1}(\frac {13}{60}))^2 dP+\frac 13 V_2(\gn)=\frac{21481}{19710000}=0.00108985.
\end{align*}
Since $V_6$ is an optimal set of six-means, $V_6\leq 0.00108985$. Since $\ga\ii C\neq \es$, $\te{card}(\ga)$ is even, and $P$ is symmetric about $\frac 12$, we can assume that $\ga$ contains equal number of points from both sides of the point $\frac 12$ implying $\te{card}(\ga\ii C)=2$, and $a_3, a_4 \in C$. Since $\ga$ contains a point from $(\frac 15, \frac 25)$, the following two cases can arise:

Case 1. $\frac 3{10}\leq a_2<\frac 25$.

Proceeding as Case~1 of Lemma~\ref{lemma71} in this case, we have

\begin{align*} V_6&\geq 2\Big(\int_{J_{11}\uu S_1[\frac 2{5}, \frac 1{2}]} (x-S_1(\frac {13}{60}))^2 dP+\int_{S_1[\frac 1{2}, \frac {3}{5}]\uu J_{12}}(x-\frac 1{10})^2 dP\Big)=0.00188245>V_6,
\end{align*}
which is a contradiction.

Case 2. $\frac 1{5}< a_2\leq \frac 3{10}$.

Due to symmetry we can assume that $\frac 7{10}\leq a_5<\frac 45$. Then, $\frac 12(a_2+a_3)>\frac 25$ implying $a_3>\frac 45 - a_2\geq \frac 45-\frac 3{10}=\frac 12$. Similarly,  $\frac 12(a_4+a_5)<\frac 35$ implying  $a_4<\frac 12$, which leads to a contradiction as $a_3<a_4$.

Therefore, $\ga$ does not contain any point from $(\frac 15, \frac 25)$. Likewise, $\ga$ does not contain any point from $(\frac 35, \frac 45)$, which completes the proof of the lemma.
\end{proof}

\begin{remark1} \label{remark99} By Lemma~\ref{lemma81}, it is easy to see that if $\ga_6$ is an optimal set of six-means for $P$, then $\ga_6=S_1(\ga_2)\uu \ga_2(\gn)\uu S_2(\ga_2)$, where $\ga_2$ is an optimal set of two-means for $P$, and the corresponding quantization error is given by $V_6=\frac{21481}{19710000}=0.00108985$.

\end{remark1}

\begin{lemma1}\label{lemma91}
Let $\ga$ be an optimal set of seven-means. Then, $\ga$ does not contain any point from $(\frac 1 5, \frac 25)\uu (\frac 35, \frac 45)$.
\end{lemma1}

\begin{proof} Consider the set $\gb:=S_1(\ga_2)\uu S_2(\ga_3)\uu \ga_2(\gn)$, where $\ga_2 \in \C C_2(P)$, $\ga_3\in \C C_3(P)$, and $\ga_2(\gn) \in \C C_2(\gn)$. Then,
\begin{align*}
&\int\min_{a\in \gb}(x-a)^2 dP=\int_{J_1}\min_{a \in S_1(\ga_2)}(x-a)^2 dP+ \int_{J_2}\min_{a \in S_2(\ga_3)}(x-a)^2 dP+\frac 13 V_2(\gn) \\
&=\frac 1{75} V_2+ \frac 1{75}\Big(\frac 1{75} V+\frac 13 W+\frac 1{75} V\Big)+\frac{1}{3} \frac 1{1200}=\frac{7273}{9855000}=0.000738001.
\end{align*}
Since $V_7$ is the quantization error for seven-means, we have $V_7\leq 0.000738001$. Let $\ga:=\set{a_1<a_2<a_3<a_4<a_6<a_7}$ be an optimal set of seven-means. Proposition~\ref{prop11} says that $\ga$ contains points from $J_1$, $C$ and $J_2$. Suppose that $\ga$ contains a point from $(\frac 15, \frac 25)$. We now prove the following claim:

\tit{Claim~1. $\te{card}(\ga\ii J_1)\geq 2$ and  $\te{card}(\ga\ii J_2)\geq 2$.}

Suppose that $\te{card}(\ga\ii J_1)=1$. Then, two cases can arise:

Case 1. $\frac 3{10}\leq a_2<\frac 15$.

Then, $\frac 12(a_1+a_2)<\frac 15$ implying $a_1< \frac{1}{10}$, and so
\begin{align*} V_7&\geq \int_{J_{11}\uu S_1[\frac 2{5}, \frac 1{2}]} (x-S_1(\frac {13}{60}))^2 dP+\int_{S_1[\frac 1{2}, \frac {3}{5}]\uu J_{12}}(x-\frac 1{10})^2 dP=\frac{37103}{39420000}=0.000941223>V_7,
\end{align*}
which is a contradiction.

Case 2. $a_2\leq  \frac 3{10}$.

Then, $\frac 12(a_2+a_3)>\frac 25$ yielding $a_3>\frac 12$, and so
\begin{align*} V_7&\geq \int_{J_1} \min_{a\in \set{S_1(\frac{13}{60}), S_1(\frac{47}{60})}} (x-a)^2 dP+\int_{[\frac 25, \frac 12]}(x-\frac 12)^2 dP=\frac 1{75} V_2+ \frac 13 \int_{[\frac 25, \frac 12]}(x-\frac 12)^2 d\gn,
\end{align*}
which implies $V_7\geq \frac{18953}{19710000}=0.000961593>V_7$, and so a contradiction arises.

Thus, we can assume that $\te{card}(\ga\ii J_1)\geq 2$.
Similarly, we can show that $\te{card}(\ga\ii J_2)\geq 2$.  Thus, the claim is true.

Next, we prove the following claim:

\tit{Claim 2. $\te{card}(\ga\ii C)=2$.}

By the hypothesis $\ga$ contains a point from $(\frac 15, \frac 25)$; by Claim 1, $\te{card}(\ga\ii J_1)\geq 2$ and  $\te{card}(\ga\ii J_2)\geq 2$. So, we can assume that $\tit{card}(\ga\ii C)\leq 2$. Suppose that $\tit{card}(\ga\ii C)=1$. Let $a_j=\max\set{a_i : a_i<\frac 25 \te{ and } 1\leq i\leq 7}$. Then, $a_j<\frac 25$. Consider the following two cases:

Case~A. $\frac 3{10}\leq a_j< \frac 25$.

Then, $\frac 12(a_{j-1}+a_j)<\frac 15$ implying $a_{j-1}<\frac 1{10}$. First, suppose that $\ga$ does not contain any point from $(\frac 35, \frac 45)$.
Proceeding as \eqref{eq001}, we have  \begin{align*} \min_{a\in \set{a_j, a_{j+1}}} \left\{\int_C (x - a)^2 dP : \frac 3{10}< a_j\leq \frac 25, \,  \frac{2}{5}\leq a_{j+1}\leq \frac{3}{5}, \te{ and }   \frac 25\leq \frac{a_j+a_{j+1}}{2}\right\}
=\frac{1}{2025},
\end{align*}
which occurs when $a_j=\frac 2{5}$ and $a_{j+1}=\frac{8}{15}$. Then,
\begin{align*} V_7\geq \int_{S_1[\frac 1{2}, \frac {3}{5}]\uu J_{12}}(x-\frac 1{10})^2 dP+\frac{1}{2025}=\frac{1457}{1182600}=0.00123203>V_7,
\end{align*}
which gives a contradiction. Next, suppose that $\ga$ contains a point from $(\frac 35, \frac 45)$, i.e., $a_{j+2} \in (\frac 35, \frac 45)$. Then,
\begin{align*}   \min_{a\in \set{a_j, a_{j+1}, a_{j+2}}} &\Big\{\int_C (x - a)^2 dP : a_j\in [\frac 3{10}, \frac 25), \, a_{j+1}\in [\frac{2}{5}, \frac{3}{5}], \, a_{j+2} \in (\frac 35, \frac 45), \\
& \te{ and }   \frac 25\leq \frac{a_j+a_{j+1}}{2}<\frac{a_{j+1}+a_{j+2}}{2}\leq \frac 35\Big\}=\frac{1}{3600},
\end{align*}
which occurs when $a_j=\frac 25, \, a_{j+1}=\frac 12$ and $a_{j+2}=\frac 35$, and so
\begin{align*} V_7\geq \int_{S_1[\frac 1{2}, \frac {3}{5}]\uu J_{12}}(x-\frac 1{10})^2 dP+\frac{1}{3600}=\frac{89}{87600}=0.00101598>V_7
\end{align*}
which leads to a contradiction. Thus, a contradiction arises in Case 3.

Case~B. $\frac 15<a_j\leq \frac 3{10}$.

Suppose that $\ga$ does not contain any point from $(\frac 35, \frac 45)$. Then,

\begin{align*} \min_{a\in \set{a_j, a_{j+1}}} \left\{\int_C (x - a)^2 dP : \frac 15 < a_j\leq \frac 3{10}, \,  \frac{2}{5}\leq a_{j+1}\leq \frac{3}{5}, \te{ and }   \frac 25\leq \frac{a_j+a_{j+1}}{2}\right\}
=\frac{1}{900},
\end{align*}
which occurs when $a_j=\frac 3{10}$ and $a_{j+1}=\frac 12$ yielding
$V_7\geq \frac{1}{900}>0.00111111>V_7$, which is a contradiction. Next, suppose that $\ga$ contains a point from $(\frac 35, \frac 7{10}]$. This case can be considered as a reflection of the last part of Case~A. Thus, if $\ga$ contains a point from $(\frac 35, \frac 7{10}]$, a contradiction arises.
Now, suppose that $\ga$ contains a point from $[\frac 7{10}, \frac 45)$, i.e., $a_{j+2}\in [\frac 7{10}, \frac 45)$. Then,
\begin{align*}   \min_{a\in \set{a_j, a_{j+1}, a_{j+2}}} &\Big\{\int_C (x - a)^2 dP : a_j\in (\frac 15, \frac 3{10}], \, a_{j+1}\in [\frac{2}{5}, \frac{3}{5}], \, a_{j+2} \in [\frac 7{10}, \frac 45), \\
& \te{ and }   \frac 25\leq \frac{a_j+a_{j+1}}{2}<\frac{a_{j+1}+a_{j+2}}{2}\leq \frac 35\Big\}=\frac{1}{900},
\end{align*}
which occurs when $a_j=\frac 3{10}, \, a_{j+1}=\frac 12$ and $a_{j+2}=\frac 7{10}$.
Then,
\begin{align*} V_7\geq \frac{1}{900}=0.00111111>V_7
\end{align*}
which yields a contradiction.

Hence, $\te{card}(\ga\ii C)=2$, i.e., Claim~2 is true.

By the hypothesis $\ga$ contains a point from $(\frac 15, \frac 35)$; by Claim~1 and Claim~2, we have
$\te{card}(\ga\ii J_1)\geq 2$, $\te{card}(\ga\ii J_2)\geq 2$, and $\te{card}(\ga\ii C)= 2$. So, we can assume that $\te{card}(\ga\ii J_1)= \te{card}(\ga\ii J_2)=2$, $\te{card}(\ga\ii C)= 2$, and $\ga$ does not contain any point from $(\frac 35, \frac 45)$. Then, the following two cases arise:

Case~5. $\frac 3{10} \leq a_3 <\frac 25$.

Then,
\begin{align*} V_7&\geq \int_{S_1[\frac 1{2}, \frac {3}{5}]\uu J_{12}}(x-\frac 1{10})^2 dP+\int_{J_2}\min_{a\in S_2(\ga_2)}(x-a)^2 dP=0.00114424>V_7,
\end{align*}
which is a contradiction.

Case~6. $\frac 15<a_3\leq  \frac 3{10}$.

Then,
\begin{align*}   \min_{a\in \set{a_3, a_{4}, a_{5}}} &\Big\{\int_C (x - a)^2 dP : \frac 15<a_3\leq \frac{3}{10}, \, \frac{2}{5}\leq a_4<a_5\leq \frac{3}{5},
\te{ and }   \frac 25\leq \frac{a_3+a_{4}}{2}\Big\}=\frac{1}{1620},
\end{align*}
which occurs when $a_3=\frac{3}{10}, \, a_4=\frac{1}{2}$, and $a_5=\frac{17}{30}$. Thus, in this case,
\[V_7\geq \int_{C}\min_{a\in \set{a_3,a_4, a_5}}(x-a)^2 dP+ \int_{J_2}\min_{a\in S_2(\ga_2)}(x-a)^2 dP=\frac{1}{1620}+\frac 1{75}V_2=\frac{60509}{59130000},\]
i.e., $V_7\geq \frac{60509}{59130000}=0.00102332>V_7$,
which gives a contradiction.

Therefore, we can assume that $\ga$ does not contain any point from the interval $(\frac 15, \frac 25)$. Reflecting the situation with respect to the point $\frac 12$, we can also assume that $\ga$ does not contain any point from the interval $(\frac 35, \frac 45)$. Hence, the proof of the lemma is complete.
\end{proof}

\begin{remark1}\label{remark101}
In Lemma~\ref{lemma91}, we have proved that if $\ga$ is an optimal set of seven-means, then $\te{card}(\ga\ii J_1)\geq 2$, $\te{card}(\ga\ii J_2)\geq 2$, and $\te{card}(\ga\ii C)=2$. Moreover, $\ga$ does not contain any point from $(\frac 15, \frac 25)\uu (\frac 35, \frac 45)$.  Therefore, we can assume that if $\ga$ is an optimal set of seven-means, then $\ga=S_1(\ga_3)\uu S_2(\ga_2) \uu \ga_2(\gn)$, or $\ga=S_1(\ga_2)\uu S_2(\ga_3) \uu \ga_2(\gn)$, where $\ga_2 \in \C C_2(P)$, $\ga_3\in \C C_3(P)$, and $\ga_2(\gn) \in \C C_2(\gn)$, and the corresponding quantization error is $V_7=\frac{7273}{9855000}=0.000738001$.
\end{remark1}
\begin{propo} \label{prop21}
Let $\ga$ be an optimal set of $n$-means with $n\geq 3$. Then, $\ga$ does not contain any point from $(\frac 1 5, \frac 25)\uu (\frac 35, \frac 45)$.
\end{propo}

\begin{proof} By Proposition~\ref{prop51} and Lemma~\ref{lemma61}, Lemma~\ref{lemma71}, Lemma~\ref{lemma81}, and Lemma~\ref{lemma91}, the proposition is true for $3\leq n\leq 7$. We now prove that the proposition is true for all $n\geq 8$. Let $\ga:=\set{a_1<a_2<\cdots<a_n}$ be on optimal set of $n$ means for $n\geq 8$. Consider the set $\gb=S_1(\ga_3)\uu S_2(\ga_3)\uu \ga_2(\gn)$, where $\ga_3\in \C C_3(P)$ and $\ga_2(\gn)\in \C C_2(\gn)$. Then,
\begin{align*}
&\int\min_{a\in \gb}(x-a)^2 dP\\
&=2 \Big(\int_{J_{11}}(x-S_{11}(\frac 12))^2 +\int_{C_1}(x-S_1(\frac 1 2))^2 dP+\int_{J_{12}}(x-S_{12}(\frac 12))^2\Big)+\int_C\min_{a\in \ga_2(\gn)}(x-a)^2 dP\\
&=2 \Big(\frac 1{75^2} V+ \frac 1{75} \frac 13 W+\frac 1{75^2} V\Big)+\frac{1}{3}\frac 1{1200}=\frac{2537}{6570000}=0.000386149,
\end{align*}
Since $V_n$ is the quantization error for $n$-means with $n\geq 8$, we have $V_n\leq V_8\leq 0.000386149$. For the sake of contradiction, assume that $\ga$ contains a point from the open interval $(\frac 15, \frac 25)$. Recall that there can not be more than one point from $\ga$ in any of the open intervals $(\frac 15, \frac 25)$ and $(\frac 35, \frac 45)$. Suppose that $a_j\in (\frac 15, \frac 25)$, where $j=\max\set{i : a_i< \frac 25 \te{ and } 1\leq i\leq n}$. Then, following cases can arise:

Case 1. $\frac 3{10}\leq a_j<\frac 25$.

Then, $a_{j-1}<\frac 1{10}$ yielding
\[V_n\geq  \int_{S_1[\frac 1{2}, \frac {3}{5}]\uu J_{12}}(x-\frac 1{10})^2 dP=\frac{97}{131400}=0.000738204>V_n,\]
which is a contradiction.

Case 2. $\frac 1{5}< a_j\leq \frac 3{10}$.

Then, $\frac 12(a_j+a_{j+1})>\frac 25$ implies $a_{j+1}>\frac 45-a_j\geq \frac 45-\frac 3{10}=\frac 12$. Thus,
\begin{align*}
V_n&\geq \int_{[\frac 25, \frac 12]}\min_{a\in \set{a_j, a_{j+1}}}(x-a)^2 dP\geq\int_{[\frac 25, \frac 12]}(x-\frac 12)^2 dP=\frac{1}{1800}=0.000555556>V_n,
\end{align*}
which leads to a contradiction.

Similarly, we can show that for any $a\in \ga$ if $a\in (\frac 35, \frac 45)$, then a contradiction arises.
Thus, the proof of the proposition is complete.
\end{proof}

\subsection{Optimal sets and the quantization error for a given sequence $F(n)$}
In this subsection we first define the two sequences $\set{a(n)}_{n\geq 0}$ and  $\set{F(n)}_{n\geq 0}$. These two sequences play important role in the rest of the paper.

\begin{defi} \label{defi31}  Define the sequence $\set{a(n)}_{n\geq 0}$ such that
$a(0)=0$, and $a(n)=2n-1$ for all $n\geq 1$.
Define the sequence $\set{F(n)}_{n\geq 0}$ such that $F(n)=4^n+2^{n+1}$, i.e.,
\[\set{F(n)}_{n\geq 0}=\{3,8,24,80,288,1088,4224,16640,66048,263168,1050624, \cdots\}.\]
\end{defi}

\begin{lemma1}
Let $\set{a(n)}_{n\geq 0}$ and $\set{F(n)}_{n\geq 0}$ be the sequences defined by Definition~\ref{defi31}. Then, $F(n+1)=2^{a(n+1)}+2 F(n)$.
\end{lemma1}
\begin{proof}
We have,
$2^{a(n+1)}+2F(n)=2^{2n+1}+2(4^n+2^{n+1})=4^{n+1}+2^{n+2}=F(n+1)$, and thus the lemma follows.
\end{proof}

For $n\in \D N$, we identify the sequence of sets
$ \ga_{2^{a(n)}}(\gn), \, \uu_{\go \in I} S_\go(\ga_{2^{a(n-1)}}(\gn)), \, \\ \uu_{\go \in I^2} S_\go(\ga_{2^{a(n-2)}}(\gn)), \, \cdots, \, \uu_{ \go \in I^{n-2}} S_\go(\ga_{2^{a(2)}}(\gn)), \, \uu_{\go \in I^{n-1}} S_\go(\ga_{2^{a(1)}}(\gn)), \, \uu_{\go \in I^{n}} S_\go(\ga_{2^{a(0)}}(\gn))$, and \\
$ \set{S_\go(\frac 12) : \go \in I^{n+1}}$, respectively, by $S(n)$, $S(n-1)$, $S(n-2)$,  $\cdots$, $S(2)$, $S(1)$, $S(0)$ and $S(n+1)$. For $0\leq \ell\leq n$, write
\[S^{(2)}(\ell):=\uu_{\go \in I^{n-\ell}} S_\go(\ga_{2^{a(\ell)+1}}(\gn)) \te{ and } S^{(2)(2)}(\ell):=\uu_{\go \in I^{n-\ell}} S_\go(\ga_{2^{a(\ell)+2}}(\gn)). \]
Further, we write
$S^{(2)}(n+1):=\set{S_\go(\ga_2(P)) : \go \in I^{n+1}}=\set{S_\go(\frac {13}{60}), S_\go(\frac{47}{60}) : \go \in I^{n+1}}$,
and $S^{(2)(2)}(n+1):=\uu_{\go \in I^{n+1}} S_\go(\ga_{2^{a(0)}}(\gn))\uu \set{S_\go(\frac 12) : \go \in I^{n+2}}.$ Moreover, for any $\ell\in \D N \uu\set{0}$, if $A:=S(i)$, we identify $S^{(2)}(i)$ and $S^{(2)(2)}(i)$, respectively, by $A^{(2)}$ and $A^{(2)(2)}$. For $n\in\D N \uu \set{0}$, set
\begin{equation} \label{eq67} \ga_{F(n)}:=S(n)\uu S(n-1)\uu S(n-2)\cdots S(1)\uu S(0)\uu S(n+1),\end{equation}
and \[SF(n):=\set{S(n), S(n-1), S(n-2), \cdots, S(1), S(0), S(n+1)}.\]
In addition, write
\begin{align*} \label{eq35} SF^\ast(n): =\set{S(n), S(n-1), \cdots,  S(0), S(n+1),  S^{(2)}(n), S^{(2)}(n-1), \cdots,  S^{(2)}(1), S^{(2)}(n+1)}.
\end{align*}
For any element $a\in A \in SF^\ast(n)$, by the Voronoi region of $a$ it is meant the Voronoi region of $a$ with respect to the set $\uu_{B\in SF^\ast(n)} B$.
Similarly, for any $a\in A\in SF(n)$, by the Voronoi region of $a$ it is meant the Voronoi region of $a$ with respect to the set $\uu_{B\in SF(n)} B$. Notice that if $a, b\in A$, where $A \in SF(n)$ or $A\in SF^\ast(n)$, the error contributed by $a$ in the Voronoi region of $a$ equals to the error contributed by $b$ in the Voronoi region of $b$.  Let us now define an order $\succ $ on the set $SF^\ast(n)$ as follows: For $A, B\in SF^\ast(n)$ by $A\succ B$ it is meant that the error contributed by any element $a\in A$ in the Voronoi region of $a$ is larger than the error contributed by any element $b \in B$ in the Voronoi region of $b$. Similarly, we define the order relation $\succ $ on the set $SF(n)$.

\begin{remark1}\label{remark1001} By Definition~\ref{defi31}, we have
\[\ga_{F(n)}=S_1(\ga_{F(n-1)})\uu \ga_{2^{a(n)}}(\gn)\uu S_2(\ga_{F(n-1)}).\]

\end{remark1}
\begin{lemma1} \label{lemma192}  Let $\succ $ be the order relation on $SF^\ast(n)$. Then,

$(i)$ $S(n)\succ S(n-1)\succ \cdots \succ S(2)\succ S(1)$ for all $n\geq 2$, and  $S^{(2)}(n)\succ S^{(2)}(n-1)\succ \cdots \succ S^{(2)}(2)\succ S^{(2)}(1)$ for all $n\geq 3$.

$(ii)$ $S(1)\succ S^{(2)}(n)$ for all $0\leq n\leq 14$.

$(iii)$ $S(n-13)\succ S^{(2)}(n)\succ S(n-14)$ for all $n\geq 14$.

$(iv)$ $S^{(2)}(2)\succ S(n+1)$ for all $n\geq 2$, and $S(1)\succ S(n+1)\succ S^{(2)}(1)\succ S(0)\succ S^{(2)}(n+1)\succ S^{(2)}(0)$ for all $n\geq 1$. On the other hand, $S(1)\succ S(0)\succ S^{(2)}(1)\succ S^{(2)}(0)$ for $n=0$.

\end{lemma1}
\begin{proof}
$(i)$ Let $1\leq p<q\leq n$. The distortion error due to any element in the set $S(q):=\uu_{\go\in I^{n-q}}(\ga_{2^{a(q)}}(\gn))$ is given by
$\frac 1{75^{n-q}} \frac 13 \frac{W}{2^{3(a(q))}}=\frac 1{75^{n-q}} \frac 13 \frac{W}{2^{6q-1}}$.
Similarly, the distortion error due to any element in the set $S(p)$ is given by $\frac 1{75^{n-p}} \frac 13 \frac{W}{2^{6p-1}}$. Thus, $S(q)\succ S(p)$ is true if
$\frac 1{75^{n-q}} \frac 13 \frac{W}{2^{6q-1}}>\frac 1{75^{n-p}} \frac 13 \frac{W}{2^{6p-1}}$, i.e., if $\Big(\frac{75}{64}\Big)^{q-p}>1$, which is clearly true since $q>p$. Hence, $S(n)\succ S(n-1)\succ \cdots \succ S(2)\succ S(1)$ for all $n\geq 2$, and similarly, we can prove that $S^{(2)}(n)\succ S^{(2)}(n-1)\succ \cdots \succ S^{(2)}(2)\succ S^{(2)}(1)$ for all $n\geq 3$.

$(ii)$ $S(1)\succ S^{(2)}(n)$ is true if $\frac 1{75^{n-1}} \frac 13 \frac{W}{2^{3a(1)}}>\frac 13 \frac{W}{2^{3(a(n)+1)}}$, i.e., if $\frac {75}{8} \Big(\frac {64}{75}\Big)^n > 1$, which is true for $0\leq n\leq 14$.

$(iii)$ For $0\leq n\leq 14$, we see that $S(n-13)\succ S^{(2)}(n)\succ S(n-14)$ is true if $\frac{1}{75^{13}}\frac 13 \frac W{2^{3 a(n-13)}}>\frac 13\frac {W}{2^{3(a(n)+1)}}>\frac{1}{75^{14}}\frac 13 \frac W{2^{3a(n-14)}}$, i.e., if $\frac{1}{75^{13}} \frac 1{2^{6n-81}}>\frac {1}{2^{6n}}>\frac{1}{75^{14}} \frac 1{2^{6n-87}}$, i.e., if $\frac {2^{81}}{75^{13}}>1>\frac {2^{87}}{75^{14}}$, which is obviously true.

$(iv)$ For $n\geq 2$, the relation $S^{(2)}(2)\succ S(n+1)$ is true if $\frac{1}{75^{n-2}}\frac 13 \frac W{2^{3 (a(2)+1)}}>\frac 1{75^{n+1}}V$, i.e., if $\frac {75^3}{3} \frac {W}{2^{12}}>V$ which is obvious. Similarly, we can prove the rest of the inequalities.
\end{proof}

We now give the following proposition.

\begin{propo} \label{prop711}  Let $\succ $ be the order relation on $SF^\ast(n)$. Then,

$(i)$ $S(n)\succ S(n-1)\succ S(n-2)\succ \cdots \succ S(2)\succ S(1)\succ S^{(2)}(n)\succ S^{(2)}(n-1)\succ S^{(2)}(n-2)\succ \cdots \succ S^{(2)}(2)\succ S(n+1)\succ S^{(2)}(1)\succ S(0)\succ S^{(2)}(n+1)\succ S^{(2)}(0)$ for all $2\leq n\leq 14$.

$(ii)$ $S(n)\succ S(n-1)\succ S(n-2)\succ \cdots \succ S(n-13)\succ S^{(2)}(n)\succ S(n-14)\succ S^{(2)}(n-1)\succ S(n-15)\succ S^{(2)}(n-2)\succ S(n-16)\succ \cdots\succ S^{(2)}(17)\succ S(3)\succ S^{(2)}(16)\succ S(2)\succ S^{(2)}(15)\succ S(1)\succ S^{(2)}(14)\succ S^{(2)}(13)
\succ \cdots\succ S^{(2)}(2)\succ S(n+1)\succ S^{(2)}(1)\succ S(0)\succ S^{(2)}(n+1)\succ S^{(2)}(0)$ for all $n\geq 15$.
\end{propo}
\begin{proof} The proof of the proposition follows by combining the inequalities in Lemma~\ref{lemma192}.
\end{proof}

\begin{lemma1} \label{lemma431}
Let $\ga_{F(n)}$ and $SF(n)$ be the sets as defined before. Then,
\[\ga_{F(n+1)}=\Big(\mathop{\uu}\limits_{A\in (SF(n)\setminus S(0))}A^{(2)(2)}\Big)\uu S^{(2)}(0).\]
\end{lemma1}

\begin{proof} For $S(n), S(n-1), \cdots, S(2), S(0), S(n+1) \in SF(n)$, we have
\begin{align*}
S^{(2)(2)}(n)&= \ga_{2^{a(n)+2}}(\gn)=\ga_{2^{a(n+1)}}(\gn), \\
S^{(2)(2)}(n-1)&= \uu_{\go\in I}S_\go(\ga_{2^{a(n-1)+2}}(\gn))=\uu_{\go\in I}S_\go(\ga_{2^{a(n)}}(\gn)), \\
&\vdots \\
S^{(2)(2)}(1)&= \uu_{\go\in I^{n-1}}S_\go(\ga_{2^{a(1)+2}}(\gn))=\uu_{\go\in I^{n-1}}S_\go(\ga_{2^{a(2)}}(\gn)), \\
S^{(2)}(0)&=\uu_{\go\in I^{n}}S_\go(\ga_{2^{a(0)+1}}(\gn))=\uu_{\go\in I^{n}}S_\go(\ga_{2^{a(1)}}(\gn)),\\
S^{(2)(2)}(0)&=\uu_{\go\in I^{n+1}}S_\go(\ga_{2^{a(0)}}(\gn))\uu \set{S_\go(\frac 12) : \go \in I^{n+2}}\\
& =\uu_{\go\in I^{n+1}}S_\go(\ga_{2^{a(0)}}(\gn))\uu \set{S_\go(\frac 12) : \go \in I^{n+2}}.
\end{align*}
Thus, by the expression~\eqref{eq67}, the proof of the lemma follows.
\end{proof}

 We now prove the following lemma.

\begin{lemma1}\label{lemma711}  For any two sets $A, B\in SF (n)$, let $A\succ B$. Then, the distortion error due to the set $ (SF(n)\setminus A)\uu A^{(2)}\uu B$ is less than the distortion error due to the set $ (SF(n)\setminus B)\uu B^{(2)}\uu A$.
\end{lemma1}

\begin{proof} By Lemma~\ref{lemma192} and Proposition~\ref{prop711}, for all $n\geq 0$, we have
\[S(n)\succ S(n-1)\succ S(n-2)\succ \cdots\succ S(1)\succ S(n+1)\succ S(0).\]
Let $V(\ga_{F(n)})$ be the distortion error due to the set $\ga_{F(n)}$ with respect to the condensation measure $P$. First take $A=S(a(k))$ and $B=S(a(k'))$ for some $1\leq k'< k\leq n$. Then, the distortion error due to the set $(\ga_{F(n)}\setminus A)\uu A^{(2)}\uu B$ is given by
\begin{align}\label{eq45}
&V(\ga_{F(n)})-(\frac 2{75})^{n-k} \frac 13 \frac W{2^{2a(k)}}+(\frac 2{75})^{n-k} \frac 13 \frac W{2^{2(a(k)+1)}}+(\frac 2{75})^{n-k'} \frac 13 \frac W{2^{2a(k')}}\notag\\
&=V(\ga_{F(n)})-(\frac 2{75})^{n-k} \frac 13 \frac{3W}{2^{4k}}+(\frac 2{75})^{n-k'} \frac 13 \frac W{2^{4k'}}
\end{align}
Similarly, the distortion error due to the set $(\ga_{F(n)}\setminus B)\uu B^{(2)}\uu A$ is
\begin{align} \label{eq46}
V(\ga_{F(n)})-(\frac 2{75})^{n-k'} \frac 13 \frac{3W}{2^{4k'}}+(\frac 2{75})^{n-k} \frac 13 \frac W{2^{4k}}.
\end{align}
Thus, \eqref{eq45} will be less than \eqref{eq46} if $(\frac 2{75})^{n-k} \frac 13 \frac{7W}{2^{4k}}>(\frac 2{75})^{n-k'} \frac 13 \frac{7W}{2^{4k'}}$, i.e., if $(\frac {75}{32})^{k-k'}>1$, which is clearly true since $k'<k$. Now, take $A=S(k)$ for $1\leq k\leq n$, and $B=S(n+1)$. Then, the distortion error due to the set  $(\ga_{F(n)}\setminus A)\uu A^{(2)}\uu B$ is less than the distortion error due to the set  $(\ga_{F(n)}\setminus B)\uu B^{(2)}\uu A$ if
\[V_{F(n)}-(\frac 2{75})^{n-k}\frac 1{3} \frac{3W}{2^{4k}}+(\frac 2{75})^{n+1}V<V_{F(n)}-(\frac 2{75})^{n+1}V+(\frac 2{75})^{n+1}\frac 12 V_2+(\frac 2{75})^{n-k}\frac 1{3} \frac{4W}{2^{4k}},\]
i.e., if $(\frac 2{75})^{n-k}\frac 1{3} \frac{7W}{2^{4k}}>(\frac 2{75})^{n+1}(2V-\frac 12 V_2)$, i.e., if $(\frac {75}{32})^k \frac{7W}{3}>\frac 2{75}(2V-\frac 12V_2)$, which is obviously true for $k\geq 1$. Similarly, if $A=S(n+1)$ and $B=S(0)$, we can show that the distortion error due to the set  $(\ga_{F(n)}\setminus A)\uu A^{(2)}\uu B$ is less than the distortion error due to the set  $(\ga_{F(n)}\setminus B)\uu B^{(2)}\uu A$. Thus, the proof of the lemma is complete.
\end{proof}

Using the similar technique as Lemma~\ref{lemma711}, the following lemma can be proved.

\begin{lemma1}\label{lemma721}  For any two sets $A, B\in SF^\ast (n)$, let $A\succ B$. Then, the distortion error due to the set $ (SF^\ast(n)\setminus A)\uu A^{(2)}\uu B$ is less than the distortion error due to the set $ (SF^\ast(n)\setminus B)\uu B^{(2)}\uu A$.
\end{lemma1}

The following lemma is useful.

\begin{lemma1} \label{lem11}
Let $\ga$ be an optimal set of $n$-means with $\ga\ii (\frac 15, \frac 25)=\es=\ga\ii (\frac 35, \frac 45)$, and let $i=1, 2$. Set $\gb_i:=\ga \ii J_i$,  $n_i:=\te{card}(\gb_i)$, $\gb_c:=\ga\ii C$, and $n_c:=\te{card}(\gb_c)$. Then, $S_i^{-1}(\ga_i)$ is an optimal set of $n_i$-means, and $\gb_c$ is an optimal set of $n_c$-means for $\gn$. Moreover,
$V_n(P)=\frac 1{75}(V_{n_1}(P)+V_{n_2}(P))+\frac 13 V_{n_c}(\gn).$
\end{lemma1}

\begin{proof} Proceeding in the similar way as \cite[Lemma~3.6]{CR}, we can show that $S_i^{-1}(\gb_i)$ is an optimal set of $n_i$-mean for $i=1, 2$. Proof of $\gb_c$ is an optimal set of $n_c$-means for $\gn$ is also similar with a minor modification. $\ga\ii (\frac 15, \frac 25)=\es=\ga\ii (\frac 35, \frac 45)$ implies $n=n_1+n_2+n_c$ and $\ga= \gb_1\uu \gb_2\uu\gb_c$. For $i=1, 2$, we have
\[\int_{J_i}\min_{a\in \gb_i}(x-a)^2 dP=\frac 13\int_{J_i}\min_{a\in \gb_i}(x-a)^2 d(P\circ S_i^{-1})=\frac1 {75} \int\min_{a\in S_i^{-1}(\gb_i)}(x-a)^2 dP=\frac 1{75} V_{n_i}(P).\]
Moreover, \[V_{n_c}(\gn)=\int_{\gb_c}\min_{a\in \gb_c}(x-a)^2 d\gn.\] Thus, we have
\[V_n(P)=\sum_{i=i}^2\int_{J_i}\min_{a\in \gb_i}(x-a)^2 dP+\int_{\gb_c}\min_{a\in \gb_c}(x-a)^2 dP=\frac 1{75}(V_{n_1}(P)+V_{n_2}(P))+\frac 13 V_{n_c}(\gn).\] This completes the proof of the lemma.
\end{proof}

\begin{remark1}
By Remark~\ref{remark1001} and Lemma~\ref{lem11}, we see that if $\ga_{F(n)}$ is an optimal set of $F(n)$-means for any $n\geq 1$, then $\ga_{F(n)}$ contains $F(n-1)$ elements from each of $J_1$ and $J_2$, and $2^{a(n)}$ elements from $C$ yielding the fact that
$V_{F(n)}(P)=  \frac 2 {75}  V_{F(n-1)}(P)+\frac 1 3 V_{2^{a(n)}}(\gn)$.
\end{remark1}

\begin{propo} \label{prop721}
 For any $n\geq 0$ the set $\ga_{F(n)}$ is an optimal set of $F(n)$-means for the condensation measure $P$ with quantization error given by
 \[V_{F(n)}:=V_{F(n)}(P)=\left\{\begin{array}{ll}
 \frac{89}{21900} & \te{ if } n=0, \\[0.02 in]
 \frac{2537}{6570000} & \te{ if } n=1,\\
\frac{1}{129}\frac 1{16^n} -\frac{3473}{941700} (\frac 2{75})^n & \te{ if } n\geq 2.
 \end{array}\right.\]
\end{propo}

\begin{proof}
By Proposition~\ref{prop51}, we know that $\ga_{F(0)}$ is an optimal set of $F(0)$-means with quantization error  $\frac{89}{21900}$. Using Lemma~\ref{lemma012} and the expression~\eqref{eq2}, the distortion error due to the set $\ga_{F(1)}$ is obtained as
\begin{align*}\label{eq67}
&\int\min_{a\in \ga_{F(1)}}(x-a)^2 dP=\frac 13 V_{2^{2a(1)}}(\gn)+\sum_{\go\in I} \int_{C_\go}\min_{a\in S_\go(\ga_{2^{a(0)}}(\gn))}(x-a)^2 dP+\sum_{\go\in I^2}\int_{J_\go}(x-a)^2 dP \notag \\
&=\frac 1 3 \frac W{2^2}+\frac 13 \frac 2{75} \frac{W}{2^{2a(0)}} +(\frac 2{75})^2 V=\frac{2537}{6570000}=0.000386149.
\end{align*}
First, we show that $\ga_{F(1)}$ is an optimal set of $F(1)$-means. Let $\gb$ be an optimal set of $F(1)$-means and $V_{F(1)}$ is the corresponding quantization error. Then, from the above calculation, we have  $V_{F(1)}\leq 0.000386149$. Recall that $\gb$ contains points from $J_1, J_2$ and $C$, and $\gb$ does not contain any point from the open intervals $(\frac 15, \frac 25)$ and $(\frac 25, \frac 35)$. Due to symmetry of $P$ we can assume that $\gb$ contains equal number of elements from each of $J_1$ and $J_2$, and the rest of the elements of $\gb$ are equally spaced over the set $C$ with respect to the uniform distribution $\gn$. Thus, $\te{card}(\gb\ii J_1)=\te{card}(\gb\ii J_2)\leq 3$. Suppose that $\te{card}(\gb\ii J_1)=\te{card}(\gb\ii J_2)=1$, then
\[V_{F(1)}\geq 2 \int_{J_1}\min_{a\in \gb\ii J_1}(x-a)^2 dP= \frac 2 {75} V=0.00295282>V_{F(1)},\]
which is a contradiction. Next, suppose that $\te{card}(\gb\ii J_1)=\te{card}(\gb\ii J_2)=2$, then,
\[V_{F(1)}\geq 2 \int_{J_1}\min_{a\in \gb\ii J_1}(x-a)^2 dP= \frac 2 {75} V_2=0.000812075>V_{F(1)},\]
which leads to another contradiction. So, we can assume that  $\te{card}(\gb\ii J_1)=\te{card}(\gb\ii J_2)=3$. Then, by Lemma~\ref{lem11}, we have $S_1^{-1}(\gb\ii J_1)=S_2^{-1}(\gb \ii J_2)=\ga_{F(0)}$. Hence, by Remark~\ref{remark1001},
$\gb=S_1(\ga_{F(0)})\uu S_2(\ga_{F(0)})\uu \ga_2(\gn)=\ga_{F(1)}$, i.e., $\ga_{F(1)}$ is an optimal set of $F(1)$-means, and the corresponding quantization error is given by $V_{F(1)}=\frac{2537}{6570000}$. The distortion error due to the set $\ga_{F(n)}$ for any $n\geq 2$ is given by
\begin{align*}
&\int\min_{a\in \ga_{F(n)}}(x-a)^2 dP\\
&=\frac 13 V_{2^{2a(n)}}(\gn)+\sum_{k=0}^{n-1}\sum_{\go\in I^{n-k}} \int_{C_\go}\min_{a\in S_\go(\ga_{2^{a(k)}}(\gn))}(x-a)^2 dP+\sum_{\go\in I^{n+1}}\int_{J_\go}(x-a)^2 dP\\
&=\frac 13 \frac W{2^{2a(n)}}+\sum_{k=1}^{n-1}\sum_{\go\in I^{n-k}} \frac 13 \frac 1{75^{n-k}} \frac W{2^{2 a(k)}} + \sum_{\go \in I^n}\frac 13 \frac 1{75^{n}} \frac W{2^{2 a(0)}} +(\frac 2{75})^{n+1} V\\
&=\frac 13 \frac {4W}{16^n}+\sum_{k=1}^{n-1}\frac 13 (\frac 2{75})^{n-k} \frac {4W}{16^k} + \frac 13 (\frac 2{75})^{n}W +(\frac 2{75})^{n+1} V\\
&=\frac 13 \frac {4W}{16^n}+(\frac 2{75})^n \frac {4W} 3 \sum_{k=1}^{n-1} (\frac{75}{32})^k + \frac 13 (\frac 2{75})^{n}W +(\frac 2{75})^{n+1} V\\
&=\frac 13 \frac {4W}{16^n}+(\frac 2{75})^n\Big(\frac W 3-\frac {4W}{3} \frac {75}{43}\Big(1-(\frac {75}{32})^{n-1}\Big)+\frac 2{75}V\Big),
\overline{}\end{align*}
yielding
\begin{equation} \label{eq68} \int\min_{a\in \ga_{F(n)}}(x-a)^2 dP=\frac{1}{129}\frac 1{16^n} -\frac{3473}{941700} (\frac 2{75})^n.\end{equation}
Now, we show that $\ga_{F(2)}$ is an optimal set of $F(2)$-means. Let $\gg$ be an optimal set of $F(2)$-means, and $V_{F(2)}$ is the corresponding quantization error. By \eqref{eq68}, we have $V_{F(2)}\leq 0.0000276584$. We show that $\te{card}(\gg\ii J_1)=\te{card}(\gg\ii J_2)=8$. Due to symmetry of $P$ we can assume that $\gg$ contains equal number of elements from each of $J_1$ and $J_2$, and the rest of the elements of $\gb$ are equally spaced over the set $C$ with respect to the uniform distribution $\gn$. Suppose that $\te{card}(\gg\ii J_1)=\te{card}(\gg\ii J_2)\leq 6$. Then, if $\ga_6$ is an optimal set of six-means, by Remark~\ref{remark99}, we have
\[V_{F(2)}\geq 2 \int_{J_1}\min_{a\in \gg\ii J_1}(x-a)^2 dP\geq  2\int_{J_1} \min_{a \in S_1(\ga_6)}(x-a)^2 dP= \frac 2 {75} \frac{21481}{19710000}\]
i.e., $V_{F(2)}\geq 0.0000290627>V_{F(2)}$, which is a contradiction. Suppose that $\te{card}(\gg\ii J_1)=\te{card}(\gg\ii J_2)=7$. Then, by Remark~\ref{remark101}, we have
\[V_{F(2)}\geq 2 \int_{J_1}\min_{a\in \gg\ii J_1}(x-a)^2 dP+\frac 13 V_{10}(\gn) =\frac 2{75} \frac{7273}{9855000}+\frac 13 \frac W{10^2}=0.0000307911>V_{F(2)},\]
which leads a contradiction. Similarly, we can show that if $\te{card}(\gg\ii J_1)=\te{card}(\gg\ii J_2)\geq 9$, a contradiction arises, i.e., $\te{card}(\gg\ii J_1)=\te{card}(\gg\ii J_2)=8$, and so by Remark~\ref{remark1001},
$\gg=S_1(\ga_{F(1)})\uu S_2(\ga_{F(1)})\uu \ga_{2^{a(2)}}(\gn)=\ga_{F(2)}$, i.e., $\ga_{F(2)}$ is an optimal set of $F(2)$-means, and the corresponding quantization error is given by $V_{F(2)}=0.0000276584$.

Let $\ga_{F(n)}$ be an optimal set of $F(n)$-means for some $n\geq 2$. We show that $\ga_{F(n+1)}$ is an optimal set of $F(n+1)$-means. We have
$\ga_{F(n)}=\uu_{A\in SF(n)}A$. Recall that, by Proposition~\ref{prop21}, an optimal set of $n$-means for any $n\geq 3$ does not contain any point from the open intervals $(\frac 15, \frac 25)$ and $(\frac 35, \frac 45)$. In the first step, let $A(1) \in SF(n)$ be such that $A(1)\succ B$ for any other $B\in SF(n)$. Then, by Lemma~\ref{lemma711},
the set  $(\ga_{F(n)}\setminus A(1)) \uu A^{(2)}(1)$ gives an optimal set of $F(n)-\te{card}(A(1))+\te{card}(A^{(2)}(1))$-means. In the 2nd step, let $A(2) \in (SF(n)\setminus \set{A(1)})\uu \set{A^{(2)}(1)}$ be such that $A(2)\succ B$ for any other set $B\in  (SF(n)\setminus \set{A(1)})\uu \set{A^{(2)}(1)}$.
Then, using the similar technique as Lemma~\ref{lemma711}, we can show that the distortion error due to the following set:
\begin{equation} \label{eq44} \Big(((\ga_{F(n)}\setminus A(1)) \uu A^{(2)}(1))\setminus A(2)\Big)\uu A^{(2)}(2)
\end{equation}
 with cardinality $F(n)-\te{card}(A(1))+\te{card}(A^{(2)}(1))-\te{card}(A(2))+\te{card}(A^{(2)}(2))$ is smaller than the distortion error due to the set obtained by replacing $A(2)$ in the set \eqref{eq44} by the set $B$. In other words, $\Big(((\ga_{F(n)}\setminus A(1)) \uu A^{(2)}(1))\setminus A(2)\Big)\uu A^{(2)}(2)$ forms an optimal set of
$F(n)-\te{card}(A(1))+\te{card}(A^{(2)}(1))-\te{card}(A(2))+\te{card}(A^{(2)}(2))$-means. Proceeding inductively in this way, up to $(2n+3)$ steps, we can see that $\ga_{F(n+1)}=\Big(\mathop{\uu}\limits_{A\in (SF(n)\setminus S(0))}A^{(2)(2)}\Big)\uu S^{(2)}(0)$ forms an optimal set of $F(n+1)$-means.
Thus, by the induction principle, we can say that for any $n\geq 0$, the set $\ga_{F(n)}$ forms an optimal set of $F(n)$-means with quantization error $V_{F(n)}$ as given in the hypothesis. Thus, the proof of the proposition is complete.
\end{proof}

\subsection{Asymptotics for the $n$th quantization error $V_n(P)$}
In this subsection, we show that the quantization dimension of the condensation measure $P$ exists and equals to the quantization dimension of the uniform distribution $\gn$. In addition, we show that the $D(P)$-dimensional quantization coefficient for the condensation measure $P$ does not exist, and the lower and upper  quantization coefficients for $P$ are finite and positive.

\begin{theo} \label{Th11}
Let $P$ be the condensation measure associated with the uniform distribution $\gn$. Then, $\lim_{n\to \infty} \frac{2 \log n}{\log V_n(P)}=1$, i.e., the quantization dimension $D(P)$ of the measure $P$ exists and equals to the quantization dimension $D(\gn)$ of $\gn$.
\end{theo}
\begin{proof}
For $n\in \D N$, $n\geq 4$,  let $\ell(n)$ be the least positive integer such that $F(\ell(n))\leq n<F(\ell(n)+1)$. Then,
$V_{F(\ell(n)+1)}<V_n\leq V_{F(\ell(n))}$. Thus, we have
\begin{align*}
\frac {2\log\left(F(\ell(n))\right)}{-\log\left(V_{F(\ell(n)+1)}\right)}< \frac {2\log n}{-\log V_n}< \frac {2\log\left(F(\ell(n)+1)\right)}{-\log\left(V_{F(\ell(n))}\right)}.
\end{align*}
Notice that when $n\to \infty$, then $\ell(n)\to \infty$. By Proposition~\ref{prop721}, we have
\begin{align*}
&\lim_{\ell(n)\to\infty} \frac {2\log\left(F(\ell(n))\right)}{-\log\left(V_{F(\ell(n)+1)}\right)}=2 \lim_{\ell(n)\to\infty} \frac {\log(4^{\ell(n)}+2^{\ell(n)+1})}{-\log\Big(\frac{1}{129}\frac 1{16^{\ell(n)+1}} -\frac{3473}{941700} (\frac 2{75})^{\ell(n)+1}\Big)}(\frac{\infty} {\infty} \te{ form})\\
&=2 \lim_{\ell(n)\to\infty} \Big(\frac{\frac{1}{129}\frac 1{16^{\ell(n)+1}} -\frac{3473}{941700} (\frac 2{75})^{\ell(n)+1}}{4^{\ell(n)}+2^{\ell(n)+1}}\Big)\Big(\frac {4^{\ell(n)} \log 4+2^{\ell(n)+1}\log 2}{\frac 1{129}\frac 1{16^{\ell(n)+1}} \log 16 +\frac{3473}{941700} (\frac 2{75})^{\ell(n)+1}\log \frac 2{75}}\Big)\\
&=\frac {2\log 4}{\log 16}=1.
\end{align*}
Similarly, $\mathop{\lim}\limits_{\ell(n)\to\infty} \frac {2\log\left(F(\ell(n)+1)\right)}{-\log\left(V_{F(\ell(n))}\right)}=1$. Thus, $1\leq \liminf_n \frac{2\log n}{-\log V_n}\leq \limsup_n \frac{2\log n}{-\log V_n}\leq1$ implying the fact that the quantization dimension of the measure $P$ exists and equals to the quantization dimension $D(\gn)$ of $\gn$, which is the theorem.
\end{proof}

\begin{theo} \label{Th21}
The $D(P)$-dimensional quantization coefficient for the condensation measure $P$ does not exist.
\end{theo}

\begin{proof} For any $n\in\D N$, by Proposition~\ref{prop711} and Lemma~\ref{lemma711}, we see that $(\ga_{F(n)}\setminus S(n))\uu S^{(2)}(n)$ is an optimal set of $F(n)+2^{a(n)}$-means with quantization error
\begin{align*}
&V_{F(n)+2^{a(n)}}=V_{F(n)}-\frac 13 \frac W{2^{2a(n)}}+\frac 13 \frac W{2^{2(a(n)+1)}}=V_{F(n)}-\frac W{16^n}=(\frac{1}{129}-W)\frac 1{16^n} -\frac{3473}{941700} (\frac 2{75})^n\\
&=\frac{19}{4300}\frac 1{16^n} -\frac{3473}{941700} (\frac 2{75})^n.
\end{align*}
Since $F(n)+2^{a(n)}=4^n+2^{n+1}+2^{2n-1}=\frac 32 4^n+2^{n+1}$,  we have
\begin{equation} \label{eq55} \lim_{n\to \infty} (F(n)+2^{a(n)})^2 V_{F(n)+2^{a(n)}}=\lim_{n\to \infty} (\frac 32 4^n+2^{n+1})^2 (\frac{19}{4300}\frac 1{16^n} -\frac{3473}{941700} (\frac 2{75})^n)=\frac{171}{17200}.\end{equation}
Again, $(F(n))^2=(4^{n}+2^{n+1})^2=16^{n}+2^{3n+2}+4^{n+1}=16^{n} (1+\frac 4{2^{n}}+\frac{4}{4^{n}})$,
and \[V_{F(n)}(P)=\frac{1}{129}\frac 1{16^{n}} -\frac{3473}{941700} (\frac 2{75})^{n} =\frac 1{16^{n}}\Big(\frac{1}{129}-\frac{3473}{941700}(\frac {32}{75})^{n}\Big).\]
Thus,
\begin{equation} \label{eq56}
\lim_{n\to \infty} (F(n))^2V_{F(n)}=\frac {1}{129}.
\end{equation}
By the equations \eqref{eq55} and \eqref{eq56}, we see that $\liminf_n n^2 V_n(P)\leq \frac {1}{129}<\frac{171}{17200}\leq \limsup_n n^2 V_n(P)$, in other words, $\lim_{n\to\infty} n^2 V_n(P)$ does not exist, i.e., the $D(P)$-dimensional quantization coefficient for the condensation measure $P$ does not exist.
\end{proof}

\begin{theo} \label{Th31}
The $D(P)$-dimensional lower and upper quantization coefficients for the condensation measure $P$ are finite and positive.
\end{theo}

\begin{proof}
For $n\in \D N$, $n\geq 4$,  let $\ell(n)$ be the least positive integer such that $F(\ell(n))\leq n<F(\ell(n)+1)$. Then,
$V_{F(\ell(n)+1)}<V_n\leq V_{F(\ell(n))}$ implying $(F(\ell(n)))^2V_{F(\ell(n)+1)}<n^2V_n< (F(\ell(n)+1))^2V_{F(\ell(n))}$. As $\ell(n)\to \infty$ whenever $n\to \infty$, we have
\[\lim_{n\to \infty} \frac{(F(\ell(n)))^2}{(F(\ell(n)+1))^2}=\lim_{n\to \infty} \Big(\frac{4^{\ell(n)}+2^{\ell(n)+1}}{4^{\ell(n)+1}+2^{\ell(n)+2}}\Big)^2=\frac 1{16}.\]
Again,  $(F(\ell(n)))^2=(4^{\ell(n)}+2^{\ell(n)+1})^2=16^{\ell(n)}+2^{3\ell(n)+2}+4^{\ell(n)+1}=16^{\ell(n)} (1+\frac 4{2^{\ell(n)}}+\frac{4}{4^{\ell(n)}})$,
and \[V_{F(\ell(n))}(P)=\frac{1}{129}\frac 1{16^{\ell(n)}} -\frac{3473}{941700} (\frac 2{75})^{\ell(n)} =\frac 1{16^{\ell(n)}}\Big(\frac{1}{129}-\frac{3473}{941700}(\frac {32}{75})^{\ell(n)}\Big).\]
Thus,
\[
 \lim_{n\to \infty} (F(\ell(n)))^2V_{F(\ell(n)+1)}=\frac 1{16} \lim_{n\to \infty}(F(\ell(n)+1))^2V_{F(\ell(n)+1)}=\frac 1 {16}\frac 1{129}, \te{ and similarly}\]
\[\lim_{n\to \infty} (F(\ell(n)+1))^2V_{F(\ell(n))}=\frac {16}{129},\]
yielding the fact that $\frac 1 {16}\frac 1{129}\leq \mathop{\liminf}\limits_{n\to\infty} n^2 V_n(P)\leq \mathop{\limsup}\limits_{n\to \infty} n^2 V_n(P)\leq \frac {16}{129}$, i.e., the $D(P)$-dimensional lower and quantization coefficients for the condensation measure $P$ are finite  and positive, which is the theorem.
\end{proof}
\begin{remark1} Notice that $(\frac 13 (\frac 15)^2)^{\frac {\gk}{2+\gk}}+(\frac 13(\frac 15)^2)^{\frac {\gk}{2+\gk}}=1$ implies $\gk=\frac{2\log 2}{\log75-\log 2}\approx 0.382496<1=D(\gn) =\max\set {\gk, D(\gn)}$. In Theorem~\ref{Th11}, we have proved that $D(P)=D(\gn).$
\end{remark1}
We conclude the paper with the following remark.
\begin{remark}
	Using the similar techniques or by giving a major overhaul to our techniques given in the last two sections Section \ref{sec1} and Section \ref{sec2}, one can investigate optimal quantization for more general condensation measures.
\end{remark}
	\section*{Statements and Declarations}
\textbf{Data availability:} Data sharing is not applicable to this article as no data sets were generated or analyzed during the current study.\\
\textbf{Funding:} The first author thanks IIIT Allahabad (Ministry of Education, India) for financial support through a Senior Research Fellowship. \\
\textbf{Conflict of interest:} 
We declare that we do not have any conflict of interest.\\
\textbf{Author Contributions:} 
All authors contributed equally to this manuscript.


\begin{thebibliography}{9999}

\bibitem{AW} Abaya, E.F.; Wise, G.L. {Some remarks on the existence of optimal quantizers}. \emph{Stat. Probab. Lett.} \textbf{1984}, \emph{2}, 349--351.
\bibitem{B} Barnsley, M.F. \emph{Fractals Everywhere}; Academic Press: New York, NY, USA;  London, UK, 1988.

\bibitem{BW} Bucklew, J.A.; Wise, G.L. {Multidimensional asymptotic quantization theory with $r$th power distortion measures}. \emph{IEEE Trans. Inf. Theory} \textbf{1982}, \emph{28}, 239--247.
\bibitem{CR} C\"omez, D.\.; Roychowdhury, M.K. {Quantization for uniform distributions on stretched Sierpinski triangles}, \emph{Monatshefte f\"ur Mathematik} \textbf{2019}, \emph{190}, 79--100. 

\bibitem{DR} Dettmann, C.P.; Roychowdhury, M.K. {Quantization for uniform distributions on equilateral triangles}.  \emph{Real Anal. Exch.} \textbf{2017}, \emph{42}, 149--166.
\bibitem{DL} Dai, M.; Liu, Z. {The quantization dimension and other dimensions of probability measures}. \emph{Int. J. Nonlinear Sci.} \textbf{2008}, \emph{5}, 267--274.
\bibitem{DRV} Dubey, S.; Roychowdhury, M.K.; Verma, S. {Quantization dimension for a generalized inhomogeneous bi-Lipschitz iterated function system.}  \emph{arXiv}, \textbf{2025}, arXiv:2503.11105.
\bibitem{F} Facloner, K.J. \emph{Techniques in Fractal Geometry}; John Wiley \& Sons Ltd.: Chichester, UK, 1997.

\bibitem{F1} Facloner, K.J. {The multifractal spectrum of statistically self-similar measures}. \emph{J. Theoret. Probab.} \textbf{1994}, \emph{7}, 681--701.

\bibitem{GG} Gersho, A.; Gray, R.M. \emph{Vector Quantization and Signal Compression}; Kluwer Academy publishers: Boston, MA, USA, 1992.

\bibitem{GKL}  Gray, R.M.; Kieffer, J.C.; Linde, Y. {Locally optimal block quantizer design}. \emph{Inf. Control} \textbf{1980}, \emph{45},  178--198.


\bibitem{GL1} Gy\"orgy, A.; Linder, T. {On the structure of optimal entropy-constrained scalar quantizers}.  \emph{IEEE Trans. Inf. Theory} \textbf{2002}, \emph{48}, 416--427.

\bibitem{GL2} Graf, S.; Luschgy, H. \emph{Foundations of Quantization for Probability Distributions}; Lecture Notes in Mathematics 1730; Springer:  Berlin/Heidelberg, Germany, 
2000.

\bibitem{GL3} Graf, S.; Luschgy, H. {The Quantization Dimension of Self-Similar Probabilities}. \emph{Math. Nachr.} \textbf{2002}, \emph{241}, 103--109.

\bibitem{L1} Lasota, A. {A variational principle for fractal dimensions}. \emph{Nonlinear Anal.} \textbf{2006}, \emph{64}, 618--628.

\bibitem{L2} Liszka, P. {The $L^q$ spectra and R\'enyi dimension of generalized inhomogeneous self-similar measures}. \emph{Cent. Eur. J. Math.} \textbf{2014}, \emph{12}, 1305--1319.

\bibitem{LM} Lindsay, L.J.; Mauldin, R.D. {Quantization dimension for conformal iterated function systems}. \emph{Nonlinearity} \textbf{2002}, \emph{15}, 189--199.

\bibitem{GN}  Gray, R.; Neuhoff, D. {Quantization}. \emph{IEEE Trans. Inform. Theory}  \textbf{1998}, \emph{44},  2325--2383.


\bibitem {P} P\"otzelberger, K. {The quantization dimension of distributions}. \emph{Math. Proc. Camb. Phil. Soc.} \textbf{2001}, \emph{131}, 507--519.

\bibitem {OS1} Olsen, L.; Snigireva, N. {Multifractal spectra of in-homogenous self-similar measures}. \emph{Indiana U. Math. J.} \textbf{2008}, \emph{57}, 1789--1844.

\bibitem {OS2} Olsen, L.; Snigireva, N. {$L^q$ spectra and R\'enyi dimensions of in-homogeneous self-similar measures}. \emph{Nonlinearity}  \textbf{2007}, \emph{20}, 151--175.

\bibitem{PRV} Priyadarshi, A.; Roychowdhury, M.K.; Verma, M. {Quantization dimensions for inhomogeneous bi-Lipschitz iterated function systems}. \emph{Monatshefte f\"ur Mathematik} \textbf{2025}, \emph{207}, 125--140. 
\bibitem{RR} Rosenblatt, J.; Roychowdhury, M.K. \emph{Optimal quantization for piecewise uniform distributions}. \emph{Unif. Distrib. Theory} \textbf{2018}, \emph{13}, 23--55.
\bibitem{R1} Roychowdhury, M.K. {Optimal quantizers for some absolutely continuous probability measures}. \emph{Real Anal. Exch.} \textbf{2017}, \emph{43}, 105--136.

\bibitem{R3} Roychowdhury, M.K. {Least upper bound of the exact formula for optimal quantization of some uniform Cantor distributions}. \emph{Discrete and Continuous Dynamical Systems- Series A} \textbf{2018}, \emph{38}, 4555--4570.
\bibitem{R5} Roychowdhury, M.K. {Quantization dimension estimate for condensation systems of conformal mappings}. \emph{Real Anal. Exch.} \textbf{2013}, \emph{38}, 317--335.

\bibitem{R6} Roychowdhury, M.K. {Quantization dimension function and ergodic measure with bounded distortion}. \emph{Bull. Pol. Acad. Sci. Math.} \textbf{2009}, \emph{57}, 251--262.

\bibitem{Z1} Zam, R. \emph{Lattice Coding for Signals and Networks: A Structured Coding Approach to Quantization, Modulation, and Multiuser Information Theory}; Cambridge University Press: Cambridge, MA, USA, 2014.

\bibitem{Z2} Zhu, S. {Quantization dimension for condensation systems}. \emph{Mathematische Zeitschrift} \textbf{2008}, \emph{259}, 33--43.

\bibitem{Z3} Zhu, S. {Asymptotic quantization errors for in-homogeneous self-similar measures supported on self-similar sets}. \emph{Sci. China Math.} \textbf{2016}, \emph{59}, 337--350.






\end{thebibliography}
\end{document}